\documentclass{amsart}

\usepackage[a4paper, margin=2.5cm]{geometry}
\usepackage[foot]{amsaddr}

\usepackage[T1]{fontenc}

\usepackage{graphicx}%
\usepackage{subcaption}
\usepackage{multirow}%
\usepackage{amsmath,amssymb,amsfonts}%
\usepackage{amsthm}%
\usepackage{mathrsfs}%
\usepackage{mathtools}%
\usepackage[title]{appendix}%
\usepackage{xcolor}%
\usepackage{textcomp}%
\usepackage{manyfoot}%
\usepackage{booktabs}%
\usepackage{listings}%
\usepackage{nicematrix}
\usepackage{microtype}

\usepackage{tikz}
\usetikzlibrary{tikzmark}
\usetikzlibrary{patterns, calc}
\usetikzlibrary{positioning,matrix,fit}
\usetikzlibrary{shapes,arrows}
\usetikzlibrary{decorations.pathreplacing,angles,quotes}
\usetikzlibrary{shapes.geometric, arrows}
 \usetikzlibrary{decorations.pathmorphing}
\usepackage{tikz-cd}

\theoremstyle{thmstyleone}%
\newtheorem{theorem}{Theorem}

\theoremstyle{thmstyletwo}%
\newtheorem{remark}{Remark}%

\theoremstyle{thmstylethree}%
\newtheorem{corollary}{Corollary}[theorem]
\newtheorem{lemma}[theorem]{Lemma}

\usepackage[backend=biber, style=alphabetic, backref, url=false, eprint=false]{biblatex}
\usepackage{newpxtext, eulerpx}    
\usepackage{hyperref}    
\usepackage[nameinlink]{cleveref}    

\def\R{\mathbb{R}}
\newcommand{\longto}{\longrightarrow}
\DeclareMathOperator{\grad}{grad}
\DeclareMathOperator{\curl}{curl}

\def\divgn{\operatorname{div}}
\DeclareMathOperator{\rot}{rot}
\newcommand{\ainnerproduct}[2]{\langle #1, #2 \rangle}
\newcommand{\aInnerproduct}[2]{\bigl\langle #1, #2 \bigr\rangle}
\DeclarePairedDelimiterX{\norm}[1]{\lVert}{\rVert}{#1}

\allowdisplaybreaks

\begin{document}

\title[Energy Conserving Discretizations of Maxwell's Equations]{Energy Conserving Higher Order Mixed Finite Element Discretizations of Maxwell's Equations}

\author{Archana Arya and Kaushik Kalyanaraman}
\address{Department of Mathematics, Indraprastha Institute of Information Technology, Delhi, New Delhi, 110020, India}
\email{archanaa@iiitd.ac.in, kaushik@iiitd.ac.in}

\begin{abstract}
We study a system of Maxwell's equations that describes the time evolution of electromagnetic fields with an additional electric scalar variable to make the system amenable to a mixed finite element spatial discretization. We demonstrate stability and energy conservation for the variational formulation of this Maxwell's system. We then discuss two implicit, energy conserving schemes for its temporal discretization: the classical Crank-Nicholson scheme and an implicit leapfrog scheme. We next show discrete stability and discrete energy conservation for the semi-discretization using these two time integration methods. We complete our discussion by showing that the error for the full discretization of the Maxwell's system with each of the two implicit time discretization schemes and with spatial discretization through a conforming sequence of de Rham finite element spaces converges quadratically in the step size of the time discretization and as an appropriate polynomial power of the mesh parameter in accordance with the choice of approximating polynomial spaces. Our results for the Crank-Nicholson method are generally well known but have not been demonstrated for this Maxwell's system. Our implicit leapfrog scheme is a new method to the best of our knowledge and we provide a complete error analysis for it. Finally, we show computational results to validate our theoretical claims using linear and quadratic Whitney forms for the finite element discretization for some model problems in two and three spatial dimensions.
\end{abstract}



\maketitle

\section{Introduction} \label{sec:introduction}

We study Maxwell's equations for the electromagnetic fields $E$ and $H$ posed as the following set of equations:
\begin{subequations}
  \begin{equation}
    \begin{aligned}
      \dfrac{\partial p}{\partial t} + \nabla \cdot \varepsilon E & = f_p \text{ in } \Omega \times (0, T], \\
      \nabla p + \varepsilon \dfrac{\partial E}{\partial t} - \nabla \times H &= f_E \text{ in } \Omega \times (0, T], \\
      \mu \dfrac{\partial H}{\partial t} + \nabla \times E &= f_H \text{ in } \Omega \times (0, T],
    \end{aligned} \label{eqn:maxwells_eqns}
  \end{equation}
  where $\Omega \subset \R^2/\R^3$ is a domain with Lipschitz boundary $\partial \Omega$ and with the following homogeneous boundary conditions:
  \begin{equation}
    p = 0,  E \times n = 0, H \cdot n = 0 \text{ on } \partial\Omega \times (0, T], \label{eqn:BCs}
  \end{equation}
  where $n$ is the unit outward normal to $\partial \Omega$, and with the following initial conditions:
  \begin{equation}
    p(x, 0) = p_0(x), E(x, 0) = E_0(x), \text{ and } H(x, 0) = H_0(x) \text{~for~} x \in \Omega. \label{eqn:ICs}
  \end{equation}
\end{subequations}
In these equations, $E(x, t)$ and $H(x, t)$ denote the electric and magnetic fields, respectively, and $p(x, t)$ is a physically fictitious term that is related to the time varying divergence of the electric field. We refer to all variables without the $x$ and $t$ descriptors when this is clear to avoid clutter. The additional electric variable facilitates a mixed formulation for the Maxwell's equations in the sense of finite element exterior calculus (FEEC) and allows for a direct approximation of $\nabla \cdot \varepsilon E$. This formulation was first introduced in \cite{AdPeZi2013} and further used in \cite{AdHuZi2017, AdCaHuZi2021}, and the authors there attribute it to Douglas Arnold who is one of the authors of several works in FEEC \cite{ArFaWi2010, Arnold2018}. We slightly generalize their $(p, E, H)$ formulation to allow for arbitrary right hand side functions $f_p(x, t), f_E(x, t)$ and $f_H(x, t)$ as shown in Equation~\eqref{eqn:maxwells_eqns}. Clearly, if $\nabla \cdot \varepsilon E = 0$, then with $p_0(x) = 0$, $f_p = 0$, $f_E = J(x, t)$ (which is the current density) and $f_H = 0$, we recover the more standard form of the Maxwell's equations. In our setup, we assume that $t \in [0, T]$ for some finite maximum time $T$. The material parameters $\varepsilon$ (electric permittivity) and $\mu$ (magnetic permeability) are assumed to be time independent and positive bounded piecewise constant functions, that is, there exist positive constants: $\varepsilon_{\min}, \varepsilon_{\max}, \mu_{\min}, \mu_{\max}$ satisfying $0 < \varepsilon_{\min} \le \varepsilon(x) \le \varepsilon_{\max} < \infty$ and $0 < \mu_{\min} \le \mu(x) \le \mu_{\max} < \infty$ for all $x \in \Omega$. Although this minor technical condition is necessary, for all intents and purposes, we treat these physical parameters as simply constants in our work here and these constants can then be taken to be either the inverse of the minimum values or the maximum values as appropriate in all our proofs if these parameters are spatially varying. Finally, we shall assume that the initial conditions provided satisfy:
\[
  \nabla \cdot (\varepsilon E_0) = p_0 , \text{~and~} \nabla \cdot (\mu H_0) = 0 \text{~in~} \Omega,
\]
to be consistent with our motivation for introducing $p$ in conjunction with $E$, and because magnetic fields are divergence free due to there being no magnetic charges.

The variational formulation for Equation~\eqref{eqn:maxwells_eqns} with boundary conditions in Equation~\eqref{eqn:BCs} is given by: find $(p, E, H) \in \mathring{H}^1_{\varepsilon^{-1}}(\Omega) \times \mathring{H}_{\varepsilon}(\curl; \Omega) \times \mathring{H}_{\mu}(\divgn; \Omega)$ such that:
\begin{subequations}
  \begin{alignat}{2}
    \aInnerproduct{\dfrac{\partial p}{\partial t}}{\widetilde{p}} - \aInnerproduct{ \varepsilon E}{\nabla \widetilde{p}} &= \aInnerproduct{f_p}{\widetilde{p}}, &&\quad \widetilde{p} \in \mathring{H}^1_{\varepsilon^{-1}}(\Omega), \label{eqn:maxwell_p_wf} \\
    \aInnerproduct{\nabla p}{\widetilde{E}} + \aInnerproduct{\varepsilon \dfrac{\partial E}{\partial t}}{\widetilde{E}} - \aInnerproduct{H}{\nabla \times \widetilde{E}} &= \aInnerproduct{f_E}{\widetilde{E}}, &&\quad \widetilde{E} \in \mathring{H}_{\varepsilon}(\curl; \Omega), \label{eqn:maxwell_E_wf} \\
    \aInnerproduct{\mu \dfrac{\partial H}{\partial t}}{\widetilde{H}} + \aInnerproduct{\nabla \times E}{\widetilde{H}}, &= \aInnerproduct{f_H}{\widetilde{H}}, &&\quad \widetilde{H} \in \mathring{H}_{\mu}(\divgn; \Omega), \label{eqn:maxwell_H_wf}
  \end{alignat}
\end{subequations}
for $t \in (0, T]$ with initial conditions as in Equation~\eqref{eqn:ICs}. For this problem, we shall use the following two time discretizations.

\smallskip \noindent \textbf{Crank-Nicholson scheme:} Find $(p(t^n), E(t^n), H(t^n)) \in \mathring{H}^1_{\varepsilon^{-1}}(\Omega) \times \mathring{H}_{\varepsilon}(\curl; \Omega) \times \mathring{H}_{\mu}(\divgn; \Omega)$ such that:
\begin{subequations}
  \begin{align}
    \aInnerproduct{\dfrac{p^n - p^{n - 1}}{\Delta t}}{\widetilde{p}} - \aInnerproduct{\dfrac{\varepsilon}{2} \left(  E^n +  E^{n - 1} \right)}{\nabla \widetilde{p}} &= \aInnerproduct{\dfrac{1}{2} \left( f_p^n + f_p^{n - 1} \right)}{\widetilde{p}},\label{eqn:maxwell_p_cn} \\
    \aInnerproduct{\dfrac{1}{2} \nabla \left( p^n + p^{n - 1} \right)}{\widetilde{E}} + \aInnerproduct{\varepsilon \dfrac{E^n - E^{n - 1}}{\Delta t}}{\widetilde{E}} - \aInnerproduct{\dfrac{1}{2} \left( H^n + H^{n - 1} \right)}{\nabla \times \widetilde{E}} &= \aInnerproduct{\dfrac{1}{2} \left(f_E^n + f_E^{n - 1} \right)}{\widetilde{E}}, \label{eqn:maxwell_E_cn}\\
    \aInnerproduct{\mu \dfrac{H^n - H^{n - 1}}{\Delta t}}{\widetilde{H}} + \aInnerproduct{\dfrac{1}{2} \nabla \times \left( E^n + E^{n - 1} \right)}{\widetilde{H}} &= \aInnerproduct{\dfrac{1}{2} \left( f_H^n + f_H^{n - 1} \right)}{\widetilde{H}}, \label{eqn:maxwell_H_cn}
  \end{align}
\end{subequations}
for all $\widetilde{p} \in \mathring{H}^1_{\varepsilon^{-1}}(\Omega), \; \widetilde{E} \in \mathring{H}_{\varepsilon}(\curl; \Omega)$ and $\widetilde{H} \in \mathring{H}_{\mu}(\divgn; \Omega)$, and with uniform discretization of $[0, T]$ as $t^n \coloneq n \Delta t$, $n = 0, 1, \dots, N$ for some $\Delta t > 0$ being the fixed time step size such that $N \Delta t = T$.

\smallskip \noindent \textbf{Implicit leapfrog scheme:} Find $(p(t^{n + 1/2}), E(t^{n + 1/2}), H(t^n)) \in \mathring{H}^1_{\varepsilon^{-1}}(\Omega) \times \mathring{H}_{\varepsilon}(\curl; \Omega) \times \mathring{H}_{\mu}(\divgn; \Omega)$ such that:
\begin{subequations}
\begin{align}
  \aInnerproduct{\dfrac{p^{n + \frac{1}{2}} - p^{n - \frac{1}{2}}}{\Delta t}}{\widetilde{p}} - \aInnerproduct{\dfrac{\varepsilon}{2} \left( E^{n + \frac{1}{2}} + E^{n - \frac{1}{2}} \right)}{\nabla \widetilde{p}} &= \aInnerproduct{ f_p^{n}}{\widetilde{p}}, \label{eqn:maxwell_p_lf} \\
  \aInnerproduct{\dfrac{1}{2} \nabla \left(p^{n + \frac{1}{2}} + p^{n - \frac{1}{2}} \right)}{\widetilde{E}} + \aInnerproduct{\varepsilon \dfrac{E^{n + \frac{1}{2}} - E^{n - \frac{1}{2}}}{\Delta t}}{\widetilde{E}} - \aInnerproduct{\dfrac{1}{2} \left( H^{n + 1} + H^n \right)}{\nabla \times \widetilde{E}} &= \aInnerproduct{f_E^n}{\widetilde{E}}, \label{eqn:maxwell_E_lf} \\
  \aInnerproduct{\mu \dfrac{H^{n + 1} - H^n}{\Delta t}}{\widetilde{H}} +  \aInnerproduct{\dfrac{1}{2} \nabla \times \left( E^{n + \frac{1}{2}} + E^{n - \frac{1}{2}} \right)}{\widetilde{H}} &= \aInnerproduct{f_H^{n + \frac{1}{2}}}{\widetilde{H}}, \label{eqn:maxwell_H_lf}
\end{align}
\end{subequations}
for all $\widetilde{p} \in \mathring{H}^1_{\varepsilon^{-1}}(\Omega), \; \widetilde{E} \in \mathring{H}_{\varepsilon}(\curl; \Omega)$ and $\widetilde{H} \in \mathring{H}_{\mu}(\divgn; \Omega)$, and with two uniform staggered discretizations of $[0, T]$: for the variables $p$ and $E$ as $t^{n + \frac{1}{2}} \coloneq (n + 1/2) \Delta t$, $n = 0, 1, \dots, N - 1$ for some $\Delta t > 0$ being the fixed time step size such that $N \Delta t = T$, and for the variable $H$ as $t^n \coloneq n \Delta t$, $n = 0, 1, \dots, N$. Clearly, our method is inspired by the classical (but explicit) leapfrog scheme, and we obtain $p, E$ at half time indices and $H$ at integer time indices. We propose this semi discretization by evaluating all innerproduct terms in Equations~\labelcref{eqn:maxwell_p_wf,eqn:maxwell_E_wf,eqn:maxwell_H_wf} for $p$ and $E$ at $t^n$ and for $H$ at $t^{n+\frac{1}{2}}$. To bootstrap this scheme using the given initial values at $t = 0$ as in Equation~\eqref{eqn:ICs}, we shall use the following discretization for the first half time step for $p$ and $E$, and for the first time step for $H$:
\begin{subequations}
  \begin{align}
    \aInnerproduct{\dfrac{p^{\frac{1}{2}} - p_0}{\Delta t/2}}{\widetilde{p}} - \dfrac{1}{2} \aInnerproduct{\dfrac{\varepsilon}{2} \left( E^{\frac{1}{2}} + E_0 \right)}{\nabla \widetilde{p}} &= \aInnerproduct{ f_p^0}{\widetilde{p}}, \label{eqn:maxwell_p0_lf} \\
  \dfrac{1}{2}  \aInnerproduct{\dfrac{1}{2} \nabla \left(  p^{\frac{1}{2}} +  p_0 \right)}{\widetilde{E}} + \aInnerproduct{\varepsilon \dfrac{E^{\frac{1}{2}} - E_0}{\Delta t/2}}{\widetilde{E}} - \aInnerproduct{\dfrac{1}{2} \left( H^1 + H_0 \right)}{\nabla \times \widetilde{E}} &= \aInnerproduct{f_E^0}{\widetilde{E}}, \label{eqn:maxwell_E0_lf} \\
    \aInnerproduct{\mu \dfrac{H^1 - H_0}{\Delta t}}{\widetilde{H}} + \dfrac{1}{2} \aInnerproduct{\dfrac{1}{2} \nabla \times \left(E^{\frac{1}{2}} + E_0 \right)}{\widetilde{H}} &= \aInnerproduct{f_H^{\frac{1}{2}}}{\widetilde{H}}. \label{eqn:maxwell_H0_lf}
  \end{align}
\end{subequations}
In all the time discretizations shown in Equations~\labelcref{eqn:maxwell_p_cn,eqn:maxwell_E_cn,eqn:maxwell_H_cn,eqn:maxwell_p_lf,eqn:maxwell_E_lf,eqn:maxwell_H_lf,eqn:maxwell_p0_lf,eqn:maxwell_E0_lf,eqn:maxwell_H0_lf}, we use a superscript on our variables of interest to denote their values at the time corresponding to this time index, for example, $p^n \coloneq p(t^n, x)$ and so on.

\subsection{Antecedents} \label{sec:antecedents}

A system of Maxwell's equations such as in Equation~\eqref{eqn:maxwells_eqns} has an exceedingly rich history of solution methodologies and we will not be able to justly point out all contributions, nevertheless we shall highlight some of the important ones biased and limited by our knowledge and depth of historical understanding. The field of computational electromagnetics probably had its first big break with the development of the finite difference time domain (FDTD) scheme in~\cite{Yee1966} which led to further developments later on in mimetic finite difference schemes for Maxwell's equations~\cite{HySh1999, Bossavit2001, LiMaSh2014, AdCaHuZi2021}. Posing Maxwell's equations in the language of differential forms~\cite{Deschamps1981}, and a combinatorial and geometric discretization of exterior calculus as discrete exterior calculus (DEC)~\cite{Hirani2003} has led to a plethora of methods for Maxwell's equations such as direct geometric discretizations popularized by~\cite{Tonti2001, Tonti2001a, Tonti2002}, and solutions of Maxwell's equations using DEC by~\cite{Rabina2014, RaMoRo2015, ChCh2017, SaPo2018, ZhNaJiCh2023} to illustrate a few. Then there are time domain finite element methods for the solution of Maxwell's equations in various transient and time harmonic formulations. We begin tracing their history from works by~\cite{Nedelec1978, Nedelec1980, Nedelec1986} which introduced in the context of computational electromagnetics what we now refer to as Whitney forms. Then there are several contribution by Monk such as~\cite{Monk1992, Monk1992a, MaMo1995}, his book~\cite{Monk2003} and the tutorial-style article~\cite{ChMo2012} which have provided analyses for finite element methods for time domain Maxwell's equations. Bossavit's popularization of the ``rationale for using Whitney elements'' in computational electromagnetism through his various works over the decades~\cite{Bossavit1988, Bossavit1988a, Bossavit1990, Bossavit1992}, via his book~\cite{Bossavit1998a} as well as his insights on the discretization of electromagnetism posed via calculus on manifolds~\cite{Bossavit1991, Bossavit1998, Bossavit2010} have inspired several newer generations of researchers in this field like us. We cannot miss pointing to Hiptmair's works~\cite{Hiptmair1999, Hiptmair2002, Hiptmair2015}, Christiansen's works in this context such as~\cite{Christiansen2009}, and many other broader finite element exterior calculus works such as~\cite{RaBo2009, RaBo2009a, CHMuOw2011_AN, BoRa2014, ChRa2016} and all of which provide the foundations on which our current work stands. We finally wish to point out that error analysis for a full discretization of Maxwell's equations using finite elements in space and finite difference methods in time are still being analysed and clarified for even simpler implicit schemes such as backward Euler for time differencing as in~\cite{AnAn2019}. Our work presented here is in similar spirit to fill the gap for the analysis of some second-order accurate implicit time discretizations including providing a new implicit method for discretization of Maxwell's equations.

The rest of our paper is organized as follows. In Section~\ref{sec:prelim}, we recall some of the basics needed for our work. In Section~\ref{sec:stability_wf}, we prove the stability of solution for the variational formulation outlined in Equations~\labelcref{eqn:maxwell_p_wf,eqn:maxwell_E_wf,eqn:maxwell_H_wf} and therefore infer its existence and uniqueness. In Section~\ref{sec:time_disc_stab}, we present our discrete stability and energy conservation results for the Crank-Nicholson discretization in Equations~\labelcref{eqn:maxwell_p_cn,eqn:maxwell_E_cn,eqn:maxwell_H_cn} and for the implicit leapfrog scheme in Equations~\labelcref{eqn:maxwell_p_lf,eqn:maxwell_E_lf,eqn:maxwell_H_lf,eqn:maxwell_p0_lf,eqn:maxwell_E0_lf,eqn:maxwell_H0_lf}. We state and prove our theorems on the convergence of the errors for a full discretization using higher order Whitney finite elements for these two time discretizations in Section~\ref{sec:fllerrrestmts}. Finally, in Section~\ref{sec:numerics}, we present some illustrative computational results on model problems in $\R^2$ and $\R^3$ to both validate our various theoretical results and demonstrate the feasibility of performing computations.

\section{Preliminaries} \label{sec:prelim}

We present here bare minimum basics needed for a ready reference and discussion of the various theorems and results in our work. We have already pointed to several excellent sources for FEEC and time domain finite element methods for computational electromagnetics in Section~\ref{sec:antecedents}. We are indebted to those foundations and have tried to liberally borrow their overall proof methodologies and general proof strategies for presenting our contributions here.

We begin by first stating the various Hilbert spaces that we shall use as our function spaces in this work. To begin to do so, we first state that the innerproduct on our Hilbert space $(U, \aInnerproduct{\cdot}{\cdot}_{\alpha})$ is weighted in the following way: for any $u, v \in U$ and a positive real number $\alpha$, $\aInnerproduct{u}{v}_{\alpha} \coloneq \aInnerproduct{\alpha u}{v}$. Consequently, for the same positive real number we have that $\norm{u}_{\alpha}^2 \coloneq \aInnerproduct{u}{u}_{\alpha}$. Next, we note that we do not distinguish between scalar- and vector-valued points or functions by their typeface and leave it to the context to make it clear as we find this choice visually friendly. Therefore, instead of $x \in \R$ or $\boldsymbol{x} \in \R^n$, we simply write $x \in \Omega$ and likewise $f$ may represent either a scalar- or vector-valued function. Consequently, we also do not distinguish in notation function spaces defined over regular (or smooth) scalar or vector functions and simply use $C^r(\Omega)$ where $r$ is the regularity to denote such functions. Thus, $L^2(\Omega)$ can mean the space of either scalar- or vector-valued functions with a bounded $L^2$ norm. With these reading provisos, for a positive number $\alpha$ we denote by $L_\alpha^2$ the set of functions which have a bounded $L^2$ norm $\norm{\cdot}_{\alpha}$. We next define for our purposes the various function spaces which we need in our work here.

\begin{align*}
  H_{\varepsilon^{-1}}^1(\Omega) &\coloneq \{u \in L^2_{\varepsilon^{-1}}(\Omega) \mid \nabla u \in L^2_{\varepsilon^{-1}}(\Omega) \}, \\
  \mathring{H}_{\varepsilon^{-1}}^1(\Omega) &\coloneq \{u \in H^1_{\varepsilon^{-1}}(\Omega) \mid u = 0 \text{ on } \partial \Omega) \}, \\
  H_{\varepsilon}(\curl; \Omega) &\coloneq \{u \in L^2_\varepsilon(\Omega) \mid \nabla \times u \in L^2_\varepsilon(\Omega) \}, \\
  \mathring{H}_{\varepsilon}(\curl; \Omega) &\coloneq \{u \in H_\varepsilon(\curl; \Omega) \mid u \times n = 0 \text{ on } \partial \Omega \}, \\
  H_{\mu}(\divgn; \Omega) &\coloneq \{u \in L^2_\mu(\Omega) \mid \nabla \cdot u \in L^2_\mu(\Omega) \}, \\
  \mathring{H}_{\mu}(\divgn; \Omega) &\coloneq \{u \in H_\mu(\divgn; \Omega) \mid u \cdot n = 0 \text{ on } \partial \Omega \},
\end{align*}

\subsection{de Rham Complex Finite Elements}

We refer to \cite{ArFaWi2006,ArFaWi2010,Arnold2018} for conforming finite element spaces forming a de Rham complex as also to \cite{CHMuOw2011_AN,ChRa2016} for various details pertaining to the construction of finite element systems, the Whitney forms basis, and the cannonical and smoothed projection operators. We wish to simply summarize what we need by stating the following two commuting diagrams of de Rham complexes of the function spaces and their discretizations in $\R^2$ and $\R^3$. We wish to note that $\mathcal{P}_r^-$ denotes the degree $r$ Whitney forms which may then be interepreted to mean scalar Lagrange basis or vector Nedelec or Raviart-Thomas elements. Also, in these diagrams $\Pi_h$ represents the appropriate cannonical or smoothed projection to provide us with the necessary approximability results~\cite[see][Theorems 5.3, 5.6]{ArFaWi2006}.

\begin{center}
  \begin{tikzpicture}[baseline=(a).base]
    \node[scale=.9] (a) at (0,0){
      \begin{tikzcd}[column sep=3em, row sep=2em, text height=1ex, text depth=0ex, ampersand replacement=\&, every
        label/.style={font=\small, auto}]
        \mathring{H}_{\varepsilon^{-1}}^1 (\Omega) \arrow[r, -stealth, shift left, "\rot"] \arrow[d, -stealth, "\Pi_h^0"] \& \arrow[l, -stealth, shift left, "\curl"] \mathring{H}_{\varepsilon} (\divgn, \Omega) \arrow[r, -stealth, shift left, "-\divgn"] \arrow[d, -stealth, "\Pi_h^1"] \& \arrow[l, -stealth, shift
        left, "\grad"] L_{\mu}^2(\Omega) \arrow[d, -stealth, "\Pi_h^2"]  \\[2em] 
        \mathring{H}_{\varepsilon^{-1}}^1(\Omega) \hspace{.5em}  \arrow[r, -stealth, shift left, yshift = -.7em, "\rot"]  \& \arrow[l, -stealth, shift left, yshift = -.7em, "\curl"] \mathring{H}_{\varepsilon} (\divgn, \Omega) \arrow[r, -stealth, shift
        left, yshift = -.7em, "-\divgn"] \& \arrow[l, -stealth, shift
        left, yshift = -.7em, "\grad"] \hspace{.5em} L_{\mu}^2(\Omega) \\[-2em]
        \cap \, \mathcal{P}_r^-(\Omega) \hspace{.5em}\& \cap \, \mathcal{P}_r^-(\Omega) \& \hspace{.5em} \cap \, \mathcal{P}_r^-(\Omega) 
      \end{tikzcd}
    };
  \end{tikzpicture}
\end{center}

\begin{center}
  \begin{tikzpicture}[baseline=(a).base]
    \node[scale=.9] (a) at (0,0){
      \begin{tikzcd}[column sep=3em, row sep=2em, text height=1ex, text depth=0ex, ampersand replacement=\&, every
        label/.style={font=\small, auto}]
        \mathring{H}_{\varepsilon^{-1}}^1 (\Omega) \arrow[r, -stealth, shift left, "\grad"] \arrow[d, -stealth, "\Pi_h^0"] \& \arrow[l, -stealth, shift left, "-\divgn"] \mathring{H}_{\varepsilon} (\curl, \Omega) \arrow[r, -stealth, shift left, "\curl"]  \arrow[d, -stealth, "\Pi_h^1"] \& \arrow[l, -stealth, shift left, "\curl"] \mathring{H}_{\mu} (\divgn, \Omega) \arrow[r, -stealth, shift left, "-\divgn"]  \arrow[d, -stealth, "\Pi_h^2"] \& \arrow[l, -stealth, shift left, "\grad"] L^2(\Omega) \arrow[d, -stealth, "\Pi_h^3"] \\[2em]
        \mathring{H}_{\varepsilon^{-1}}^1(\Omega) \hspace{.5em} \arrow[r, -stealth, shift left, yshift = -.7em, "\grad"] \& \arrow[l, -stealth, shift left, yshift = -.7em, "-\divgn"] \mathring{H}_{\varepsilon} (\curl, \Omega) \arrow[r, -stealth, yshift = -.7em, shift left, "\curl"]  \& \arrow[l, -stealth, yshift = -.7em, shift left, "\curl"] \mathring{H}_{\mu} (\divgn, \Omega) \arrow[r, -stealth, yshift = -.7em, shift left, "-\divgn"]  \& \arrow[l, -stealth, yshift = -.7em, shift left, "\grad"]  \hspace{.5em} L^2(\Omega) \\[-2em]
        \cap \, \mathcal{P}_r^-(\Omega) \hspace{.5em}\& \cap \, \mathcal{P}_r^-(\Omega) \& \cap \, \mathcal{P}_r^-(\Omega) \& \hspace{.5em} \cap \, \mathcal{P}_r^-(\Omega)
      \end{tikzcd}
    };
  \end{tikzpicture}
\end{center}

\subsection{Algebraic Identities and Inequalities}

In the following, we shall take $u, v$ to be in a Hilbert space $(U, \ainnerproduct{\cdot}{\cdot}_{\alpha})$ for a fixed positive scalar $\alpha$, and $a$ and $b$ to be nonnegative real numbers. 
\begin{description}
\item[Triangle] $\norm{u + v}_{\alpha} \le \norm{u}_{\alpha} + \norm{v}_{\alpha}$.
  
\item[Cauchy-Schwarz] $\aInnerproduct{u}{v}_{\alpha} \le \norm{u}_{\beta^{-1}\alpha} \norm{v}_{\beta \alpha}$ for a nonnegative scalar $\beta$.
  
\item[Polarization identity] $2 \aInnerproduct{u - v}{u}_{\alpha} = \norm{u}_{\alpha}^2 + \norm{u - v}_{\alpha}^2 - \norm{v}_\alpha^2$.

\item[Polarization inequality] $2 \aInnerproduct{u}{v}_{\alpha} \le \beta^{-1} \norm{u}_{\alpha}^2 + \beta \norm{v}_{\alpha}^2$ for a nonnegative scalar $\beta$.

\item[Arithmetic mean-geometric mean (AM-GM)] $\sqrt{a b}\le \dfrac{a + b}{2} \implies \left( a + b \right)^2 \le 2 \left( a^2 + b^2 \right)$.
\end{description}

\subsection{Gronwall-OuLang Inequalities}

A recurring theme in the various proofs for establishment of stability and error convergence for time discretization of the Maxwell's system is Gronwall's inequality. We state below the versions of it most pertinent to our usage in both the smooth and discrete settings. We wish to also note the relation of Gronwall's inequality in the smooth setting to the inequality in~\cite{OuLang1957}, and therefore we refer to these results in our work here together as Gronwall-OuLang inequality.

\begin{lemma}[OuLang, \cite{OuLang1957}] \label{lemma:oulang}
  Let $u$ and $f$ be real-valued nonnegative continuous functions defined for all $t \ge 0$. If:
  \[
    u^2(t) \le c^2 + 2 \int\limits_0^t f(s) u(s) ds \text{ for all } t \ge 0,
  \]
  where $c \ge 0$ is a constant, then:
  \[
    u(t) \le c + \int\limits_0^t f(s) ds \text{ for all } t \ge 0.
  \]
  
\end{lemma}

\begin{lemma}[Gronwall, {\cite[Lemma 1.4.1]{QuVa1994}}] \label{lemma:gronwall}
  Let $f \in L^1(t_0, T)$ be a nonnegative function, $g$ and $u$ be continuous functions on $[t_0, T]$. If $u$ satisfies:
  \[
    u(t) \le g(t) + \int\limits_{t_0}^t f(s) u(s) ds \text{ for all } t \in [t_0, T],
  \]
  then:
  \[
    u(t) \le g(t) + \int\limits_{t_0}^t f(s) g(s) \exp\left( \int\limits_{s}^t f(\tau) d\tau \right) \text{ for all } t \in [t_0, T].
  \]
  Moreover, if $g$ is nondecreasing, then:
  \[
    u(t) \le g(t) \exp \left( \int\limits_{t_0}^t f(\tau) d\tau \right) \text{ for all } t \in [t_0, T].
  \]  
\end{lemma}

\begin{lemma}[Discrete Gronwall, {\cite[Lemma 2]{AnAn2019}}] \label{lemma:gronwall_dscrt}
  Let $\delta \ge 0$, $g_0 \ge 0$ and $(a_n)$, $(b_n)$, $(c_n)$ and $(\gamma_n)$ be sequences of nonnegative numbers such that:
  \[
    a_N + \delta \sum\limits_{n = 0}^N b_n \le \delta \sum\limits_{n = 0}^{N} \left( \gamma_n a_n + c_n \right) + g_0 \text{ for all } N = 0, 1, \dots.
  \]
  Assuming that $\gamma_n \delta < 1$ for all $n$ and setting $\sigma_n \coloneq \left( 1 - \gamma_n \delta \right)^{-1}$, for all $N \ge 0$ we have that:
  \[
    a_N + \delta \sum\limits_{n = 0}^N b_n \le \left( \delta \sum\limits_{n = 0}^N c_n + g_0 \right) \exp \left( \delta \sum\limits_{n = 0}^N \sigma_n \gamma_n \right).
  \]
\end{lemma}

\section{Stability of Variational Formulation} \label{sec:stability_wf}

We first demonstrate the existence and uniqueness of the solution to the variational formulation of the Maxwell's system as presented in Equations~\labelcref{eqn:maxwell_p_wf,eqn:maxwell_E_wf,eqn:maxwell_H_wf}, To do so, we have the following energy estimate for it.

\begin{theorem}[Energy Estimate] \label{thm:smth_enrgy_estmt}
  Let $f_p \in L^1[0, T] \times L_{\varepsilon^{-1}}^2(\Omega)$, $f_E \in L^1[0, T] \times L^2_{\varepsilon^{-1}}(\Omega)$, and $f_H \in L^1[0, T] \times L^2_{\mu^{-1}}(\Omega)$. Then the solution $(p, E, H)$ of Equations~\labelcref{eqn:maxwell_p_wf,eqn:maxwell_E_wf,eqn:maxwell_H_wf} with initial conditions as in Equation~\eqref{eqn:ICs} and assuming sufficient regularity with $p \in C^1[0, T] \times \mathring{H}^1_{\varepsilon^{-1}}(\Omega), \; E \in C^1[0, T] \times \mathring{H}_{\varepsilon}(\curl; \Omega)$, and $H \in C^1[0, T] \times \mathring{H}_{\mu}(\divgn; \Omega)$ satisfies:
  \begin{multline*}
    \norm{p}_{\varepsilon^{-1}} + \norm{E}_{\varepsilon} + \norm{H}_{\mu} \le \sqrt{3} \Big[ \norm{p_0}_{\varepsilon^{-1}} + \norm{E_0}_{\varepsilon} + \norm{H_0}_{\mu} \, + \\
    \norm{f_p}_{L^1[0, T] \times L^2_{\varepsilon^{-1}}(\Omega)} + \norm{f_E}_{L^1[0, T] \times L^2_{\varepsilon^{-1}}(\Omega)} + \norm{f_H}_{L^1[0, T] \times L^2_{\mu^{-1}}(\Omega)} \Big].
  \end{multline*}
\end{theorem}

\begin{proof}
  Since Equations~\labelcref{eqn:maxwell_p_wf,eqn:maxwell_E_wf,eqn:maxwell_H_wf} are true for all $\widetilde{p}\in \mathring{H}^1_{\varepsilon^{-1}}(\Omega), \; \widetilde{E} \in \mathring{H}_{\varepsilon}(\curl; \Omega), \; \widetilde{H} \in \mathring{H}_{\mu}(\divgn; \Omega)$, we choose $\widetilde{p} = \varepsilon^{-1} p, \widetilde{E} = E$ and $\widetilde{H} = H$ which then yields the following set of equations:
  \begin{align*}
      \aInnerproduct{\dfrac{\partial p}{\partial t}}{\varepsilon^{-1} p} - \aInnerproduct{ \varepsilon E}{\nabla \varepsilon^{-1} p} &= \aInnerproduct{f_p}{\varepsilon^{-1} p}, \\
      \aInnerproduct{\nabla p}{E} + \aInnerproduct{\varepsilon \dfrac{\partial E}{\partial t}}{E} - \aInnerproduct{H}{\nabla \times E} &= \aInnerproduct{f_E}{E}, \\
      \aInnerproduct{\mu \dfrac{\partial H}{\partial t}}{H} + \aInnerproduct{\nabla \times E}{H} &= \aInnerproduct{f_H}{H}.
\end{align*}
Adding these together and using the properties of the inner product, we obtain:
\begin{equation}
  \aInnerproduct{\dfrac{\partial p}{\partial t}}{\varepsilon^{-1} p} + \aInnerproduct{\varepsilon \dfrac{\partial E}{\partial t}}{E} + \aInnerproduct{\mu \dfrac{\partial H}{\partial t}}{H} = \aInnerproduct{f_p}{\varepsilon^{-1} p} + \aInnerproduct{f_E}{E} + \aInnerproduct{f_H}{H}. \label{eqn:p+E+H_wf}
\end{equation}
Now consider that:
\[
  \dfrac{d}{dt} \norm{p}^2_{\varepsilon^{-1}} = \dfrac{d}{dt} \ainnerproduct{\varepsilon^{-1}p}{p} = 2\ainnerproduct{\dfrac{\partial p}{\partial t}}{\varepsilon^{-1} p}.
\]
Similarly, we also have:
\[
  \dfrac{d}{dt} \norm{E}^2_{\varepsilon} = 2\ainnerproduct{\varepsilon \dfrac{\partial E}{\partial t}}{E}, \text{ and } \dfrac{d}{dt} \norm{H}^2_{\mu} = 2\ainnerproduct{\mu \dfrac{\partial H}{\partial t}}{H}.  \]
Using these in Equation~\eqref{eqn:p+E+H_wf} and the Cauchy-Schwarz inequality, we get:
\begin{align*}
  \dfrac{1}{2} \dfrac{d}{dt} \big[ \norm{p}^2_{\varepsilon^{-1}} + \norm{E}^2_{\varepsilon} + \norm{H}^2_{\mu} \big] &= \aInnerproduct{f_p}{\varepsilon^{-1} p} + \aInnerproduct{\varepsilon^{-\frac{1}{2}} f_E}{\varepsilon^{\frac{1}{2}} E} + \aInnerproduct{\mu^{-\frac{1}{2}} f_H}{\mu^{\frac{1}{2}} H}, \\
& \le \norm{f_p}_{\varepsilon^{-1}} \norm{p}_{\varepsilon^{-1}} + \norm{f_E}_{\varepsilon^{-1}} \norm{E}_{\varepsilon} + \norm{f_H}_{\mu^{-1}} \norm{H}_{\mu}.
\end{align*}
Integrating both sides of this inequality with respect to $t$ from $0$ to $T$, we next get:
\begin{multline*}
  \norm{p}^2_{\varepsilon^{-1}} + \norm{E}^2_{\varepsilon} + \norm{H}^2_{\mu} \le \\ \norm{p_0}^2_{\varepsilon^{-1}} + \norm{E_0}^2_{\varepsilon} + \norm{H_0}^2_{\mu} + 2 \int\limits_0^T \left( \norm{f_p}_{\varepsilon^{-1}} + \norm{f_E}_{\varepsilon^{-1}} + \norm{f_H}_{\mu^{-1}} \right) \sqrt{\norm{p}^2_{\varepsilon^{-1}} + \norm{E}^2_{\varepsilon} + \norm{H}^2_{\mu}} ds,
\end{multline*}
Then using the Gronwall-OuLang inequality as described in Lemma~\labelcref{lemma:oulang}, this results in:
\[
  \sqrt{\norm{p}^2_{\varepsilon^{-1}} + \norm{E}^2_{\varepsilon} + \norm{H}^2_{\mu}} \le \sqrt{\norm{p_0}^2_{\varepsilon^{-1}} + \norm{E_0}^2_{\varepsilon} + \norm{H_0}^2_{\mu}} + \int\limits_0^T \left(\norm{f_p}_{\varepsilon^{-1}} + \norm{f_E}_{\varepsilon^{-1}} + \norm{f_H}_{\mu^{-1}} \right) ds.
\]
Finally, using the equivalence of $1$- and $2$-norms, we obtain our result:
\begin{multline*}
  \norm{p}_{\varepsilon^{-1}} + \norm{E}_{\varepsilon} + \norm{H}_{\mu} \le 
  \sqrt{3} \Big[ \norm{p_0}_{\varepsilon^{-1}} + \norm{E_0}_{\varepsilon} + \norm{H_0}_{\mu} + \\
  \norm{f_p}_{L^1[0, T] \times L^2_{\varepsilon^{-1}}(\Omega)} + \norm{f_E}_{L^1[0, T] \times L^2_{\varepsilon^{-1}}(\Omega)} + \norm{f_H}_{L^1[0, T] \times L^2_{\mu^{-1}}(\Omega)} \Big]. \qedhere
\end{multline*}
\end{proof}

\begin{corollary}[Energy Conservation] \label{corr:smth_enrgy_cnsrvtn}
  If the forcing functions in Equation~\eqref{eqn:maxwells_eqns} of the Maxwell's equations are all zero, that is, $f_p = 0$ and $f_E = f_H = 0$ and with initial conditions as in Equation~\eqref{eqn:ICs} then:
  \[
    \norm{p}^2_{\varepsilon^{-1}} + \norm{E}^2_{\varepsilon} + \norm{H}^2_{\mu} = \norm{p_0}^2_{\varepsilon^{-1}} + \norm{E_0}^2_{\varepsilon} + \norm{H_0}^2_{\mu}.
  \]
\end{corollary}

\begin{remark}[Uniqueness of Solution]
As a result of Theorem~\ref{thm:smth_enrgy_estmt}, and by a standard argument, for example \cite[Theorem 5, Section 2.4, Chapter 2]{Evans2010}, the solution to the variational formulation of the Maxwell's equations in Equations~\labelcref{eqn:maxwell_p_wf,eqn:maxwell_E_wf,eqn:maxwell_H_wf} with initial conditions as in Equation~\eqref{eqn:ICs} has a unique solution.
\end{remark}

\section{Time Discretization Stability} \label{sec:time_disc_stab}

For each of our time discretization schemes, we next show that these semidiscretizations in time yield stable methods and we provide error estimates which show that these schemes are quadratically convergent in each of their chosen sufficiently small but fixed time step.

\subsection{Crank-Nicholson Scheme}

\begin{theorem}[Discrete Energy Estimate] \label{thm:dscrt_enrgy_estmt_cn}
For the semidiscretization using the Crank-Nicholson scheme as given in Equations~\labelcref{eqn:maxwell_p_cn,eqn:maxwell_E_cn,eqn:maxwell_H_cn}, for any fixed time step $\Delta t > 0$ sufficiently small, there exists a positive bounded constant $C$ independent of $\Delta t$ such that:
\[
  \norm{p^N}_{\varepsilon^{-1}} + \norm{E^N}_{\varepsilon} + \norm{H^N}_{\mu} \le C.
\]
\end{theorem}

\begin{proof}
Since Equations~\labelcref{eqn:maxwell_p_cn,eqn:maxwell_E_cn,eqn:maxwell_H_cn} are true for all $\widetilde{p}\in \mathring{H}^1_{\varepsilon^{-1}}(\Omega)$, $\widetilde{E} \in \mathring{H}_{\varepsilon}(\curl; \Omega)$, $\widetilde{H} \in \mathring{H}_{\mu}(\divgn; \Omega)$, using $\widetilde{p} = 2 \Delta t \varepsilon^{-1} \left(p^n + p^{n-1}\right), \widetilde{E} = 2 \Delta t \left(E^n + E^{n-1}\right) $ and $\widetilde{H} = 2 \Delta t \left(H^n + H^{n-1}\right)$ in them, we obtain the following:
\begin{multline*}
  2 \aInnerproduct{p^n - p^{n-1}}{\varepsilon^{-1} \left( p^n + p^{n-1} \right)} -
  \Delta t \aInnerproduct{E^n + E^{n-1}}{\nabla \left(p^n + p^{n-1}\right)} =
  \Delta t \aInnerproduct{f_p^n + f_p^{n-1}}{\varepsilon^{-1} \left(p^n + p^{n-1}\right)},
\end{multline*}
\vspace{-1em} \begin{multline*}
  \Delta t \aInnerproduct{\nabla \left(p^n + p^{n-1} \right)}{E^n + E^{n-1}} +  2 \aInnerproduct{\varepsilon \left(E^n - E^{n-1}\right)}{E^n + E^{n-1}} -
  \Delta t \aInnerproduct{H^n + H^{n-1}}{\nabla \times \left(E^n + E^{n-1}\right)} \\ = \Delta t \aInnerproduct{f_E^n + f_E^{n-1}}{E^n + E^{n-1}},
\end{multline*}
\vspace{-1em} \begin{multline*}
  2 \aInnerproduct{\mu \left(H^n - H^{n-1} \right)}{H^n + H^{n-1}} + \Delta t \aInnerproduct{\nabla \times \left(E^n + E^{n-1} \right)}{H^n + H^{n-1}} = \Delta t \aInnerproduct{ f_H^n + f_H^{n-1}}{H^n + H^{n-1}}. 
\end{multline*}
Adding these equations together and using the properties of inner product, we get:
\begin{multline}
  2 \left( \aInnerproduct{\varepsilon^{-1} \left(p^n - p^{n - 1} \right)}{p^n + p^{n - 1}} + \right. \aInnerproduct{\varepsilon \left(E^n - E^{n - 1} \right)}{E^n + E^{n - 1}} + 
  \left. \aInnerproduct{\mu \left( H^n - H^{n - 1} \right)}{H^n + H^{n - 1}} \right) \\
  = \Delta t \left( \aInnerproduct{f_p^n + f_p^{n - 1}}{\varepsilon^{-1} \left( p^n + p^{n - 1} \right)} + \aInnerproduct{f_E^n + f_E^{n - 1}}{E^n + E^{n - 1}} + \aInnerproduct{ f_H^n + f_H^{n - 1}}{H^n + H^{n - 1}} \right). \label{eqn:p+E+H_cn}
\end{multline}
Using the Cauchy-Schwarz, AM-GM, and Triangle inequalities for the first term on the right hand side of the above expression, we derive the following inequality:
\begin{align*}
  \aInnerproduct{f_p^n + f_p^{n - 1}}{\varepsilon^{-1} \left(p^n + p^{n - 1}\right)} &\le \norm{f_p^n + f_p^{n - 1}}_{\varepsilon^{-1}} \norm{p^n + p^{n - 1}}_{\varepsilon^{-1}}, \\
&\le \dfrac{1}{2} \left( \norm{f_p^n + f_p^{n - 1}}^2_{\varepsilon^{-1}} + \norm{p^n + p^{n - 1}}^2_{\varepsilon^{-1}} \right), \\
&\le \norm{f_p^n}^2_{\varepsilon^{-1}} + \norm{f_p^{n - 1}}^2_{\varepsilon^{-1}} + \norm{p^n}^2_{\varepsilon^{-1}} + \norm{p^{n - 1}}^2_{\varepsilon^{-1}}.
\end{align*}
Similarly in Equation~\eqref{eqn:p+E+H_cn}, for the second and third terms on the right hand side, we have:
\begin{align*}
\aInnerproduct{f_E^n + f_E^{n - 1}}{E^n + E^{n - 1}} & \le \norm{f_E^n}^2_{\varepsilon^{-1}} + \norm{f_E^{n - 1}}^2_{\varepsilon^{-1}} + \norm{E^n}^2_{\varepsilon} + \norm{E^{n - 1}}^2_{\varepsilon}, \\
\aInnerproduct{f_H^n + f_H^{n - 1}}{H^n + H^{n - 1}} &\le \norm{f_H^n}^2_{\mu^{-1}} + \norm{f_H^{n - 1}}^2_{\mu^{-1}} + \norm{H^n}^2_{\mu} + \norm{H^{n - 1}}^2_{\mu}.
\end{align*}
Incorporating these inequalities into Equation~\eqref{eqn:p+E+H_cn}, we obtain:
\begin{multline*}
  2 \Big[ \norm{p^n}^2_{\varepsilon^{-1}} - \norm{p^{n - 1}}^2_{\varepsilon^{-1}} + \norm{E^n}^2_{\varepsilon} - \norm{E^{n - 1}}^2_{\varepsilon} + \norm{H^n}^2_{\mu} - \norm{H^{n - 1}}^2_{\mu} \Big] \\
  \le \Delta t \Big[ \norm{f_p^n}^2_{\varepsilon^{-1}} + \norm{f_p^{n - 1}}^2_{\varepsilon^{-1}} + \norm{p^n}^2_{\varepsilon^{-1}} + \norm{p^{n - 1}}^2_{\varepsilon^{-1}} + \norm{f_E^n}^2_{\varepsilon^{-1}} + \norm{f_E^{n - 1}}^2_{\varepsilon^{-1}} + \\
  \norm{E^n}^2_{\varepsilon} + \norm{E^{n - 1}}^2_{\varepsilon} +
\norm{f_H^n}^2_{\mu^{-1}} + \norm{f_H^{n - 1}}^2_{\mu^{-1}} + \norm{H^n}^2_{\mu} + \norm{H^{n - 1}}^2_{\mu} \Big].
\end{multline*}
Next, summing over $n = 1$ to $N$ and using the initial conditions as in Equation~\eqref{eqn:ICs}, we have that:
\begin{multline*}
  2 \Big[ \norm{p^N}^2_{\varepsilon^{-1}} + \norm{E^N}^2_{\varepsilon} + \norm{H^N}^2_{\mu} \Big] \le \Delta t \Big[ \sum\limits_{n = 0}^{N} 2 \left( \norm{p^n}^2_{\varepsilon^{-1}} + \norm{E^n}^2_{\varepsilon} + \norm{H^n}^2_{\mu} \right. + \\
  \left. \norm{f_p^n}^2_{\varepsilon^{-1}} + \norm{f_E^n}^2_{\varepsilon^{-1}} + \norm{f_H^n}^2_{\mu^{-1}} \right) \Big] + 2 \Big[ \norm{p_0}^2_{\varepsilon^{-1}} + \norm{E_0}^2_{\varepsilon} + \norm{H_0}^2_{\mu} \Big],
\end{multline*}
which in turn provides us with the following estimate:
\begin{multline*}
  \norm{p^N}^2_{\varepsilon^{-1}} + \norm{E^N}^2_{\varepsilon} + \norm{H^N}^2_{\mu} \le \\
  \Delta t \sum\limits_{n = 0}^{N} \Big[ \norm{p^n}^2_{\varepsilon^{-1}} +\norm{E^n}^2_{\varepsilon} + \norm{H^n}^2_{\mu} \Big] + \Delta t \sum\limits_{n = 0}^{N} \Big[ \norm{f_p^n}^2_{\varepsilon^{-1}} + \norm{f_E^n}^2_{\varepsilon^{-1}} + \norm{f_H^n}^2_{\mu^{-1}}\Big] + \Big[ \norm{p_0}^2_{\varepsilon^{-1}} + \norm{E_0}^2_{\varepsilon} + \norm{H_0}^2_{\mu} \Big].
\end{multline*}
Now, we are ready to apply the discrete Gronwall inequality as in Lemma~\ref{lemma:gronwall_dscrt}. To do so, we set $\displaystyle \delta \coloneq \Delta t$, $\displaystyle g_0 \coloneq \norm{p_0}^2_{\varepsilon^{-1}} + \norm{E_0}^2_{\varepsilon} + \norm{H_0}^2_{\mu}$, $\displaystyle a_n \coloneq \norm{p^n}^2_{\varepsilon^{-1}} + \norm{E^n}^2_{\varepsilon} + \norm{H^n}^2_{\mu}$, $\displaystyle b_n \coloneq 0$, $c_n \coloneq \norm{f_p^n}^2_{\varepsilon^{-1}} + \norm{f_E^n}^2_{\varepsilon^{-1}} + \norm{f_H^n}^2_{\mu^{-1}}$, and $\displaystyle \gamma_n \coloneq 1$. Note that for the condition $\gamma_n \delta < 1$ to hold, we need to have $\Delta t < 1/2$ and therefore we have that $\sigma_n = \left(1 - \Delta t \right)^{-1}$. Then, we get:
\begin{multline*}
  \norm{p^N}^2_{\varepsilon^{-1}} + \norm{E^N}^2_{\varepsilon} + \norm{H^N}^2_{\mu} \le \\
  \shoveright{\Bigg[ \Delta t \sum\limits_{n = 0}^{N} \left( \norm{f_p^n}^2_{\varepsilon^{-1}} + \norm{f_E^n}^2_{\varepsilon^{-1}} + \norm{f_H^n}^2_{\mu^{-1}} \right) + \norm{p_0}^2_{\varepsilon^{-1}} + \norm{E_0}^2_{\varepsilon} + \norm{H_0}^2_{\mu}\Bigg] \exp \left[ \Delta t \sum\limits_{n = 0}^{N} \left( 1 - \Delta t \right)^{-1} \right],} \\
  \le \Big[ \Delta t \sum\limits_{n = 0}^{N} \left( \norm{f_p^n}^2_{\varepsilon^{-1}} + \norm{f_E^n}^2_{\varepsilon^{-1}} + \norm{f_H^n}^2_{\mu^{-1}}\right) + \norm{p_0}^2_{\varepsilon^{-1}} + \norm{E_0}^2_{\varepsilon} + \norm{H_0}^2_{\mu}\Big] \exp\left( 2 T + 1 \right),
\end{multline*}
where the second inequality is obtained using that $\displaystyle (N + 1) \Delta t = T + \Delta t < T + 1/2$. We now estimate the term:
\[
  \sum \limits_{n = 0}^{N} \Delta t \left( \norm{f_p^n}^2_{\varepsilon^{-1}} + \norm{f_E^n}^2_{\varepsilon^{-1}} + \norm{f_H^n}^2_{\mu^{-1}} \right),
\]
which is a lower sum approximation of the integral:
\[
  \int\limits_0^T \left(\norm{f_p}^2_{\varepsilon^{-1}} + \norm{f_E}^2_{\varepsilon^{-1}} + \norm{f_H}^2_{\mu^{-1}}\right) dt,
\]
and which in turn is equal to:
\[
  \norm{f_p}^2_{L^2[0, T] \times L^2_{\varepsilon^{-1}}(\Omega)} + \norm{f_E}^2_{L^2[0, T] \times L^2_{\varepsilon^{-1}}(\Omega)} + \norm{f_H}^2_{L^2[0, T] \times L^2_{\mu^{-1}}(\Omega)},
\]
and hence is bounded. Set this bound to be a positive constant $M$. By using the equivalence of $1$- and $2$-norms, and $C = \sqrt{ 3 \big( M + \norm{p_0}^2_{\varepsilon^{-1}} + \norm{E_0}^2_{\varepsilon} + \norm{H_0}^2_{\mu} \big) \exp(2 T + 1)}$, we obtain our desired result:
\[
  \norm{p^N}_{\varepsilon^{-1}} + \norm{E^N}_{\varepsilon} + \norm{H^N}_{\mu} \le C. \qedhere
\]
\end{proof}

\begin{corollary}[Discrete Energy Conservation]\label{corr:dscrt_enrgy_cnsrvtn_cn}
If the forcing functions in Equation~\eqref{eqn:maxwells_eqns} are all zero, that is, $f_p = 0$ and $f_E = f_H = 0$ then:
\[
  \norm{p^N}^2_{\varepsilon^{-1}} + \norm{E^N}^2_{\varepsilon} + \norm{H^N}^2_{\mu} = \norm{p_0}^2_{\varepsilon^{-1}} + \norm{E_0}^2_{\varepsilon} + \norm{H_0}^2_{\mu}.
\]
\end{corollary}

\begin{theorem}[Discrete Error Estimate]\label{thm:dscrt_error_estmt_cn}
For the semidiscretization using the Crank-Nicholson scheme as given in Equations~\labelcref{eqn:maxwell_p_cn,eqn:maxwell_E_cn,eqn:maxwell_H_cn}, for the solution $(p, E, H)$ of Equations~\labelcref{eqn:maxwell_p_wf,eqn:maxwell_E_wf,eqn:maxwell_H_wf} with initial conditions as in Equation~\eqref{eqn:ICs} and assuming sufficient regularity with $p \in C^3[0, T] \times \mathring{H}^1_{\varepsilon^{-1}}(\Omega), \; E \in C^3[0, T] \times \mathring{H}_{\varepsilon}(\curl; \Omega)$, and $H \in C^3[0, T] \times \mathring{H}_{\mu}(\divgn; \Omega)$, and for the time step $\Delta t$ sufficiently small, there exists a positive bounded constant $C$ independent of $\Delta t$ such that:
\begin{equation}
\norm{e_p^N}_{\varepsilon^{-1}} + \norm{e_E^N}_{\varepsilon} + \norm{e_H^N}_{\mu} \le C \left[(\Delta t)^2 + \norm{e_p^0}_{\varepsilon^{-1}} + \norm{e_E^0}_{\varepsilon} + \norm{e_H^0}_{\mu} \right],
\label{eqn:errorestimate_cn}
\end{equation}
where $e_p^n \coloneq p(t^n) - p^n, \; e_E^n \coloneq E(t^n) - E^n$ and $e_H^n \coloneq H(t^n) - H^n$ are the errors in the time semidiscretization of $p$, $E$ and $H$, respectively and at the indicated time indices.
\end{theorem}

\begin{proof}
Using the Taylor remainder theorem, and expressing $p(t)$ about $t = t^{n - 1/2}$, we have that:
\[
  p(t) = p(t^{n - \frac{1}{2}}) + \dfrac{\partial p}{\partial t}(t^{n - \frac{1}{2}}) (t - t^{n - \frac{1}{2}}) + \dfrac{\partial^2 p}{\partial t^2}(t^{n - \frac{1}{2}}) \dfrac{(t - t^{n - \frac{1}{2}})^2}{2} + \int\limits_{\mathclap{t^{n - \frac{1}{2}}}}^{t} \dfrac{(t - s)^2}{2} \dfrac{\partial^3 p}{\partial t^3}(s) ds,
\]
which when evaluated at $t = t^n$ and $t = t^{n - 1}$ provides us with:
\begin{align*}
    p(t^n) &= p(t^{n - \frac{1}{2}}) + \dfrac{\partial p}{\partial t}(t^{n - \frac{1}{2}}) (t^n - t^{n - \frac{1}{2}}) + \dfrac{\partial^2 p}{\partial t^2}(t^{n - \frac{1}{2}}) \dfrac{(t^n - t^{n - \frac{1}{2}})^2}{2} + \int\limits_{\mathclap{t^{n - \frac{1}{2}}}}^{t^n} \dfrac{(t^n - s)^2}{2} \dfrac{\partial^3 p}{\partial t^3}(s) ds, \\
  p(t^{n - 1}) &= p(t^{n - \frac{1}{2}}) + \dfrac{\partial p}{\partial t}(t^{n - \frac{1}{2}}) (t^{n - 1} - t^{n - \frac{1}{2}}) + \dfrac{\partial^2 p}{\partial t^2}(t^{n - \frac{1}{2}}) \dfrac{(t^{n - 1} - t^{n - \frac{1}{2}})^2}{2} + \int\limits_{\mathclap{t^{n - \frac{1}{2}}}}^{\mathclap{t^{n - 1}}} \dfrac{(t^{n - 1} - s)^2}{2} \dfrac{\partial^3 p}{\partial t^3}(s) ds.
\end{align*}
Subtracting the second of these equations from the first, we get:
\[
  \left( p(t^n) - p(t^{n - 1}) \right) = \dfrac{\partial p}{\partial t}(t^{n - \frac{1}{2}}) (t^n - t^{n - 1}) + \int\limits_{\mathclap{t^{n - \frac{1}{2}}}}^{t^n} \dfrac{(t^n - s)^2}{2} \dfrac{\partial^3 p}{\partial t^3}(s) ds - \int\limits_{\mathclap{t^{n - \frac{1}{2}}}}^{\mathclap{t^{n - 1}}} \dfrac{(t^{n - 1} - s)^2}{2} \dfrac{\partial^3 p}{\partial t^3}(s) ds,
\]
which then when cast into the form of the inner product term from the semidiscretization of the variational formulation leads to:
\[
  \ainnerproduct{\dfrac{p(t^n) - p(t^{n - 1})}{\Delta t}}{\widetilde{p}} = \ainnerproduct{\dfrac{\partial p}{\partial t}(t^{n - \frac{1}{2}})}{\widetilde{p}} + \dfrac{1}{2}\ainnerproduct{R^n_p}{\widetilde{p}},
\]
and in which we have defined that:
\[
  R^n_p \coloneq \dfrac{1}{\Delta t} \left[ \int\limits_{t^{n - 1}}^{t^{n - \frac{1}{2}}} (t^{n - 1} - s)^2 \dfrac{\partial^3 p}{\partial t^3}(s) ds + \int\limits_{\mathclap{t^{n - \frac{1}{2}}}}^{t^n} (t^n - s)^2 \dfrac{\partial^3 p}{\partial t^3}(s) ds \right].
\]
Similarly, for $E$ and $H$, we have the following inner product terms from the variational formulation:
\begin{align*}
  \ainnerproduct{\varepsilon \dfrac{E(t^n) - E(t^{n - 1})}{\Delta t}}{\widetilde{E}} &= \ainnerproduct{\varepsilon \dfrac{\partial E}{\partial t}(t^{n - \frac{1}{2}})}{\widetilde{E}} + \dfrac{1}{2} \ainnerproduct{\varepsilon R^n_E}{\widetilde{E}}, \\
  \ainnerproduct{\mu \dfrac{H(t^n) - H(t^{n - 1})}{\Delta t}}{\widetilde{H}} &= \ainnerproduct{\mu \dfrac{\partial H}{\partial t}(t^{n - \frac{1}{2}})}{\widetilde{H}} + \dfrac{1}{2} \ainnerproduct{\mu R^n_H}{\widetilde{H}},
\end{align*}
and in each of which we have defined that:
\begin{align*}
  R^n_E &\coloneq \dfrac{1}{\Delta t} \left[\int\limits_{t^{n - 1}}^{t^{n - \frac{1}{2}}} (t^{n - 1} - s)^2 \dfrac{\partial^3 E}{\partial t^3}(s) ds + \int\limits_{\mathclap{t^{n - \frac{1}{2}}}}^{t^n} (t^n - s)^2 \dfrac{\partial^3 E}{\partial t^3}(s) ds \right], \\
  R^n_H &\coloneq \dfrac{1}{\Delta t} \left[\int\limits_{t^{n - 1}}^{t^{n - \frac{1}{2}}} (t^{n - 1} - s)^2 \dfrac{\partial^3 H}{\partial t^3}(s) ds + \int\limits_{\mathclap{t^{n - \frac{1}{2}}}}^{t^n} (t^n - s)^2 \dfrac{\partial^3 H}{\partial t^3}(s) ds \right].
\end{align*}
Using these terms in Equations~\labelcref{eqn:maxwell_p_wf,eqn:maxwell_E_wf,eqn:maxwell_H_wf} at time $t = t^{n - \frac{1}{2}}$ and with the definition that:
\[
  f(t^{n - \frac{1}{2}}) \coloneq \dfrac{f(t^n) + f(t^{n - 1})}{2},
\] 
where $f$ is any of $f_p$, $f_E$ or $f_H$, and wielding the Taylor remainder theorem to also write:
\[
  u(t^{n - \frac{1}{2}}) \coloneq \dfrac{u(t^n) + u(t^{n - 1})}{2} - r_u^n,
\] 
for $u$ being $p$, $E$ or $H$, and with the definition that:
\[
r_u^n \coloneqq \left[ \int\limits_{t^{n - \frac{1}{2}}}^{t^n} (t^n - s) \dfrac{\partial^2 u}{\partial t^2}(s) ds - \int\limits_{\mathclap{t^{n - 1}}}^{\mathclap{t^{n - \frac{1}{2}}}} (t^{n - 1} - s) \dfrac{\partial^2 u}{\partial t^2}(s) ds\right],
\]
leads to the following set of equations:
\begin{subequations}
  \begin{equation}
    \ainnerproduct{\dfrac{p(t^n) - p(t^{n - 1})}{\Delta t}}{\widetilde{p}} - \dfrac{1}{2} \aInnerproduct{ \varepsilon \left(E(t^n) + E(t^{n - 1})\right)}{\nabla \widetilde{p}} = 
    \dfrac{1}{2} \aInnerproduct{f_p^n + f_p^{n - 1} + R_p^n - \varepsilon r_E^n }{\widetilde{p}}, \label{eqn:remainder_p_cn}
  \end{equation} 
  \begin{multline}
    \dfrac{1}{2} \aInnerproduct{\nabla \left( p(t^n) + p(t^{n - 1}) \right)}{\widetilde{E}} + \ainnerproduct{\varepsilon \dfrac{E(t^n) - E(t^{n - 1})}{\Delta t}}{\widetilde{E}} - \dfrac{1}{2} \aInnerproduct{H(t^n) + H(t^{n - 1})}{\nabla \times \widetilde{E}} = \\
    \dfrac{1}{2} \aInnerproduct{f_E^n + f_E^{n - 1} + \varepsilon R_E^n + \nabla r_p^n - r_H^n}{\widetilde{E}}, \label{eqn:remainder_E_cn}
  \end{multline}
  \begin{equation}
    \ainnerproduct{\mu \dfrac{H(t^n) - H(t^{n - 1})}{\Delta t}}{\widetilde{H}} + \dfrac{1}{2} \aInnerproduct{\nabla \times \left(E(t^n) + E(t^{n - 1}) \right)}{\widetilde{H}} = \dfrac{1}{2} \aInnerproduct{f_H^n + f_H^{n - 1} + \mu R_H^n + \nabla \times r_E^n}{\widetilde{H}}. \label{eqn:remainder_H_cn}
  \end{equation}
\end{subequations}
Subtracting the equations defining the Crank-Nicholson semidiscretization as in Equations~\labelcref{eqn:maxwell_p_cn,eqn:maxwell_E_cn,eqn:maxwell_H_cn} from Equations~\labelcref{eqn:remainder_p_cn,eqn:remainder_E_cn,eqn:remainder_H_cn} leads us to:
\begin{align*}
\aInnerproduct{\dfrac{e_p^n - e_p^{n - 1}}{\Delta t}}{\widetilde{p}} - \dfrac{1}{2}\aInnerproduct{ \varepsilon \left(e_E^n + e_E^{n - 1} \right)}{\nabla \widetilde{p}} &= \dfrac{1}{2} \aInnerproduct{ R_p^n - \varepsilon r_E^n}{\widetilde{p}}, \\
\dfrac{1}{2} \aInnerproduct{\nabla \left( e_p^n +  e_p^{n - 1} \right) }{\widetilde{E}} + \aInnerproduct{\varepsilon \dfrac{e_E^n - e_E^{n - 1}}{\Delta t}}{\widetilde{E}} - \dfrac{1}{2}\aInnerproduct{e_H^n + e_H^{n - 1}}{\nabla \times \widetilde{E}} &= \dfrac{1}{2} \aInnerproduct{\varepsilon R_E^n + \nabla r_p^n - r_H^n}{\widetilde{E}}, \\
\aInnerproduct{\mu \dfrac{e_H^n - e_H^{n - 1}}{\Delta t}}{\widetilde{H}} + \dfrac{1}{2} \aInnerproduct{\nabla \times \left( e_E^n + e_E^{n - 1} \right) }{\widetilde{H}} &= \dfrac{1}{2} \aInnerproduct{\mu  R_H^n + \nabla \times r_E^n}{\widetilde{H}}.
\end{align*}
Now, in this set of weak formulation equations for the errors, we choose the test functions to be $\widetilde{p} = 2 \Delta t \varepsilon^{-1} \left( e_p^n + e_p^{n - 1} \right)$, $\widetilde{E} = 2 \Delta t \left( e_E^n + e_E^{n - 1} \right)$ and $\widetilde{H} = 2 \Delta t \left( e_H^n + e_H^{n - 1} \right)$. Next, by following essentially the same sequence of steps as in Theorem~\ref{thm:dscrt_enrgy_estmt_cn}, we obtain the estimate for these error terms to be:
\begin{multline*}
  2 \left[ \norm{e_p^n}^2_{\varepsilon^{-1}} - \norm{e_p^{n - 1}}^2_{\varepsilon^{-1}} + \norm{e_E^n}^2_{\varepsilon} - \norm{e_E^{n - 1}}^2_{\varepsilon} + \norm{e_H^n}^2_{\mu} - \norm{e_H^{n - 1}}^2_{\mu} \right] \le \\ \Delta t \left[ \norm{e_p^n}^2_{\varepsilon^{-1}} + \norm{e_p^{n - 1}}^2_{\varepsilon^{-1}} +  \norm{e_E^n}^2_{\varepsilon} \right. + \left. \norm{e_E^{n - 1}}^2_{\varepsilon} + \norm{e_H^n}^2_{\mu} + \norm{e_H^{n - 1}}^2_{\mu} \right] + \\
  \Delta t \left[\norm{R_p^n} ^2_{\varepsilon^{-1}} + \norm{R_E^n}^2_{\varepsilon} +  \norm{R_H^n}^2_{\mu} + \norm{\nabla r_p^n} ^2_{\varepsilon^{-1}} + \norm{r_E^n}^2_{\varepsilon} + \varepsilon^{-1} \mu^{-1} \norm{\nabla \times r_E^n}^2_{\varepsilon} +  \varepsilon^{-1} \mu^{-1} \norm{r_H^n}^2_{\mu} \right].
\end{multline*}
Now, summing over $n = 1$ to $N$, using the initial conditions as in Equation~\eqref{eqn:ICs}, and the positivity of all the right hand side terms, we have that:
\begin{multline*}
  \norm{e_p^N}^2_{\varepsilon^{-1}} + \norm{e_E^N}^2_{\varepsilon} + \norm{e_H^N}^2_{\mu} \le \Delta t \sum\limits_{n = 0}^{N} \left[ \norm{e_p^n}^2_{\varepsilon^{-1}} +\norm{e_E^n}^2_{\varepsilon} + \norm{e_H^n}^2_{\mu} \right] + \Delta t \sum\limits_{n = 0}^{N} \left[\norm{R_p^n}^2_{\varepsilon^{-1}} + \norm{R_E^n}^2_{\varepsilon} + \norm{R_H^n}^2_{\mu} \right. + \\ \left. \norm{\nabla r_p^n} ^2_{\varepsilon^{-1}} + \norm{r_E^n}^2_{\varepsilon} + \varepsilon^{-1} \mu^{-1} \norm{\nabla \times r_E^n}^2_{\varepsilon} +  \varepsilon^{-1} \mu^{-1} \norm{r_H^n}^2_{\mu}\right] + \left[ \norm{e_p^0}^2_{\varepsilon^{-1}} + \norm{e_E^0}^2_{\varepsilon} + \norm{e_H^0}^2_{\mu} \right].
\end{multline*}
Next to apply the discrete Gronwall inequality as in Lemma~\ref{lemma:gronwall_dscrt}, we choose $\delta \coloneq \Delta t$, $g_0 \coloneq \norm{e_p^0}^2_{\varepsilon^{-1}} + \norm{e_E^0}^2_{\varepsilon} + \norm{e_H^0}^2_{\mu}$, $a_n \coloneq \norm{e_p^n}^2_{\varepsilon^{-1}} + \norm{e_E^n}^2_{\varepsilon} + \norm{e_H^n}^2_{\mu}$, $b_n \coloneq 0$, $c_n \coloneq \norm{R_p^n}^2_{\varepsilon^{-1}} + \norm{R_E^n}^2_{\varepsilon} + \norm{R_H^n}^2_{\mu} + \norm{\nabla r_p^n} ^2_{\varepsilon^{-1}} + \norm{r_E^n}^2_{\varepsilon} + \varepsilon^{-1} \mu^{-1} \norm{\nabla \times r_E^n}^2_{\varepsilon} +  \varepsilon^{-1} \mu^{-1} \norm{r_H^n}^2_{\mu}$, and $\gamma_n \coloneq 1$. Now the required condition $\gamma_n \delta < 1$ holds with $\Delta t < 1/2$ and $\sigma_n = \left( 1 - \Delta t \right)^{-1}$. Therefore, using that $(N + 1) \Delta t = T + \Delta t \le T + 1/2$, we get that:
\begin{multline}
  \norm{e_p^N}^2_{\varepsilon^{-1}} + \norm{e_E^N}^2_{\varepsilon} + \norm{e_H^N}^2_{\mu} \le \Big[ \Delta t \sum\limits_{n = 0}^{N} \left( \norm{R_p^n}^2_{\varepsilon^{-1}} + \norm{R_E^n}^2_{\varepsilon} + \norm{R_H^n}^2_{\mu} + \norm{\nabla r_p^n} ^2_{\varepsilon^{-1}} + \norm{r_E^n}^2_{\varepsilon} + \right. \\ \left. \varepsilon^{-1} \mu^{-1} \norm{\nabla \times r_E^n}^2_{\varepsilon} +  \varepsilon^{-1} \mu^{-1} \norm{r_H^n}^2_{\mu} \right) + 
  \left( \norm{e_p^0}^2_{\varepsilon^{-1}} + \norm{e_E^0}^2_{\varepsilon} + \norm{e_H^0}^2_{\mu} \right) \Big] \exp\left( 2 T + 1 \right). \label{eqn:error_Gronwall_cn}
\end{multline}
Now, we need to obtain bounding estimates for each of the Taylor remainder terms and to do so, we first consider the first remainder term corresponding to $p$ and argue as follows:
\begin{align*}
\norm{R^n_p}^2_{\varepsilon^{-1}} &= \dfrac{1}{\left( \Delta t \right)^2} \norm[\bigg]{\int\limits_{\mathclap{t^{n - 1}}}^{\mathclap{t^{n - \frac{1}{2}}}} (t^{n - 1} - s)^2 \dfrac{\partial^3 p}{\partial t^3}(s) ds + \int\limits_{\mathclap{t^{n - \frac{1}{2}}}}^{t^n} (t^n - s)^2 \dfrac{\partial^3 p}{\partial t^3}(s) ds}^2_{\varepsilon^{-1}}, \\
&\le \dfrac{1}{\left( \Delta t \right)^2} \norm[\bigg]{\int\limits_{t^{n - 1}}^{t^n} (t^n - s)^2 \dfrac{\partial^3 p}{\partial t^3}(s) ds}^2_{\varepsilon^{-1}}, \quad \text{(using $t^{n - 1} < t^n$)} \\
&\le \dfrac{1}{\left(\Delta t\right)^2} \int\limits_{t^{n - 1}}^{t^n} (s - t^n)^4 ds \int\limits_{\mathclap{t^{n - 1}}}^{t^n} \norm[\bigg]{\dfrac{\partial^3 p}{\partial t^3}(s)}^2_{\varepsilon^{-1}} ds, \quad \text{(by Cauchy-Schwarz)} \\
&= \dfrac{\left( \Delta t \right)^3}{5} \int\limits_{\mathclap{t^{n - 1}}}^{t^n} \norm[\bigg]{\dfrac{\partial^3 p}{\partial t^3}(s)}^2_{\varepsilon^{-1}} ds,
\end{align*}
and now summing both sides over $n = 0$ to $N$, we have that:
\begin{equation}
  \sum\limits_{n = 0}^N \norm{R^n_p}^2_{\varepsilon^{-1}} \le \dfrac{\left( \Delta t \right)^3}{5} \int\limits_0^T \norm[\bigg]{\dfrac{\partial^3 p}{\partial t^3}(s)}^2_{\varepsilon^{-1}} ds = \dfrac{\left( \Delta t \right)^3}{5} \norm[\bigg]{\dfrac{\partial^3 p}{\partial t^3}}^2_{L^2[0, T] \times L^2_{\varepsilon^{-1}}(\Omega)}. \label{eqn:Remainder_norm_p_cn}
\end{equation}
Similarly, for the remaining Taylor remainder terms, we have that:
\begin{alignat}{3}
\sum\limits_{n = 0}^N \norm{R^n_E}^2_{\varepsilon} &\le \dfrac{\left( \Delta t \right)^3}{5} \norm[\bigg]{\dfrac{\partial^3 E}{\partial t^3}}^2_{L^2[0, T] \times L^2_\varepsilon(\Omega)}, & \sum\limits_{n = 0}^N \norm{R^n_H}^2_\mu &\le \dfrac{\left( \Delta t \right)^3}{5} \norm[\bigg]{\dfrac{\partial^3 H}{\partial t^3}}^2_{L^2[0, T] \times L^2_\mu(\Omega)}, \nonumber \\
\sum\limits_{n = 0}^N \norm{\nabla r^n_p}^2_{\varepsilon^{-1}} &\le  \dfrac{\left( \Delta t \right)^3}{12} \norm[\bigg]{\dfrac{\partial^2 \left(\nabla p \right)}{\partial t^2}}^2_{L^2[0, T] \times \mathring{H}^1_{\varepsilon^{-1}}(\Omega)}, & \sum\limits_{n = 0}^N \norm{r^n_E}^2_{\varepsilon} &\le \dfrac{\left( \Delta t \right)^3}{12} \norm[\bigg]{\dfrac{\partial^2 E}{\partial t^2}}^2_{L^2[0, T] \times L^2_\varepsilon(\Omega)}, \label{eqn:Remainder_norms_pEH_cn} \\
\sum\limits_{n = 0}^N \norm{\nabla \times r^n_E}^2_{\varepsilon} &\le \dfrac{\left( \Delta t \right)^3}{12} \norm[\bigg]{\dfrac{\partial^2 (\nabla \times E)}{\partial t^2}}^2_{L^2[0, T] \times \mathring{H}_\varepsilon(\curl; \Omega)}, & \sum\limits_{n = 0}^N \norm{r^n_H}^2_\mu &\le \dfrac{\left( \Delta t \right)^3}{12} \norm[\bigg]{\dfrac{\partial^2 H}{\partial t^2}}^2_{L^2[0, T] \times L^2_\mu(\Omega)} \nonumber.
\end{alignat}
By substituting all these Taylor remainder term estimates into Equation~\eqref{eqn:error_Gronwall_cn}, we obtain:
\begin{multline*}
\norm{e_p^N}^2_{\varepsilon^{-1}} + \norm{e_E^N}^2_{\varepsilon} + \norm{e_H^N}^2_{\mu} \le \Bigg[ \dfrac{\left( \Delta t \right)^4}{5} \left( \norm[\bigg]{\dfrac{\partial^3 p}{\partial t^3}}^2_{L^2[0, T] \times L^2_{\varepsilon^{-1}}(\Omega)} \! + \norm[\bigg]{\dfrac{\partial^3 E}{\partial t^3}}^2_{L^2[0, T] \times L^2_\varepsilon(\Omega)} \! + \norm[\bigg]{\dfrac{\partial^3 H}{\partial t^3}}^2_{L^2[0, T] \times L^2_\mu(\Omega)} \right) + \\
\dfrac{\left( \Delta t \right)^4}{12} \left( \norm[\bigg]{\dfrac{\partial^2 \left( \nabla p \right)}{\partial t^2}}^2_{L^2[0, T] \times \mathring{H}^1_{\varepsilon^{-1}}(\Omega)} \! + \norm[\bigg]{\dfrac{\partial^2 E}{\partial t^2}}^2_{L^2[0, T] \times L^2_\varepsilon(\Omega)} \! + \varepsilon^{-1} \mu^{-1} \norm[\bigg]{\dfrac{\partial^2 (\nabla \times E)}{\partial t^2}}^2_{L^2[0, T] \times \mathring{H}_\varepsilon(\curl; \Omega)} + \right. \\
\left. \varepsilon^{-1} \mu^{-1} \norm[\bigg]{\dfrac{\partial^2 H}{\partial t^2}}^2_{L^2[0, T] \times L^2_\mu(\Omega)} \right) + \left( \norm{e_p^0}^2_{\varepsilon^{-1}} + \norm{e_E^0}^2_{\varepsilon} + \norm{e_H^0}^2_{\mu} \right) \Bigg] \exp \left( 2 T + 1 \right).
\end{multline*}
By our assumption on the regularity of $p$, $E$ and $H$, there exists a positive bounded constant $M$ such that:
\begin{multline*}
\dfrac{1}{5} \left( \norm[\bigg]{\dfrac{\partial^3 p}{\partial t^3}}^2_{L^2[0, T] \times L^2_{\varepsilon^{-1}}(\Omega)} \! + \norm[\bigg]{\dfrac{\partial^3 E}{\partial t^3}}^2_{L^2[0, T] \times L^2_\varepsilon(\Omega)} \! + \norm[\bigg]{\dfrac{\partial^3 H}{\partial t^3}}^2_{L^2[0, T] \times L^2_\mu(\Omega)} \right) + \dfrac{1}{12} \left( \norm[\bigg]{\dfrac{\partial^2 \left(\nabla p \right)}{\partial t^2}}^2_{L^2[0, T] \times \mathring{H}^1_{\varepsilon^{-1}}(\Omega)} + \right. \\
\left. \norm[\bigg]{\dfrac{\partial^2 E}{\partial t^2}}^2_{L^2[0, T] \times L^2_\varepsilon(\Omega)} \! + \varepsilon^{-1} \mu^{-1} \norm[\bigg]{\dfrac{\partial^2 (\nabla \times E)}{\partial t^2}}^2_{L^2[0, T] \times \mathring{H}_\varepsilon(\curl; \Omega)} \! + \varepsilon^{-1} \mu^{-1} \norm[\bigg]{\dfrac{\partial^2 H}{\partial t^2}}^2_{L^2[0, T] \times L^2_\mu(\Omega)} \right) \le M.
\end{multline*}
Thus, we have that:
\[
  \norm{e_p^N}^2_{\varepsilon^{-1}} + \norm{e_E^N}^2_{\varepsilon} + \norm{e_H^N}^2_{\mu} \le \widetilde{C} \left[ (\Delta t)^4 + \norm{e_p^0}^2_{\varepsilon^{-1}} + \norm{e_E^0}^2_{\varepsilon} + \norm{e_H^0}^2_{\mu} \right],
\]
where $\widetilde{C} = \exp(2 T +1) \max\{M, 1\}$. Finally using the equivalence of $1$- and $2$-norms, our desired result follows with $C = \sqrt{3 \widetilde{C}}$.
\end{proof}

\subsection{Implicit Leapfrog Scheme}

\begin{theorem}[Discrete Energy Estimate] \label{thm:dscrt_enrgy_estmt_lf}
  For the semidiscretization using the implicit leapfrog scheme as given in Equations~\labelcref{eqn:maxwell_p_lf,eqn:maxwell_E_lf,eqn:maxwell_H_lf,eqn:maxwell_p0_lf,eqn:maxwell_E0_lf,eqn:maxwell_H0_lf}, for any time step $\Delta t$ sufficiently small, there exists a positive bounded constant $C$ independent of $\Delta t$ such that:
\[
  \norm{p^{N - \frac{1}{2}}}_{\varepsilon^{-1}} + \norm{E^{N - \frac{1}{2}}}_{\varepsilon} + \norm{H^N}_{\mu} \le C.
\]
\end{theorem}

\begin{proof}
  Since Equations~\labelcref{eqn:maxwell_p_lf,eqn:maxwell_E_lf,eqn:maxwell_H_lf} are true for all $\widetilde{p}\in \mathring{H}^1_{\varepsilon^{-1}}(\Omega)$, $\widetilde{E} \in \mathring{H}_{\varepsilon}(\curl; \Omega)$, $\widetilde{H} \in \mathring{H}_{\mu}(\divgn; \Omega)$, using $\widetilde{p} = 2 \Delta t \varepsilon^{-1} \left( p^{n + \frac{1}{2}} + p^{n - \frac{1}{2}} \right)$, $\widetilde{E} = 2 \Delta t \left(E^{n + \frac{1}{2}} + E^{n - \frac{1}{2}} \right)$ and $\widetilde{H} = 2 \Delta t \left(H^{n + 1} + H^n \right)$ in them, we obtain the following:
  \begin{multline*}
    2 \aInnerproduct{p^{n+\frac{1}{2}} - p^{n - \frac{1}{2}}}{\varepsilon^{-1} \left( p^{n + \frac{1}{2}} + p^{n-\frac{1}{2}} \right)} - \Delta t \aInnerproduct{E^{n + \frac{1}{2}} + E^{n - \frac{1}{2}}}{\nabla \left( p^{n + \frac{1}{2}} + p^{n - \frac{1}{2}} \right)} = 2 \Delta t \aInnerproduct{f_p^n}{\varepsilon^{-1} \left(p^{n + \frac{1}{2}} + p^{n - \frac{1}{2}}\right)},
\end{multline*}
\vspace{-1em} \begin{multline*}
  \Delta t \aInnerproduct{\nabla \left( p^{n + \frac{1}{2}} + p^{n - \frac{1}{2}} \right)}{E^{n + \frac{1}{2}} + E^{n - \frac{1}{2}}} + 2 \aInnerproduct{\varepsilon \left(E^{n + \frac{1}{2}} - E^{n - \frac{1}{2}} \right)}{E^{n + \frac{1}{2}} + E^{n - \frac{1}{2}}} \, - \\
  \Delta t \aInnerproduct{\left( H^{n + 1} + H^n \right)}{\nabla \times \left( E^{n + \frac{1}{2}} + E^{n - \frac{1}{2}} \right)} = 2 \Delta t \aInnerproduct{f_E^n}{E^{n + \frac{1}{2}} + E^{n - \frac{1}{2}}},
\end{multline*}
\vspace{-1em} \begin{multline*}
  2 \aInnerproduct{\mu \left(H^{n + 1} - H^{n} \right)}{H^{n + 1} + H^n} + \Delta t \aInnerproduct{\nabla \times \left( E^{n + \frac{1}{2}} + E^{n - \frac{1}{2}} \right)}{H^{n + 1} + H^n} = 2 \Delta t \aInnerproduct{f_H^{n + \frac{1}{2}}}{H^{n + 1} + H^n}.
\end{multline*}
Adding these equations together and using properties of the inner product, we get:
\begin{multline*}
  2 \aInnerproduct{\varepsilon^{-1} \left( p^{n + \frac{1}{2}} - p^{n - \frac{1}{2}} \right)}{p^{n + \frac{1}{2}} + p^{n - \frac{1}{2}}} + 2 \aInnerproduct{\varepsilon \left(E^{n + \frac{1}{2}} - E^{n - \frac{1}{2}} \right)}{E^{n + \frac{1}{2}} + E^{n - \frac{1}{2}}} + 2 \aInnerproduct{\mu \left( H^{n + 1} - H^{n} \right)}{H^{n + 1} + H^n} = \\
\Delta t \left( 2 \aInnerproduct{f_p^n}{\varepsilon^{-1} \left(p^{n + \frac{1}{2}} + p^{n - \frac{1}{2}} \right)} + 2 \aInnerproduct{f_E^n}{E^{n + \frac{1}{2}} + E^{n - \frac{1}{2}}} + 2 \aInnerproduct{f_H^{n + \frac{1}{2}}}{H^{n + 1} + H^n} \right).
\end{multline*}
Using the same arguments as in Theorem~\ref{thm:dscrt_enrgy_estmt_cn}, we have the following estimates for the right hand side terms of this equation:
\begin{align*}
2 \aInnerproduct{f_p^n}{\varepsilon^{-1} \left(p^{n + \frac{1}{2}} + p^{n - \frac{1}{2}} \right)} &\le \norm{f_p^n}^2_{\varepsilon^{-1}} + 2 \left( \norm{p^{n + \frac{1}{2}}}^2_{\varepsilon^{-1}} + \norm{p^{n - \frac{1}{2}}}^2_{\varepsilon^{-1}} \right) \\
2 \aInnerproduct{f_E^n}{E^{n + \frac{1}{2}} + E^{n - \frac{1}{2}}} &\le \norm{f_E^n}^2_{\varepsilon^{-1}} + 2 \left( \norm{E^{n + \frac{1}{2}}}^2_{\varepsilon} + \norm{E^{n - \frac{1}{2}}}^2_{\varepsilon} \right), \\ 
2 \aInnerproduct{f_H^{n + \frac{1}{2}}}{H^{n + 1} + H^n} &\le \norm{f_H^{n + \frac{1}{2}}}^2_{\mu^{-1}} + 2 \left( \norm{H^{n + 1}}^2_{\mu} + \norm{H^n}^2_{\mu} \right).
\end{align*}
Using these inequalities into the previous expression, and summing over $n = 1$ to $N - 1$ leads us to:
\begin{multline*}
  \norm{p^{N - \frac{1}{2}}}^2_{\varepsilon^{-1}} - \norm{p^\frac{1}{2}}^2_{\varepsilon^{-1}} + \norm{E^{N - \frac{1}{2}}}^2_{\varepsilon} - \norm{E^\frac{1}{2}}^2_{\varepsilon} + \norm{H^N}^2_{\mu} - \norm{H^1}^2_{\mu} \le \\
  \Delta t \sum\limits_{n = 1}^{N - 1} \left[ \norm{f_p^n}^2_{\varepsilon^{-1}} + \norm{f_E^n}^2_{\varepsilon^{-1}} + \norm{f_H^{n + \frac{1}{2}}}^2_{\mu^{-1}} \right] + 
  2 \Delta t \sum\limits_{n = 1}^{N - 2} \left[ \norm{p^{n + \frac{1}{2}}}^2_{\varepsilon^{-1}} + \norm{E^{n + \frac{1}{2}}}^2_{\varepsilon} + \norm{H^{n + 1}}^2_{\mu} \right] + \\
  \Delta t \left[ \norm{p^{N - \frac{1}{2}}}^2_{\varepsilon^{-1}} + \norm{p^{\frac{1}{2}}}^2_{\varepsilon^{-1}} + \norm{E^{N - \frac{1}{2}}}^2_{\varepsilon} + \norm{E^{\frac{1}{2}}}^2_{\varepsilon} + \norm{H^N}^2_{\mu} + \norm{H^1}^2_{\mu} \right]. 
\end{multline*}
Now, $p^{\frac{1}{2}}$, $E^{\frac{1}{2}}$ and $H^1$ satisfy Equations~\labelcref{eqn:maxwell_p0_lf,eqn:maxwell_E0_lf,eqn:maxwell_H0_lf}, and so using $\widetilde{p} = 2 \Delta t \varepsilon^{-1} \left( p^{\frac{1}{2}} + p^0 \right)$, $\widetilde{E} = 2 \Delta t \left(E^{\frac{1}{2}} + E^0 \right)$ and $\widetilde{H} = 4 \Delta t \left(H^{1} + H^0 \right)$ there and repeating the arguments presented above leads us to the following estimate:
\begin{multline*}
  \norm{p^{\frac{1}{2}}}^2_{\varepsilon^{-1}}  - \norm{p^0}^2_{\varepsilon^{-1}} + \norm{E^{\frac{1}{2}}}^2_{\varepsilon} - \norm{E^0}^2_{\varepsilon} + \norm{H^{1}}^2_{\mu} - \norm{H^0}^2_{\mu}  \le \\
  \Delta t \left[ \norm{f_p^0}^2_{\varepsilon^{-1}} + \norm{f_E^0}^2_{\varepsilon^{-1}} + \norm{f_H^{\frac{1}{2}}}^2_{\mu^{-1}} \right] + \Delta t \left[ \norm{p^{\frac{1}{2}}}^2_{\varepsilon^{-1}} + \norm{p^0}^2_{\varepsilon^{-1}}  + \norm{E^{\frac{1}{2}}}^2_{\varepsilon} + \norm{E^0}^2_{\varepsilon}  + \norm{H^{1}}^2_{\mu}  + \norm{H^0}^2_{\mu} \right] 
\end{multline*}
Adding these last two inequalities, we get:
\begin{multline*}
  \norm{p^{N - \frac{1}{2}}}^2_{\varepsilon^{-1}} + \norm{E^{N - \frac{1}{2}}}^2_{\varepsilon} + \norm{H^N}^2_{\mu} \le \left[ \norm{p^0}^2_{\varepsilon^{-1}} + \norm{E^0}^2_{\varepsilon} + \norm{H^0}^2_{\mu} \right] + \\
  \Delta t \sum\limits_{n = 0}^{N - 1} \left[ \norm{f_p^n}^2_{\varepsilon^{-1}} + \norm{f_E^n}^2_{\varepsilon^{-1}} + \norm{f_H^{n + \frac{1}{2}}}^2_{\mu^{-1}} \right] + 2 \Delta t \sum\limits_{n = 0}^{N - 2} \left[ \norm{p^{n + \frac{1}{2}}}^2_{\varepsilon^{-1}} +\norm{E^{n + \frac{1}{2}}}^2_{\varepsilon} +  \norm{H^{n + 1}}^2_{\mu} \right] + \\
  \Delta t \left[\norm{p^{N - \frac{1}{2}}}^2_{\varepsilon^{-1}} + \norm{p^0}^2_{\varepsilon^{-1}} + \norm{E^{N - \frac{1}{2}}}^2_{\varepsilon} +
\norm{E^0}^2_{\varepsilon}  + \norm{H^N}^2_{\mu} + \norm{H^0}^2_{\mu} \right],
\end{multline*}
which in turn leads to:
\begin{multline*}
  \norm{p^{N - \frac{1}{2}}}^2_{\varepsilon^{-1}} + \norm{E^{N - \frac{1}{2}}}^2_{\varepsilon} + \norm{H^N}^2_{\mu} \le \dfrac{1 + \Delta t}{1 - \Delta t} \left[ \norm{p^0}^2_{\varepsilon^{-1}} + \norm{E^0}^2_{\varepsilon} + \norm{H^0}^2_{\mu} \right] + \\
  \dfrac{2 \Delta t}{1 - \Delta t} \sum\limits_{n = 0}^{N - 1} \left[ \norm{p^{n + \frac{1}{2}}}^2_{\varepsilon^{-1}} +\norm{E^{n + \frac{1}{2}}}^2_{\varepsilon} +  \norm{H^{n + 1}}^2_{\mu} \right] + \dfrac{ \Delta t}{1 - \Delta t} \sum\limits_{n = 0}^{N - 1} \left[ \norm{f_p^n}^2_{\varepsilon^{-1}} + \norm{f_E^n}^2_{\varepsilon^{-1}} + \norm{f_H^{n + \frac{1}{2}}}^2_{\mu^{-1}} \right].
\end{multline*}
In order to apply the discrete Gronwall inequality as in Lemma~\ref{lemma:gronwall_dscrt}, we choose $\delta \coloneq \dfrac{\Delta t}{1 - \Delta t}$, $g_0 \coloneq \dfrac{1 + \Delta t}{1 - \Delta t} \left[\norm{p^0}^2_{\varepsilon^{-1}} + \norm{E^0}^2_{\varepsilon} + \norm{H^0}^2_\mu \right]$, $a_n \coloneq \norm{p^{n+\frac{1}{2}}}^2_{\varepsilon^{-1}} + \norm{E^{n+\frac{1}{2}}}^2_{\varepsilon} + \norm{H^n}^2_\mu$, $b_n \coloneq 0$, $c_n \coloneq \norm{f_p^n}^2_{\varepsilon^{-1}} + \norm{f_E^n}^2_{\varepsilon^{-1}} + \norm{f_H^{n+\frac{1}{2}}}^2_{\mu^{-1}}$, and $\gamma_n \coloneq 2$. Note that for the condition $\gamma_n \delta < 1$ to hold, we need to have $\Delta t < 1/3$ and thus we get $\sigma_n = \left(1 - \dfrac{2 \Delta t}{1 - \Delta t} \right)^{-1}$. Then for $\Delta t < 1/6$, we have that:
\begin{multline*}
  \norm{p^{N - \frac{1}{2}}}^2_{\varepsilon^{-1}} + \norm{E^{N - \frac{1}{2}}}^2_{\varepsilon} + \norm{H^N}^2_{\mu} \le \Big[ \dfrac{\Delta t}{1 - \Delta t} \sum\limits_{n = 0}^{N - 1} \left( \norm{f_p^n}^2_{\varepsilon^{-1}} + \norm{f_E^n}^2_{\varepsilon^{-1}} + \norm{f_H^{n + \frac{1}{2}}}^2_{\mu^{-1}} \right) + \\
  \dfrac{1 + \Delta t}{1 - \Delta t} \left( \norm{p^0}^2_{\varepsilon^{-1}} + \norm{E^0}^2_{\varepsilon} + \norm{H^0}^2_\mu \right) \Big] \exp\left[ \dfrac{\Delta t}{1 - \Delta t} \sum\limits_{n = 0}^{N - 1} 2 \left(1 -  \dfrac{2 \Delta t}{1 - \Delta t}\right)^{-1} \right],
\end{multline*}
\vspace{-1em}
\[ 
  \le \Big[ \dfrac{6 \Delta t}{5} \sum\limits_{n = 0}^{N - 1} \left( \norm{f_p^n}^2_{\varepsilon^{-1}} + \norm{f_E^n}^2_{\varepsilon^{-1}} + \norm{f_H^{n + \frac{1}{2}}}^2_{\mu^{-1}} \right) +
  \dfrac{7}{5} \left( \norm{p^0}^2_{\varepsilon^{-1}} + \norm{E^0}^2_{\varepsilon} + \norm{H^0}^2_\mu \right) \Big] \exp\left( 4 T \right),
\]
where the second inequality is obtained by using that $N \Delta t = T$. We conclude by using a similar argument as in the end of Theorem~\ref{thm:dscrt_enrgy_estmt_cn} by setting $M \coloneq \norm{f_p}^2_{L^2[0, T] \times L^2_{\varepsilon^{-1}}(\Omega)} + \norm{f_E}^2_{L^2[0, T] \times L^2_{\varepsilon^{-1}}(\Omega)} + \norm{f_H}^2_{L^2[0, T] \times L^2_{\mu^{-1}}(\Omega)}$ and using $C = \sqrt{6 \left[ M +\norm{p_0}^2_{\varepsilon^{-1}} + \norm{E_0}^2_{\varepsilon} + \norm{H_0}^2_{\mu} \right] \exp(4T)}$, we obtain our desired result:
\[
  \norm{p^{N - \frac{1}{2}}}_{\varepsilon^{-1}} + \norm{E^{N - \frac{1}{2}}}_{\varepsilon} + \norm{H^N}_{\mu} \le C. \qedhere
\]
\end{proof}

\begin{corollary}[Discrete Energy Conservation]\label{corr:dscrt_enrgy_cnsrvtn_lf}
If the forcing functions in Equation~\eqref{eqn:maxwells_eqns} are all zero, that is, $f_p = 0$ and $f_E = f_H = 0$, then:
\[
  \norm{p^{N - \frac{1}{2}}}^2_{\varepsilon^{-1}} + \norm{E^{N - \frac{1}{2}}}^2_{\varepsilon} + \norm{H^N}^2_{\mu} = \norm{p_0}^2_{\varepsilon^{-1}} + \norm{E_0}^2_{\varepsilon} + \norm{H_0}^2_{\mu}.
\]
\end{corollary}

\begin{theorem}[Discrete Error Estimate]\label{thm:dscrt_error_estmt_lf}
For the semidiscretization using the implicit leapfrog scheme as given in Equations~\labelcref{eqn:maxwell_p_lf,eqn:maxwell_E_lf,eqn:maxwell_H_lf}, and ~\labelcref{eqn:maxwell_p0_lf,eqn:maxwell_E0_lf,eqn:maxwell_H0_lf}, for the solution $(p, E, H)$ of Equations~\labelcref{eqn:maxwell_p_wf,eqn:maxwell_E_wf,eqn:maxwell_H_wf} with initial conditions as in Equation~\eqref{eqn:ICs} and assuming sufficient regularity with $p \in C^3[0, T] \times \mathring{H}^1_{\varepsilon^{-1}}(\Omega)$, $E \in C^3[0, T] \times \mathring{H}_{\varepsilon}(\curl; \Omega)$, and $H \in C^3[0, T] \times \mathring{H}_{\mu}(\divgn; \Omega)$, and for the time step $\Delta t$ sufficiently small, there exists a positive bounded constant $C$ independent of $\Delta t$ such that:
\[
  \norm{e_p^{N - \frac{1}{2}}}_{\varepsilon^{-1}} + \norm{e_E^{N - \frac{1}{2}}}_{\varepsilon} + \norm{e_H^N}_{\mu} \le C \left[ \left(\Delta t\right)^2 + \norm{e_p^0}_{\varepsilon^{-1}} + \norm{e_E^0}_{\varepsilon} + \norm{e_H^0}_{\mu} \right],
\]
where $e_p^{n + \frac{1}{2}} \coloneq p(t^{n + \frac{1}{2}}) - p^{n+\frac{1}{2}}$, $e_E^{n + \frac{1}{2}} \coloneq E(t^{n + \frac{1}{2}}) - E^{n + \frac{1}{2}}$ and $e_H^n \coloneq H(t^n) - H^n$ are the errors in the time semidiscretization of $p$, $E$ and $H$, respectively and at the indicated time indices.
\end{theorem}

\begin{proof}
Using the Taylor remainder theorem, and expressing $p(t)$ about $t = t^n$, we have that:
\[
  p(t) = p(t^n) + \dfrac{\partial p}{\partial t}(t^n)(t - t^n) + \dfrac{\partial^2 p}{\partial t^2}(t^n) \dfrac{(t - t^n)^2}{2} + \int\limits_{t^n}^{t} \dfrac{(t - s)^2}{2} \dfrac{\partial^3 p}{\partial t^3}(s) ds,
\]
which when evaluated at $t = t^{n + \frac{1}{2}}$ and $t = t^{n - \frac{1}{2}}$ yields:
\begin{align*}
  p(t^{n + \frac{1}{2}}) &= p(t^n) + \dfrac{\partial p}{\partial t}(t^n) (t^{n + \frac{1}{2}} - t^n) + \dfrac{\partial^2 p}{\partial t^2}(t^n) \dfrac{(t^{n + \frac{1}{2}} - t^n)^2}{2} + \int\limits_{t^n}^{\mathclap{t^{n + \frac{1}{2}}}} \dfrac{(t^{n + \frac{1}{2}} - s)^2}{2} \dfrac{\partial^3 p}{\partial t^3}(s) ds, \\
p(t^{n - \frac{1}{2}}) &= p(t^n) + \dfrac{\partial p}{\partial t}(t^n)(t^{n - \frac{1}{2}} - t^n) + \dfrac{\partial^2 p}{\partial t^2}(t^n) \dfrac{(t^{n - \frac{1}{2}} - t^n)^2}{2} + \int\limits_{t^n}^{\mathclap{t^{n - \frac{1}{2}}}} \dfrac{(t^{n - \frac{1}{2}} - s)^2}{2} \dfrac{\partial^3 p}{\partial t^3}(s) ds.
\end{align*}
Subtracting these two equations, and using the result in the inner product term from the semidiscretization of the variational formulation leads to:
\[
  \ainnerproduct{\dfrac{p(t^{n + \frac{1}{2}}) - p(t^{n - \frac{1}{2}})}{\Delta t}}{\widetilde{p}} = \ainnerproduct{\dfrac{\partial p}{\partial t}(t^n)}{\widetilde{p}} + \ainnerproduct{R^n_p}{\widetilde{p}},
\]
in which we have defined that:
\[
  R^n_p \coloneq \dfrac{1}{\Delta t} \left[\int\limits_{t^{n - \frac{1}{2}}}^{t^n} \dfrac{(t^{n - \frac{1}{2}} - s)^2}{2} \dfrac{\partial^3 p}{\partial t^3}(s) ds + \int\limits_{t^n}^{t^{n + \frac{1}{2}}} \dfrac{(t^{n + \frac{1}{2}} - s)^2}{2} \dfrac{\partial^3 p}{\partial t^3}(s) ds \right].
\]
Similarly, for $E$ and $H$, we have the following:
\begin{align*}
  \ainnerproduct{\varepsilon \dfrac{E(t^{n + \frac{1}{2}}) - E(t^{n - \frac{1}{2}})}{\Delta t}}{\widetilde{E}} &= \ainnerproduct{\varepsilon \dfrac{\partial E}{\partial t}(t^n)}{\widetilde{E}} + \ainnerproduct{\varepsilon R^n_E}{\widetilde{E}}, \\
  \ainnerproduct{\mu \dfrac{H(t^{n + 1}) - H(t^n)}{\Delta t}}{\widetilde{H}} &= \ainnerproduct{\mu \dfrac{\partial H}{\partial t}(t^{n + \frac{1}{2}})}{\widetilde{H}} + \ainnerproduct{\mu R^{n + \frac{1}{2}}_H}{\widetilde{H}},
\end{align*}
and in each of which we have defined that:
\begin{align*}
  R^n_E &\coloneq \dfrac{1}{\Delta t} \left[ \int\limits_{t^{n - \frac{1}{2}}}^{t^n} \dfrac{(t^{n - \frac{1}{2}} - s)^2}{2} \dfrac{\partial^3 E}{\partial t^3}(s) ds + \int\limits_{t^n}^{t^{n + \frac{1}{2}}} \dfrac{(t^{n + \frac{1}{2}} - s)^2}{2} \dfrac{\partial^3 E}{\partial t^3}(s) ds \right], \\
R^{n + \frac{1}{2}}_H &\coloneq \dfrac{1}{\Delta t} \left[ \int\limits_{t^n}^{t^{n + \frac{1}{2}}} \dfrac{(t^n - s)^2}{2} \dfrac{\partial^3 H}{\partial t^3}(s) ds + \int\limits_{t^{n + \frac{1}{2}}}^{t^{n + 1}} \dfrac{(t^{n + 1} - s)^2}{2} \dfrac{\partial^3 H}{\partial t^3}(s) ds\right].
\end{align*}
Using these terms in the weak formulation as in Equations~\labelcref{eqn:maxwell_p_wf,eqn:maxwell_E_wf,eqn:maxwell_H_wf}  at time $t = t^n$ for $p$ and $E$ terms, and at time $t = t^{n + \frac{1}{2}}$ for $H$ and with the definition that:
\[
  f(t^n) \coloneq \dfrac{u(t^{n + \frac{1}{2}}) + u(t^{n - \frac{1}{2}})}{2}, \quad f_H(t^{n + \frac{1}{2}}) \coloneq \dfrac{f_H(t^{n+1}) + f_H(t^n)}{2}
\] 
where $f$ is either of $f_p$ or $f_E$, and using the Taylor remainder theorem again with:
\[
  u(t^n) \coloneq \dfrac{u(t^{n + \frac{1}{2}}) + u(t^{n - \frac{1}{2}})}{2} - r_u^n, \quad H(t^{n + \frac{1}{2}}) \coloneq \dfrac{H(t^{n+1}) + H(t^n)}{2} - r_H^n
\] 
where $u$ is either of $p$ or $E$, and in which we have defined that:
\[
r_u^n \coloneqq \dfrac{1}{2} \left[\int\limits_{\mathclap{t^n}}^{t^{n + \frac{1}{2}}} (t^{n + \frac{1}{2}} - s) \dfrac{\partial^2 u}{\partial t^2}(s) ds - \int\limits_{t^{n - \frac{1}{2}}}^{t^n} (t^{n - \frac{1}{2}} - s) \dfrac{\partial^2 u}{\partial t^2}(s) ds\right],
\]
\[
r_H^n \coloneqq \dfrac{1}{2} \left[\,\, \int\limits_{\mathclap{t^{n + \frac{1}{2}}}}^{t^{n+1}} (t^{n+1} - s) \dfrac{\partial^2 H}{\partial t^2}(s) ds - \int\limits_{t^n}^{t^{n + \frac{1}{2}}} (t^n - s) \dfrac{\partial^2 H}{\partial t^2}(s) ds\right],
\]
and then subtracting the implicit leapfrog semidiscretization as in Equations~\labelcref{eqn:maxwell_p_lf,eqn:maxwell_E_lf,eqn:maxwell_H_lf} leads us to the following set of equations:
\begin{align*}
\aInnerproduct{\dfrac{e_p^{n + \frac{1}{2}} - e_p^{n - \frac{1}{2}}}{\Delta t}}{\widetilde{p}} - \dfrac{1}{2} \aInnerproduct{ \varepsilon \left(e_E^{n + \frac{1}{2}} + e_E^{n - \frac{1}{2}} \right)}{\nabla \widetilde{p}} &= \aInnerproduct{R_p^{n} - \varepsilon r_E^n}{\widetilde{p}}, \\
 \dfrac{1}{2} \aInnerproduct{\nabla \left(e_p^{n + \frac{1}{2}} + e_p^{n - \frac{1}{2}}\right)}{\widetilde{E}} + \aInnerproduct{\varepsilon \dfrac{e_E^{n + \frac{1}{2}} - e_E^{n - \frac{1}{2}}}{\Delta t}}{\widetilde{E}} - \dfrac{1}{2} \aInnerproduct{e_H^{n + 1} + e_H^n}{\nabla \times \widetilde{E}} &= \aInnerproduct{\varepsilon R_E^n + \nabla r_p^n - r_H^{n+\frac{1}{2}}}{\widetilde{E}}, \\
\aInnerproduct{\mu \dfrac{e_H^{n + 1} - e_H^{n}}{\Delta t}}{\widetilde{H}} +  \dfrac{1}{2}\aInnerproduct{\nabla \times\left(e_E^{n+\frac{1}{2}} + e_E^{n - \frac{1}{2}} \right)}{\widetilde{H}} &= \aInnerproduct{\mu R_H^{n + \frac{1}{2}} + \nabla \times r_E^n}{\widetilde{H}}.
\end{align*}
Likewise, for the semidiscrete approximation of the initial system as in Equations~\labelcref{eqn:maxwell_p0_lf,eqn:maxwell_E0_lf,eqn:maxwell_H0_lf}, we obtain for their errors the following system of equations:
\begin{align*}
  \aInnerproduct{\dfrac{e_p^{\frac{1}{2}} - e_p^0}{\Delta t/2}}{\widetilde{p}} - \dfrac{1}{4} \aInnerproduct{ \varepsilon \left(e_E^{\frac{1}{2}} + e_E^0 \right)}{\nabla \widetilde{p}} &= \aInnerproduct{\dfrac{1}{2} R_p^0 -   \dfrac{\varepsilon}{2} r_E^0}{\widetilde{p}}, \\
  \dfrac{1}{4} \aInnerproduct{\nabla \left(e_p^{\frac{1}{2}} + e_p^0 \right)}{\widetilde{E}} + \aInnerproduct{\varepsilon \dfrac{e_E^{\frac{1}{2}} - e_E^0}{\Delta t/2}}{\widetilde{E}} - \dfrac{1}{2} \aInnerproduct{e_H^1 + e_H^0}{\nabla \times \widetilde{E}} &= \aInnerproduct{\dfrac{\varepsilon}{2} R_E^0 + \dfrac{1}{2} \nabla r_p^0 - r_H^{\frac{1}{2}}}{\widetilde{E}}, \\
  \aInnerproduct{\mu \dfrac{e_H^{1} - e_H^0}{\Delta t}}{\widetilde{H}} + \dfrac{1}{4} \aInnerproduct{\nabla \times \left( e_E^{\frac{1}{2}} + e_E^0 \right)}{\widetilde{H}} &= \aInnerproduct{\mu R_H^{\frac{1}{2}}+ \dfrac{1}{2} \nabla \times r_E^0}{\widetilde{H}},
\end{align*}
and we define the initial remainder terms as follows:
\begin{align*}
R^0_u & \coloneq \dfrac{1}{\Delta t} \left[ \int\limits_{t^0}^{t^\frac{1}{4}} \dfrac{(t^0 - s)^2}{2} \dfrac{\partial^3 u}{\partial t^3}(s) ds + \int\limits_{t^\frac{1}{4}}^{t^{\frac{1}{2}}} \dfrac{(t^{\frac{1}{2}} - s)^2}{2} \dfrac{\partial^3 u}{\partial t^3}(s) ds \right], \\
R^{ \frac{1}{2}}_H & \coloneq \dfrac{1}{\Delta t} \left[ \int\limits_{t^0}^{t^{ \frac{1}{2}}} \dfrac{(t^0 - s)^2}{2} \dfrac{\partial^3 H}{\partial t^3}(s) ds + \int\limits_{t^{\frac{1}{2}}}^{t^{1}} \dfrac{(t^{1} - s)^2}{2} \dfrac{\partial^3 H}{\partial t^3}(s) ds\right], \\
r_u^0 & \coloneqq \dfrac{1}{2} \left[\int\limits_{\mathclap{t^\frac{1}{4}}}^{t^{ \frac{1}{2}}} (t^{ \frac{1}{2}} - s) \dfrac{\partial^2 u}{\partial t^2}(s) ds - \int\limits_{t^0}^{t^\frac{1}{4}} (t^0 - s) \dfrac{\partial^2 u}{\partial t^2}(s) ds\right], \\
r_H^{ \frac{1}{2}} & \coloneqq \dfrac{1}{2} \left[\,\, \int\limits_{\mathclap{t^{\frac{1}{2}}}}^{t^{1}} (t^{1} - s) \dfrac{\partial^2 H}{\partial t^2}(s) ds - \int\limits_{t^0}^{t^{\frac{1}{2}}} (t^0 - s) \dfrac{\partial^2 H}{\partial t^2}(s) ds\right].
\end{align*}
Now by choosing appropriate test functions for these two systems of equations as in Theorem~\ref{thm:dscrt_error_estmt_cn} and by repeating the arguments for summing over $n = 1$ to $N - 1$, we get: 
\begin{multline*}
  \norm{e_p^{N - \frac{1}{2}}}^2_{\varepsilon^{-1}} + \norm{e_E^{N - \frac{1}{2}}}^2_{\varepsilon} + \norm{e_H^N}^2_{\mu} \le \dfrac{1 + \Delta t}{1 - \Delta t} \left[ \norm{e_p^0}^2_{\varepsilon^{-1}} + \norm{e_E^0}^2_{\varepsilon} + \norm{e_H^0}^2_{\mu} \right] + \\
  \dfrac{\Delta t}{1 - \Delta t} \sum\limits_{n = 0}^{N - 1} \Big[ 2 \left( \norm{e_p^{n + \frac{1}{2}}}^2_{\varepsilon^{-1}} + \norm{e_E^{n + \frac{1}{2}}}^2_{\varepsilon} + \norm{e_H^{n + 1}}^2_{\mu} \right) + \left( \norm{R_p^n}^2_{\varepsilon^{-1}} + \norm{R_E^n}^2_{\varepsilon} + \norm{R_H^{n + \frac{1}{2}}}^2_{\mu} + \right. \\ 
\left. \norm{\nabla r_p^n} ^2_{\varepsilon^{-1}} + \norm{r_E^n}^2_{\varepsilon} + \varepsilon^{-1} \mu^{-1} \norm{\nabla \times r_E^n}^2_{\varepsilon} +  \varepsilon^{-1} \mu^{-1} \norm{r_H^{n + \frac{1}{2}}}^2_{\mu} \right) \Big].
\end{multline*}
Applying the discrete Gronwall inequality similar to Theorem~\ref{thm:dscrt_enrgy_estmt_lf}, we obtain the estimate:
\begin{multline*}
  \norm{e_p^{N - \frac{1}{2}}}^2_{\varepsilon^{-1}} + \norm{e_E^{N - \frac{1}{2}}}^2_{\varepsilon} +\norm{e_H^N}^2_{\mu} \le \Bigg[ \dfrac{6 \Delta t}{5} \sum\limits_{n = 0}^{N - 1} \left(\norm{R_p^n}^2_{\varepsilon^{-1}} + \norm{R_E^n}^2_{\varepsilon} + \norm{R_H^{n + \frac{1}{2}}}^2_{\mu} + + \norm{\nabla r_p^n} ^2_{\varepsilon^{-1}} + \norm{r_E^n}^2_{\varepsilon} + \right. \\ 
\left. \varepsilon^{-1} \mu^{-1} \norm{\nabla \times r_E^n}^2_{\varepsilon} +  \varepsilon^{-1} \mu^{-1} \norm{r_H^{n + \frac{1}{2}}}^2_{\mu} \right) +
  \dfrac{7}{5} \left(\norm{e_p^0}^2_{\varepsilon^{-1}} + \norm{e_E^0}^2_{\varepsilon}  + \norm{e_H^0}^2_\mu \right)\Bigg] \exp\left( 4 T \right).
\end{multline*}
Now, for the Taylor remainder terms on the right hand side of this equation, using arguments similar to that in Theorem~\ref{thm:dscrt_error_estmt_cn}, we can obtain the following inequalities:
\begin{gather*}
\begin{alignat*}{3}
\sum\limits_{n = 0}^{N - 1} \norm{R^n_p}^2_{\varepsilon^{-1}} &\le \dfrac{ \left( \Delta t \right)^3}{20} \norm[\bigg]{\dfrac{\partial^3 p}{\partial t^3}}^2_{L^2[0, T] \times L^2_{\varepsilon^{-1}}(\Omega)}, & \sum\limits_{n = 0}^{N - 1} \norm{R^n_E}^2_{\varepsilon} & \le \dfrac{\left( \Delta t \right)^3}{20} \norm[\bigg]{\dfrac{\partial^3 E}{\partial t^3}}^2_{L^2[0, T] \times L^2_\varepsilon(\Omega)}, \\
  \sum\limits_{n = 0}^{N - 1} \norm{R^{n + \frac{1}{2}}_H}^2_\mu &\le \dfrac{\left( \Delta t \right)^3}{20} \norm[\bigg]{\dfrac{\partial^3 H}{\partial t^3}}^2_{L^2[0, T] \times L^2_\mu(\Omega)}, & \sum\limits_{n = 0}^{N - 1} \norm{\nabla r^n_p}^2_{\varepsilon^{-1}} &\le  \dfrac{\left( \Delta t \right)^3}{48} \norm[\bigg]{\dfrac{\partial^2 \left(\nabla p \right)}{\partial t^2}}^2_{L^2[0, T] \times \mathring{H}^1_{\varepsilon^{-1}}(\Omega)}, \\
\sum\limits_{n = 0}^{N - 1} \norm{r^n_E}^2_{\varepsilon} &\le \dfrac{\left( \Delta t \right)^3}{48} \norm[\bigg]{\dfrac{\partial^2 E}{\partial t^2}}^2_{L^2[0, T] \times L^2_\varepsilon(\Omega)}, & \sum\limits_{n = 0}^{N - 1} \norm{\nabla \times r^n_E}^2_{\varepsilon} &\le \dfrac{\left( \Delta t \right)^3}{48} \norm[\bigg]{\dfrac{\partial^2 (\nabla \times E)}{\partial t^2}}^2_{L^2[0, T] \times \mathring{H}_\varepsilon(\curl; \Omega)},
\end{alignat*}\\
\sum\limits_{n = 0}^{N - 1} \norm{r^n_H}^2_\mu \le \dfrac{\left( \Delta t \right)^3}{48} \norm[\bigg]{\dfrac{\partial^2 H}{\partial t^2}}^2_{L^2[0, T] \times L^2_\mu(\Omega)}.
\end{gather*}
As a result, we can finally obtain:
\begin{multline*}
\norm{e_p^{N - \frac{1}{2}}}^2_{\varepsilon^{-1}} + \norm{e_E^{N - \frac{1}{2}}}^2_{\varepsilon} + \norm{e_H^N}^2_{\mu} \le \\
\Bigg[ \dfrac{3}{50} \left( \Delta t \right)^4 \left( \norm[\bigg]{\dfrac{\partial^3 p}{\partial t^3}}^2_{L^2[0, T] \times L^2_{\varepsilon^{-1}}(\Omega)} + \norm[\bigg]{\dfrac{\partial^3 E}{\partial t^3}}^2_{L^2[0, T] \times L^2_\varepsilon(\Omega)} + \norm[\bigg]{\dfrac{\partial^3 H}{\partial t^3}}^2_{L^2[0, T] \times L^2_\mu(\Omega)} \right) + \\
\dfrac{1}{40} \left( \Delta t \right)^4 \left(\norm[\bigg]{\dfrac{\partial^2 \left(\nabla p \right)}{\partial t^2}}^2_{L^2[0, T] \times \mathring{H}^1_{\varepsilon^{-1}}(\Omega)} + \norm[\bigg]{\dfrac{\partial^2 E}{\partial t^2}}^2_{L^2[0, T] \times L^2_\varepsilon(\Omega)} + \varepsilon^{-1} \mu^{-1} \norm[\bigg]{\dfrac{\partial^2 (\nabla \times E)}{\partial t^2}}^2_{L^2[0, T] \times \mathring{H}_\varepsilon(\curl; \Omega)} + \right. \\
\left. \varepsilon^{-1} \mu^{-1} \norm[\bigg]{\dfrac{\partial^2 H}{\partial t^2}}^2_{L^2[0, T] \times L^2_\mu(\Omega)} \right)  + \dfrac{7}{5} (\norm{e_p^0}^2_{\varepsilon^{-1}} + \norm{e_E^0}^2_{\varepsilon} + \norm{e_H^0}^2_{\mu})\Bigg] \exp \left( 4 T \right).
\end{multline*}
Finally, using the regularity assumptions for $p$, $E$ and $H$, and $1$- and $2$-norm equivalence, we obtain the required result:
\[
  \norm{e_p^{N - \frac{1}{2}}}_{\varepsilon^{-1}} + \norm{e_E^{N - \frac{1}{2}}}_{\varepsilon} + \norm{e_H^N}_{\mu} \le C \left[ (\Delta t)^2 + \norm{e_p^0}_{\varepsilon^{-1}} + \norm{e_E^0}_{\varepsilon} + \norm{e_H^0}_{\mu} \right]. \qedhere
\]
\end{proof}

\section{Error Estimates for Full Discretization} \label{sec:fllerrrestmts}

We now present the error analysis for the full discretization of the Maxwell's equations using finite elements and with our two time integration methods.

\subsection{Crank-Nicholson Scheme}\label{sec:cn_full_error}

For the Crank-Nicholson scheme as in Equations~\labelcref{eqn:maxwell_p_cn,eqn:maxwell_E_cn,eqn:maxwell_H_cn}, using a de Rham sequence of finite dimensional subspaces of the corresponding function spaces for the spatial discretization of $(p^n, E^n, H^n)$, we obtain the following discrete problem: find $(p^n_h, E^n_h, H^n_h) \in U_h \times V_h \times W_h \subseteq \mathring{H}^1_{\varepsilon^{-1}}(\Omega) \times \mathring{H}_{\varepsilon}(\curl; \Omega) \times \mathring{H}_{\mu}(\divgn; \Omega)$ such that:
\begin{subequations}
  \begin{equation}
    \aInnerproduct{\dfrac{p_h^n - p_h^{n - 1}}{\Delta t}}{\widetilde{p}} - \aInnerproduct{\dfrac{\varepsilon}{2} \left(  E_h^n +  E_h^{n - 1} \right)}{\nabla \widetilde{p}} = \aInnerproduct{\dfrac{1}{2} \left( f_p^n + f_p^{n - 1} \right)}{\widetilde{p}}, \label{eqn:maxwell_p_cn_full}
  \end{equation}
  \begin{multline}
    \aInnerproduct{\dfrac{1}{2} \nabla \left(  p_h^n +  p_h^{n - 1} \right)}{\widetilde{E}} + \aInnerproduct{\varepsilon \dfrac{E_h^n - E_h^{n - 1}}{\Delta t}}{\widetilde{E}} - \aInnerproduct{\dfrac{1}{2} \left( H_h^n + H_h^{n - 1} \right)}{\nabla \times \widetilde{E}} = 
    \aInnerproduct{\dfrac{1}{2} \left(f_E^n + f_E^{n - 1} \right)}{\widetilde{E}}, \label{eqn:maxwell_E_cn_full}
  \end{multline}
  \begin{equation}
    \aInnerproduct{\dfrac{1}{2}  \nabla \times  \left( E_h^n + E_h^{n - 1} \right)}{\widetilde{H}} + \aInnerproduct{\mu \dfrac{H_h^n - H_h^{n - 1}}{\Delta t}}{\widetilde{H}} = \aInnerproduct{\dfrac{1}{2} \left( f_H^n + f_H^{n - 1} \right)}{\widetilde{H}}. \label{eqn:maxwell_H_cn_full}
  \end{equation}
\end{subequations}
for all $(\widetilde{p}, \widetilde{E}, \widetilde{H}) \in U_h \times V_h \times W_h$, and for $n = 1, \dots, N$ given $(p^0_h, E^0_h, H^0_h) \in U_h \times V_h \times W_h$. 

Let $\Pi_h^0: \mathring{H}^1_{\varepsilon^{-1}}(\Omega) \longto U_h$, $\Pi_h^1: \mathring{H}_{\varepsilon}(\curl; \Omega) \longto V_h$ and $\Pi_h^2: \mathring{H}_{\mu}(\divgn; \Omega) \longto W_h$ denote the respective smoothed $L^2$ projection operators as in the sense of \cite{Schoberl2008,Christiansen2007} and a detailed discussion is available in \cite{ArFaWi2006}. Now, we define the error for $p$, $E$ and $H$ at time $n \Delta t$ under the full discretization as:
\begin{alignat}{2}
  e_{p_h}^n &\coloneq p(t^n) - p_h^n &&= \eta^n - \eta_h^n, \label{eqn:p_fullerror_cn} \\
  e_{E_h}^n &\coloneq E(t^n) - E_h^n &&= \zeta^n - \zeta_h^n, \label{eqn:E_fullerror_cn} \\
  e_{H_h}^n &\coloneq H(t^n) - H_h^n &&= \xi^n - \xi_h^n, \label{eqn:H_fullerror_cn}
\end{alignat}
and in which we have the following definitions for the newly introduced terms:
\begin{alignat}{3}
  \eta^n &\coloneq p(t^n) - \Pi_h^0 p(t^n), &&\quad \eta_h^n &&\coloneq p_h^n  - \Pi_h^0 p(t^n), \label{eqn:p_fullerror_sub_cn} \\
  \zeta^n &\coloneq E(t^n) - \Pi_h^1 E(t^n), &&\quad \zeta_h^n &&\coloneq E_h^n - \Pi_h^1 E(t^n), \label{eqn:E_fullerror_sub_cn} \\
  \xi^n &\coloneq H(t^n) - \Pi_h^2 H(t^n), &&\quad \xi_h^n &&\coloneq H_h^n - \Pi_h^2 H(t^n). \label{eqn:H_fullerror_sub_cn}
\end{alignat}
With this, we can now state and prove our theorem for convergence of errors in the full discretization of the system of Maxwell's equations using the Crank-Nicholson scheme.

\begin{theorem}[Full Error Estimate] \label{thm:full_error_estimate_cn}
Let $p \in C^3[0, T] \times \mathring{H}^1_{\varepsilon^{-1}}(\Omega)$, $E \in C^3[0, T] \times \mathring{H}_{\varepsilon}(\curl; \Omega)$, and $H \in C^3[0, T] \times \mathring{H}_{\mu}(\divgn; \Omega)$ be the solution to the variational formulation of the Maxwell's equations as in Equations~\labelcref{eqn:maxwell_p_wf,eqn:maxwell_E_wf,eqn:maxwell_H_wf}, and let $(p_h^n, E_h^n, H_h^n)$ be the solution of the fully discretized Maxwell's equations using the Crank-Nicholson scheme as in Equations~\labelcref{eqn:maxwell_p_cn_full,eqn:maxwell_E_cn_full,eqn:maxwell_H_cn_full}. If the time step $\Delta t > 0$ and mesh parameter $h > 0$ are sufficiently small, then there exists a positive bounded constant $C$ independent of both $\Delta t$ and $h$ such that the following error estimate holds:
\[
  \norm{e_{p_h}^N}_{\varepsilon^{-1}} + \norm{e_{E_h}^{N}}_{\varepsilon} + \norm{e_{H_h}^{N}}_{\mu} \le C \left[ (\Delta t)^2 + h^r + h^r (\Delta t)^2 \right],
\]
where the finite element subspaces $U_h$, $V_h$ and $W_h$ are each spanned by their respective Whitney form basis of polynomial order $r \ge 1$.
\end{theorem}

\begin{proof}
First, we shall subtract the set of equations for the full discretization as in Equations~\labelcref{eqn:maxwell_p_cn_full,eqn:maxwell_E_cn_full,eqn:maxwell_H_cn_full} from Equations~\labelcref{eqn:remainder_p_cn,eqn:remainder_E_cn,eqn:remainder_H_cn},  and then use the error terms in Equations~\labelcref{eqn:p_fullerror_cn,eqn:E_fullerror_cn,eqn:H_fullerror_cn} and thereby obtain:
\begin{align*}
  \aInnerproduct{\dfrac{e_{p_h}^n - e_{p_h}^{n-1}}{\Delta t}}{\widetilde{p}} - \dfrac{1}{2}\aInnerproduct{ \varepsilon \left(e_{E_h}^n + e_{E_h}^{n-1} \right)}{\nabla \widetilde{p}} &= \dfrac{1}{2} \aInnerproduct{ R_p^{n} - \varepsilon r_E^n}{\widetilde{p}}, \\
  \dfrac{1}{2} \aInnerproduct{\nabla \left( e_{p_h}^n + \nabla e_{p_h}^{n - 1} \right)}{\widetilde{E}} + \aInnerproduct{\varepsilon \dfrac{e_{E_h}^n - e_{E_h}^{n - 1}}{\Delta t}}{\widetilde{E}} - \dfrac{1}{2}\aInnerproduct{e_{H_h}^n + e_{H_h}^{n - 1}}{\nabla \times \widetilde{E}} &= 
  \dfrac{1}{2} \aInnerproduct{\varepsilon R_E^n + \nabla r_p^n - r_H^n}{\widetilde{E}}, \\
  \dfrac{1}{2} \aInnerproduct{\nabla \times \left( e_{E_h}^n + e_{E_h}^{n - 1} \right)}{\widetilde{H}} + \aInnerproduct{\mu \dfrac{e_{H_h}^n - e_{H_h}^{n-1}}{\Delta t}}{\widetilde{H}} &= \dfrac{1}{2} \aInnerproduct{\mu R_H^n + \nabla \times r_E^n}{\widetilde{H}}.
\end{align*}
Next, using the values of the error terms $e_{p_h}^n$, $e_{E_h}^n$ and $e_{H_h}^n$ as in Equations~\labelcref{eqn:p_fullerror_cn,eqn:E_fullerror_cn,eqn:H_fullerror_cn} in the above equations, we get:
\[
  \aInnerproduct{\dfrac{\left(\eta^n - \eta^{n - 1}\right) - \left(\eta^n_h - \eta^{n - 1}_h\right)}{\Delta t}}{\widetilde{p}} - \dfrac{1}{2}\aInnerproduct{ \varepsilon \left(\zeta^n + \zeta^{n - 1} \right) - \varepsilon \left(\zeta^n_h + \zeta^{n - 1}_h \right)}{\nabla \widetilde{p}}  = \dfrac{1}{2} \aInnerproduct{ R_p^{n} - \varepsilon r_E^n}{\widetilde{p}},
\]
\vspace{-1em} \begin{multline*}
  \dfrac{1}{2} \aInnerproduct{\nabla \left(\eta^n + \eta^{n - 1}\right) - \nabla \left(\eta^n_h + \eta^{n - 1}_h \right)}{\widetilde{E}} + \aInnerproduct{\varepsilon \dfrac{\left(\zeta^n - \zeta^{n - 1} \right) - \left(\zeta^n_h - \zeta^{n - 1}_h \right)}{\Delta t}}{\widetilde{E}} \, - \\
  \dfrac{1}{2}\aInnerproduct{\left(\xi^n + \xi^{n - 1} \right) - \left(\xi^n_h + \xi^{n - 1}_h \right)}{\nabla \times \widetilde{E}} = \dfrac{1}{2} \aInnerproduct{\varepsilon R_E^n + \nabla r_p^n - r_H^n}{\widetilde{E}},
\end{multline*}
\vspace{-1em}
\[
  \aInnerproduct{\mu \dfrac{\left(\xi^n - \xi^{n - 1} \right) - \left(\xi^n_h - \xi^{n - 1}_h \right)}{\Delta t}}{\widetilde{H}} + \dfrac{1}{2} \aInnerproduct{\nabla \times \left(\zeta^n + \zeta^{n - 1} \right) - \nabla \times \left(\zeta^n_h + \zeta^{n - 1}_h \right)}{\widetilde{H}} = \dfrac{1}{2} \aInnerproduct{\mu R_H^n + \nabla \times r_E^n}{\widetilde{H}}.
\]
Since these equations are true for all $(\widetilde{p}, \widetilde{E}, \widetilde{H}) \in U_h \times V_h \times W_h$, we choose $\widetilde{p} = -2 \Delta t \varepsilon^{-1} \left( \eta_h^n + \eta_h^{n - 1}\right)$, $\widetilde{E} = -2 \Delta t \left( \zeta_h^n + \zeta_h^{n - 1} \right)$ and $\widetilde{H} = -2 \Delta t \left( \xi_h^n + \xi_h^{n - 1} \right)$ and using the fact that $\nabla U_h \subseteq V_h$ and $\nabla \times V_h \subseteq W_h$, we obtain:
\begin{multline}
2 \ainnerproduct{\varepsilon^{-1} \left( \eta^n_h - \eta^{n - 1}_h \right)}{\eta^n_h + \eta^{n - 1}_h} + 2 \ainnerproduct{\varepsilon \left( \zeta^n_h - \zeta^{n - 1}_h \right)}{\zeta^n_h + \zeta^{n - 1}_h} \, + 2 \ainnerproduct{\mu \left( \xi^n_h - \xi^{n - 1}_h \right)}{\xi^n_h + \xi^{n - 1}_h} = \\
2 \ainnerproduct{\varepsilon^{-1} \left( \eta^n - \eta^{n - 1} \right)}{\eta^n_h + \eta^{n - 1}_h} \, + 2 \ainnerproduct{\varepsilon \left( \zeta^n - \zeta^{n - 1} \right)}{\zeta^n_h + \zeta^{n - 1}_h} + 2 \ainnerproduct{\mu \left( \xi^n - \xi^{n - 1} \right)}{\xi^n_h + \xi^{n - 1}_h} \, + \\
\Delta t \ainnerproduct{- R_p^n + \varepsilon r_E^n}{\varepsilon^{-1} \left( \eta^n_h + \eta^{n - 1}_h \right)} + \Delta t \ainnerproduct{-\varepsilon R_E^n - \nabla r_p^n + r_H^n}{\zeta^n_h + \zeta^{n - 1}_h} + \Delta t \ainnerproduct{-\mu R_H^n - \nabla \times r_E^n}{\xi^n_h + \xi^{n - 1}_h}, \label{eqn:suberror_p+E+H_cn}
\end{multline}
Consider that  $\varepsilon^{-1} \left(\eta^n - \eta^{n-1}\right) = \varepsilon^{-1} \left(I - \Pi_h^0\right) \left( p(t^n) - p(t^{n-1})\right)$ by Equation~\eqref{eqn:p_fullerror_sub_cn}. Using the Taylor theorem with remainder as in Theorem~\ref{thm:dscrt_error_estmt_cn}, applying the Cauchy-Schwarz, AM-GM, and Triangle inequalities, we have the following resulting inequality:
\begin{multline*}
  \ainnerproduct{\varepsilon^{-1} \left( \eta^n - \eta^{n - 1} \right)}{\eta^n_h + \eta^{n - 1}_h} = \Delta t \ainnerproduct{\varepsilon^{-1} \left( I - \Pi_h^0 \right) \dfrac{\partial p}{\partial t} (t^{n - \frac{1}{2}}) }{\eta^n_h + \eta^{n - 1}_h} \, + \dfrac{\Delta t}{2} \ainnerproduct{\varepsilon^{-1} \left( I - \Pi_h^0 \right) R_p^n}{\eta^n_h + \eta^{n - 1}_h} \\ 
  \le \dfrac{\Delta t}{2} \bigg[ \norm[\bigg]{(I - \Pi_h^0) \dfrac{\partial p}{\partial t}(t^{n - \frac{1}{2}})}^2_{\varepsilon^{-1}} \!\! + \norm[\bigg]{(I - \Pi_h^0) R_p^n}^2_{\varepsilon^{-1}} \bigg] \! + 2 \Delta t \bigg[ \norm{\eta_h^n}^2_{\varepsilon^{-1}} + \norm{\eta_h^{n - 1}}^2_{\varepsilon^{-1}} \bigg].
\end{multline*}
Similarly, using Equations~\labelcref{eqn:E_fullerror_sub_cn,eqn:H_fullerror_sub_cn} for the error terms for $E$ and $H$, we obtain:
\begin{gather*}
  \ainnerproduct{\varepsilon \left( \zeta^n - \zeta^{n - 1} \right)}{\zeta^n_h + \zeta^{n - 1}_h} \le \dfrac{\Delta t }{2} \bigg[ \norm[\bigg]{(I - \Pi_h^1) \dfrac{\partial E}{\partial t}(t^{n - \frac{1}{2}})}^2_{\varepsilon} \!\! + \norm[\bigg]{(I - \Pi_h^1) R_E^n}^2_{\varepsilon} \bigg] \! + 2 \Delta t \bigg[ \norm{\zeta_h^n}^2_{\varepsilon} + \norm{\zeta_h^{n - 1}}^2_{\varepsilon} \bigg], \\ 
\ainnerproduct{\mu \left( \xi^n - \xi^{n - 1} \right)}{\xi^n_h + \xi^{n - 1}_h} \le \dfrac{\Delta t }{2} \bigg[ \norm[\bigg]{(I - \Pi_h^2) \dfrac{\partial H}{\partial t}(t^{n - \frac{1}{2}})}^2_\mu \!\!  + \norm[\bigg]{(I - \Pi_h^2) R_H^n}^2_{\mu} \bigg] \! + 2 \Delta t \bigg[ \norm{\xi_h^n}^2_{\mu} + \norm{\xi_h^{n - 1}}^2_{\mu} \bigg].
\end{gather*}
Using these inequalities for $\eta$, $\zeta$ and $\xi$ in Equation~\eqref{eqn:suberror_p+E+H_cn}, we thus obtain the following estimate:
\begin{multline*}
2 \bigg[ \norm{\eta_h^n}^2_{\varepsilon^{-1}} - \norm{\eta_h^{n - 1}}^2_{\varepsilon^{-1}} + \norm{\zeta_h^n}^2_{\varepsilon} - \norm{\zeta_h^{n - 1}}^2_{\varepsilon} + \norm{\xi_h^n}^2_{\mu} - \norm{\xi_h^{n - 1}}^2_{\mu} \bigg] \le \\
\Delta t \bigg[ \norm[\bigg]{(I - \Pi_h^0) \dfrac{\partial p}{\partial t}(t^{n - \frac{1}{2}})}^2_{\varepsilon^{-1}} + \norm[\bigg]{(I - \Pi_h^0) R_p^n}^2_{\varepsilon^{-1}} + \norm[\bigg]{(I - \Pi_h^1) \dfrac{\partial E}{\partial t}(t^{n - \frac{1}{2}})}^2_{\varepsilon}  + \norm[\bigg]{(I - \Pi_h^1) R_E^n}^2_{\varepsilon} + \\
  \norm[\bigg]{(I - \Pi_h^2) \dfrac{\partial H}{\partial t}(t^{n-\frac{1}{2}})}^2_\mu + \norm[\bigg]{(I - \Pi_h^2) R_H^n}^2_\mu + \norm{R_p^n}^2_{\varepsilon^{-1}} + \norm{R_E^n}^2_{\varepsilon} + \norm{R_H^n}^2_\mu + \norm{\nabla r_p^n} ^2_{\varepsilon^{-1}} + \norm{r_E^n}^2_{\varepsilon} + \\
 \varepsilon^{-1} \mu^{-1} \norm{\nabla \times r_E^n}^2_{\varepsilon} +  \varepsilon^{-1} \mu^{-1} \norm{r_H^n}^2_{\mu} + 3 \big( \norm{\eta_h^n}^2_{\varepsilon^{-1}} + \norm{\eta_h^{n - 1}}^2_{\varepsilon^{-1}} + \norm{\zeta_h^n}^2_{\varepsilon} + \norm{\zeta_h^{n - 1}}^2_{\varepsilon} + \norm{\xi_h^n}^2_\mu + \norm{\xi_h^{n - 1}}^2_\mu \big) \bigg].
\end{multline*}
Summing from $n = 1$ to $N$, we get:
\begin{multline*}
  \norm{\eta_h^N}^2_{\varepsilon^{-1}} + \norm{\zeta_h^N}^2_{\varepsilon} + \norm{\xi_h^N}^2_{\mu} \le 3 \Delta t \sum\limits_{n = 0}^{N} \bigg[ \norm{\eta_h^n}^2_{\varepsilon^{-1}} +\norm{\zeta_h^n}^2_{\varepsilon} + \norm{\xi_h^n}^2_{\mu} \bigg] + \Delta t \sum\limits_{n = 0}^{N} \bigg[ \norm[\bigg]{(I - \Pi_h^0) \dfrac{\partial p}{\partial t}(t^{n - \frac{1}{2}})}^2_{\varepsilon^{-1}} + \\
  \norm[\bigg]{(I - \Pi_h^0) R_p^n}^2_{\varepsilon^{-1}} + \norm[\bigg]{(I - \Pi_h^1) \dfrac{\partial E}{\partial t}(t^{n - \frac{1}{2}})}^2_{\varepsilon} + \norm[\bigg]{(I - \Pi_h^1) R_E^n}^2_{\varepsilon} + \norm[\bigg]{(I - \Pi_h^2) \dfrac{\partial H}{\partial t}(t^{n - \frac{1}{2}})}^2_\mu + \norm[\bigg]{(I - \Pi_h^2) R_H^n}^2_\mu + \\
  \norm{R_p^n}^2_{\varepsilon^{-1}} + \norm{R_E^n}^2_{\varepsilon} + \norm{R_H^n}^2_{\mu} + \norm{\nabla r_p^n} ^2_{\varepsilon^{-1}} + \norm{r_E^n}^2_{\varepsilon} + \varepsilon^{-1} \mu^{-1} \norm{\nabla \times r_E^n}^2_{\varepsilon} +  \varepsilon^{-1} \mu^{-1} \norm{r_H^n}^2_{\mu} \bigg] + \bigg[ \norm{\eta_h^0}^2_{\varepsilon^{-1}} + \norm{\zeta_h^0}^2_{\varepsilon} + \norm{\xi_h^0}^2_{\mu} \bigg].
\end{multline*}
We next apply the discrete Gronwall inequality by setting $\delta \coloneq \Delta t$, $g_0 \coloneq \norm{\eta_h^0}^2_{\varepsilon^{-1}} + \norm{\zeta_h^0}^2_{\varepsilon} + \norm{\xi_h^0}^2_{\mu}$, $a_n \coloneq \norm{\eta_h^n}^2_{\varepsilon^{-1}} + \norm{\zeta_h^n}^2_{\varepsilon} + \norm{\xi_h^n}^2_{\mu}$, $b_n \coloneq 0$, $c_n \coloneq \norm{(I - \Pi_h^0) \dfrac{\partial p}{\partial t}(t^{n - \frac{1}{2}})}^2_{\varepsilon^{-1}} + \norm{(I - \Pi_h^0) R_p^n}^2_{\varepsilon^{-1}} + \norm{(I - \Pi_h^1) \dfrac{\partial E}{\partial t}(t^{n - \frac{1}{2}})}^2_{\varepsilon} + \norm{(I - \Pi_h^1) R_E^n}^2_{\varepsilon} + \norm{(I - \Pi_h^2) \dfrac{\partial H}{\partial t}(t^{n-\frac{1}{2}})}^2_\mu + \norm{(I - \Pi_h^2) R_H^n}^2_\mu + \norm{R_p^n}^2_{\varepsilon^{-1}} + \norm{R_E^n}^2_{\varepsilon} + \norm{R_H^n}^2_{\mu}+ \norm{\nabla r_p^n} ^2_{\varepsilon^{-1}} + \norm{r_E^n}^2_{\varepsilon} + \varepsilon^{-1} \mu^{-1} \norm{\nabla \times r_E^n}^2_{\varepsilon} + \varepsilon^{-1} \mu^{-1} \norm{r_H^n}^2_{\mu}$, and $\gamma_n \coloneq 3$, and with $\sigma_n = \left(1 - 3 \Delta t \right)^{-1}$. Note that for the condition $\gamma_n \delta < 1$ to hold, we that $\Delta t < 1/3$, and using that $(N + 1)\Delta t = T + \Delta t \le T + 1/6$ for $\Delta t < 1/6$, we get:
\begin{multline*}
  \norm{\eta_h^N}^2_{\varepsilon^{-1}} + \norm{\zeta_h^N}^2_{\varepsilon} + \norm{\xi_h^N}^2_{\mu} \le \Bigg[\Delta t \sum\limits_{n = 0}^{N} \Bigg[ \norm[\bigg]{(I - \Pi_h^0) \dfrac{\partial p}{\partial t}(t^{n - \frac{1}{2}})}^2_{\varepsilon^{-1}} + \norm[\bigg]{(I - \Pi_h^0) R_p^n}^2_{\varepsilon^{-1}} + \\
  \norm[\bigg]{(I - \Pi_h^1) \dfrac{\partial E}{\partial t}(t^{n - \frac{1}{2}})}^2_{\varepsilon} + \norm[\bigg]{(I - \Pi_h^1) R_E^n}^2_{\varepsilon} + \norm[\bigg]{(I - \Pi_h^2) \dfrac{\partial H}{\partial t}(t^{n - \frac{1}{2}})}^2_\mu + \norm[\bigg]{(I - \Pi_h^2) R_H^n}^2_\mu + \norm{R_p^n}^2_{\varepsilon^{-1}} + \norm{R_E^n}^2_{\varepsilon} + \\
 \norm{R_H^n}^2_{\mu} + \norm{\nabla r_p^n} ^2_{\varepsilon^{-1}} + \norm{r_E^n}^2_{\varepsilon} + \varepsilon^{-1} \mu^{-1} \norm{\nabla \times r_E^n}^2_{\varepsilon} +  \varepsilon^{-1} \mu^{-1} \norm{r_H^n}^2_{\mu}\Bigg] + \left[ \norm{\eta_h^0}^2_{\varepsilon^{-1}} + \norm{\zeta_h^0}^2_{\varepsilon} + \norm{\xi_h^0}^2_{\mu} \right] \Bigg] \exp\left( 6 T + 1 \right).
\end{multline*}
Using our estimates for the Taylor remainders from Equations~\labelcref{eqn:Remainder_norm_p_cn,eqn:Remainder_norms_pEH_cn} for these terms on the right hand side of the above inequality, we further get that:
\[
  \Delta t \sum\limits_{n = 0}^N \Big[ \norm{R^n_p}^2_{\varepsilon^{-1}} + \norm{R^n_E}^2_{\varepsilon} + \norm{R^n_H}^2_\mu \Big] \le \dfrac{(\Delta t)^4}{5} \Bigg[ \norm[\bigg]{\dfrac{\partial^3 p}{\partial t^3}}^2_{L^2[0, T] \times L^2_{\varepsilon^{-1}}(\Omega)} \!\! + \, \norm[\bigg]{\dfrac{\partial^3 E}{\partial t^3}}^2_{L^2[0, T] \times L^2_\varepsilon(\Omega)} \!\! + \, \norm[\bigg]{\dfrac{\partial^3 H}{\partial t^3}}^2_{L^2[0, T] \times L^2_\mu(\Omega)} \Bigg],
\]
and that:
\begin{multline*}
  \Delta t \sum\limits_{n = 0}^N \Big[ \norm{\nabla r^n_p}^2_{\varepsilon^{-1}} + \norm{r^n_E}^2_{\varepsilon} + \varepsilon^{-1} \mu^{-1} \norm{\nabla \times r^n_E}^2_{\varepsilon} +  \varepsilon^{-1} \mu^{-1} \norm{r^n_H}^2_\mu \Big] \le \dfrac{(\Delta t)^4}{12} \Bigg[ \norm[\bigg]{\dfrac{\partial^2 \left(\nabla p \right)}{\partial t^2}}^2_{L^2[0, T] \times \mathring{H}^1_{\varepsilon^{-1}}(\Omega)} \!\! + \\ 
  \norm[\bigg]{\dfrac{\partial^2 E}{\partial t^2}}^2_{L^2[0, T] \times L^2_\varepsilon(\Omega)} \!\! + \varepsilon^{-1} \mu^{-1} \norm[\bigg]{\dfrac{\partial^2 (\nabla \times E)}{\partial t^2}}^2_{L^2[0, T] \times \mathring{H}_\varepsilon(\curl; \Omega)} \!\! + \varepsilon^{-1} \mu^{-1} \norm[\bigg]{\dfrac{\partial^2 H}{\partial t^2}}^2_{L^2[0, T] \times L^2_\mu(\Omega)} \Bigg].
\end{multline*}
Now, for $p \in \mathring{H}_{\varepsilon^{-1}}^1(\Omega)$, $E \in \mathring{H}_\varepsilon(\curl; \Omega)$ and $H \in \mathring{H}_\mu(\divgn; \Omega)$, there exists positive bounded constants $C_{1, p}$, $C_{2, p}$, $C_{1, E}$, $C_{2, E}$, $C_{1, H}$, and $C_{2, H}$ such that we have the following error bounds for the $L^2$ projections:
\begin{alignat*}{3}
  \norm[\bigg]{(I - \Pi_h^0) \dfrac{\partial p}{\partial t}(t^{n - \frac{1}{2}})}_{\varepsilon^{-1}} &\le C_{1, p} h^r \norm[\bigg]{\dfrac{\partial p}{\partial t}(t^{n - \frac{1}{2}})}_{\varepsilon^{-1}}, && \qquad \norm[\bigg]{(I - \Pi_h^0) R_p^n}_{\varepsilon^{-1}} &&\le C_{2, p} h^r \norm[\bigg]{R_p^n}_{\varepsilon^{-1}}, \\
  \norm[\bigg]{(I - \Pi_h^1) \dfrac{\partial E}{\partial t}(t^{n - \frac{1}{2}})}_{\varepsilon} &\le C_{1, E} h^r \norm[\bigg]{\dfrac{\partial E}{\partial t}(t^{n - \frac{1}{2}})}_{\varepsilon}, && \qquad \norm[\bigg]{(I - \Pi_h^1) R_E^n}_{\varepsilon} &&\le C_{2, E} h^r \norm[\bigg]{R_E^n}_{\varepsilon}, \\
  \norm[\bigg]{(I - \Pi_h^2) \dfrac{\partial H}{\partial t}(t^{n - \frac{1}{2}})}_{\mu} &\le C_{1, H} h^r \norm[\bigg]{\dfrac{\partial H}{\partial t}(t^{n - \frac{1}{2}})}_{\mu}, && \qquad \norm[\bigg]{(I - \Pi_h^2) R_H^n}_{\mu} &&\le C_{2, H} h^r \norm[\bigg]{R_H^n}_{\mu}.
\end{alignat*}
Set $C_0 \coloneq \max \{C_{1, p}, C_{2, p}, C_{1, E}, C_{2, E}, C_{1, H}, C_{2, H} \}$. We therefore have that:
 \begin{multline*}
    \Delta t \sum\limits_{n = 0}^N \left[ \norm[\bigg]{(I - \Pi_h^0) \dfrac{\partial p}{\partial t}(t^{n - \frac{1}{2}})}^2_{\varepsilon^{-1}} \!\! + \norm[\bigg]{(I - \Pi_h^1) \dfrac{\partial E}{\partial t}(t^{n - \frac{1}{2}})}^2_{\varepsilon} \!\! + \norm[\bigg]{(I - \Pi_h^2) \dfrac{\partial H}{\partial t}(t^{n - \frac{1}{2}})}^2_\mu \right], \\
\le C_0 h^{2 r} \sum\limits_{n = 0}^N \Delta t \left[ \norm[\bigg]{\dfrac{\partial p}{\partial t}(t^{n - \frac{1}{2}})}^2_{\varepsilon^{-1}} \!\! + \norm[\bigg]{\dfrac{\partial E}{\partial t}(t^{n - \frac{1}{2}})}^2_{\varepsilon} \!\! + \norm[\bigg]{\dfrac{\partial H}{\partial t}(t^{n - \frac{1}{2}})}^2_{\mu} \right], \\
\le C_0 h^{2 r} \int\limits_0^T \left[ \norm[\bigg]{\dfrac{\partial p}{\partial t}(t^{n - \frac{1}{2}})}^2_{\varepsilon^{-1}} \!\! + \norm[\bigg]{\dfrac{\partial E}{\partial t}(t^{n - \frac{1}{2}})}^2_{\varepsilon} \!\! + \norm[\bigg]{\dfrac{\partial H}{\partial t}(t^{n - \frac{1}{2}})}^2_{\mu} \right] dt, \\
= C_0 h^{2 r} \left[ \norm[\bigg]{\dfrac{\partial p}{\partial t}}^2_{L^2[0, T] \times L^2_{\varepsilon^{-1}}(\Omega)} \!\! + \norm[\bigg]{\dfrac{\partial E}{\partial t}}^2_{L^2[0, T] \times L^2_{\varepsilon}(\Omega)} \!\! + \norm[\bigg]{\dfrac{\partial H}{\partial t}}^2_{L^2[0, T] \times L^2_{\mu}(\Omega)} \right],
\end{multline*}
and likewise for the $L^2$ projection of the remainder terms:
\begin{multline*}
  \Delta t \sum\limits_{n = 0}^N \left[ \norm[\bigg]{(I - \Pi_h^0) R_p^n}_{\varepsilon^{-1}} \!\! + \norm[\bigg]{(I - \Pi_h^1) R_E^n}_{\varepsilon} \!\! + \norm[\bigg]{(I - \Pi_h^2) R_H^n}_{\mu} \right], \\
  \le C_0 h^{2 r} \sum\limits_{n = 0}^N  \Delta t \left[ \norm{R_p^n}_{L^2_{\varepsilon^{-1}}(\Omega)} + \norm{R_E^n}_{L^2_{\varepsilon}(\Omega)} + \norm{R_H^n}_{L^2_{\mu}(\Omega)} \right], \\
  \le C_0 h^{2 r} \dfrac{(\Delta t)^4}{5} \left[ \norm[\bigg]{\dfrac{\partial^3 p}{\partial t^3}}^2_{L^2[0, T] \times L^2_{\varepsilon^{-1}}(\Omega)} \!\! + \norm[\bigg]{\dfrac{\partial^3 E}{\partial t^3}}^2_{L^2[0, T] \times L^2_{\varepsilon}(\Omega)} \!\! + \norm[\bigg]{\dfrac{\partial^3 H}{\partial t^3}}^2_{L^2[0, T] \times L^2_{\mu}(\Omega)} \right].
\end{multline*}
Now, we take $p_h^0 = \Pi_h^0 p_0$, $E_h^0 = \Pi_h^1 E_0$, and $H_h^0 = \Pi_h^2 H_0$, and note that there exists positive bounded constants $M_1, \, M_2$ and $M_3$ such that:
\begin{align*}
  \norm[\bigg]{\dfrac{\partial p}{\partial t}}^2_{L^2[0, T] \times L^2_{\varepsilon^{-1}}(\Omega)} \!\! + \norm[\bigg]{\dfrac{\partial E}{\partial t}}^2_{L^2[0, T] \times L^2_{\varepsilon}(\Omega)} \!\! + \norm[\bigg]{\dfrac{\partial H}{\partial t}}^2_{L^2[0, T] \times L^2_{\mu}(\Omega)} &\le M_1, \\
  \norm[\bigg]{\dfrac{\partial^3 p}{\partial t^3}}^2_{L^2[0, T] \times L^2_{\varepsilon^{-1}}(\Omega)} \!\! + \norm[\bigg]{\dfrac{\partial^3 E}{\partial t^3}}^2_{L^2[0, T] \times L^2_{\varepsilon}(\Omega)} \!\! + \norm[\bigg]{\dfrac{\partial^3 H}{\partial t^3}}^2_{L^2[0, T] \times L^2_{\mu}(\Omega)} &\le M_2,
\end{align*}
\vspace{-1em} \begin{multline*}
\norm[\bigg]{\dfrac{\partial^2 \left(\nabla p \right)}{\partial t^2}}^2_{L^2[0, T] \times \mathring{H}^1_{\varepsilon^{-1}}(\Omega)} \!\! + \norm[\bigg]{\dfrac{\partial^2 E}{\partial t^2}}^2_{L^2[0, T] \times L^2_\varepsilon(\Omega)} \!\! + \\
\varepsilon^{-1} \mu^{-1} \norm[\bigg]{\dfrac{\partial^2 (\nabla \times E)}{\partial t^2}}^2_{L^2[0, T] \times \mathring{H}_\varepsilon(\curl; \Omega)} \!\! + \varepsilon^{-1} \mu^{-1} \norm[\bigg]{\dfrac{\partial^2 H}{\partial t^2}}^2_{L^2[0, T] \times L^2_\mu(\Omega)} \leq M_3.
\end{multline*}
Consequently, we have the following estimate:
\[
  \norm{\eta_h^N}^2_{\varepsilon^{-1}} + \norm{\zeta_h^N}^2_{\varepsilon} + \norm{\xi_h^N}^2_{\mu} \le \widetilde{C} \left[h^{2 r} + h^{2 r} (\Delta t)^4 + (\Delta t)^4 \right],
\]
where $\widetilde{C} = \max\{C_0 M_1, C_0 M_2/5, M_2/5, M_3/12 \} \exp(6 T +1)$ which then using the equivalence between $1$- and $2$-norms gives us that:
\[
  \norm{\eta_h^N}_{\varepsilon^{-1}} + \norm{\zeta_h^N}_{\varepsilon} + \norm{\xi_h^N}_{\mu} \le C_1 \left[h^{r} + h^r (\Delta t)^2 + (\Delta t)^2 \right],
\]
where $C_1 = \sqrt{3 \widetilde{C}}$. Also, there are positive bounded constants $\widetilde{C_2}$, $\widetilde{C_3}$ and $\widetilde{C_4}$ such that:
\begin{alignat*}{4}
  \norm{\eta^N}_{\varepsilon^{-1}} &= \norm{(I - \Pi_h^0) p(t^N)}_{\varepsilon^{-1}} &&\le \widetilde{C_2} h^r \norm{p(t^N)}_{L^2_{\varepsilon^{-1}}(\Omega)} &&\implies \norm{\eta^N}_{\varepsilon^{-1}} &&\le C_2 h^r, \\
  \norm{\zeta^N}_{\varepsilon} &= \norm{(I - \Pi_h^1) E(t^N)}_{\varepsilon} &&\le \widetilde{C_3} h^r \norm{E(t^N)}_{L^2_{\varepsilon}(\Omega)} &&\implies \norm{\zeta^N}_{\varepsilon} &&\le C_3 h^r, \\
  \norm{\xi^N}_\mu &= \norm{(I - \Pi_h^2) H(t^N)}_\mu &&\le \widetilde{C_4} h^r \norm{H(t^N)}_{L^2_{\mu}(\Omega)} &&\implies \norm{\xi^N}_\mu &&\le C_4 h^r,
\end{alignat*}
in which $C_2 = \widetilde{C_2} \norm{p(t^N)}_{L^2_{\varepsilon^{-1}}(\Omega)}$, $C_3 = \widetilde{C_3} \norm{E(t^N)}_{L^2_{\varepsilon}(\Omega)}$ and $C_4 = \widetilde{C_4} \norm{H(t^N)}_{L^2_{\mu}(\Omega)}$ are all bounded positive constants due to Theorem~\ref{thm:dscrt_enrgy_estmt_cn}. Finally, this provides us with our desired result by choosing $C = C_1 + C_2 + C_3 + C_4$:
\begin{align*}
  \norm{e_{p_h}^n}_{\varepsilon^{-1}} + \norm{e_E^{N,h}}_{\varepsilon} + \norm{e_H^{N,h}}_{\mu} &= \norm{\eta^N - \eta^N_h}_{\varepsilon^{-1}} + \norm{\zeta^N - \zeta^N_h}_{\varepsilon} + \norm{\xi^N - \xi^N_h}_\mu \\
&\le C \left[(\Delta t)^2 + h^r + h^r (\Delta t)^2 \right]. \qedhere
\end{align*}
\end{proof}

\subsection{Implicit Leapfrog Scheme}

For the implicit leapfrog scheme as in Equations~\labelcref{eqn:maxwell_p_lf,eqn:maxwell_E_lf,eqn:maxwell_H_lf}, using a de Rham sequence of finite dimensional subspaces of the corresponding function spaces for the spatial discretization of $(p^{n + \frac{1}{2}}, E^{n + \frac{1}{2}}, H^{n + 1})$, we obtain the following discrete problem: find $(p_h^{n + \frac{1}{2}}, E_h^{n + \frac{1}{2}}, H_h^{n + 1}) \in U_h \times V_h \times W_h \subseteq \mathring{H}_{\varepsilon^{-1}}^1 \times \mathring{H}_{\varepsilon}(\curl; \Omega) \times \mathring{H}_{\mu}(\divgn; \Omega)$ such that:
\begin{subequations}
\begin{align}
  \aInnerproduct{\dfrac{p_h^{n + \frac{1}{2}} - p_h^{n - \frac{1}{2}}}{\Delta t}}{\widetilde{p}} - \aInnerproduct{\dfrac{\varepsilon}{2} \left( E_h^{n + \frac{1}{2}} + E_h^{n - \frac{1}{2}} \right)}{\nabla \widetilde{p}} &= \aInnerproduct{ f_p^{n}}{\widetilde{p}}, \label{eqn:maxwell_p_lf_full} \\
  \aInnerproduct{\dfrac{1}{2}\nabla \left(p_h^{n + \frac{1}{2}} + p_h^{n - \frac{1}{2}} \right)}{\widetilde{E}} + \aInnerproduct{\varepsilon \dfrac{E_h^{n + \frac{1}{2}} - E_h^{n - \frac{1}{2}}}{\Delta t}}{\widetilde{E}} - \aInnerproduct{\dfrac{1}{2} \left( H_h^{n + 1} + H_h^n \right)}{\nabla \times \widetilde{E}} &= \aInnerproduct{f_E^n}{\widetilde{E}}, \label{eqn:maxwell_E_lf_full} \\
  \aInnerproduct{\mu \dfrac{H_h^{n + 1} - H_h^n}{\Delta t}}{\widetilde{H}} +  \aInnerproduct{\dfrac{1}{2}  \nabla \times  \left( E_h^{n + \frac{1}{2}} +E_h^{n - \frac{1}{2}} \right)}{\widetilde{H}} &= \aInnerproduct{f_H^{n + \frac{1}{2}}}{\widetilde{H}}, \label{eqn:maxwell_H_lf_full}
\end{align}
\end{subequations}
for all $(\widetilde{p}, \widetilde{E}, \widetilde{H}) \in U_h \times V_h \times W_h$, and for $n \in \{1, \dotso, N - 1\}$. 
The $n = 0$ bootstrapping as in Equations~\labelcref{eqn:maxwell_p0_lf,eqn:maxwell_E0_lf,eqn:maxwell_H0_lf} leads to the discrete problem: find $(p_h^{\frac{1}{2}}, E_h^{\frac{1}{2}}, H_h^1) \in U_h \times V_h \times W_h \subseteq \mathring{H}_{\varepsilon^{-1}}^1 \times \mathring{H}_{\varepsilon}(\curl; \Omega) \times \mathring{H}_{\mu}(\divgn; \Omega)$ such that:
\begin{subequations}
  \begin{align}
    \aInnerproduct{\dfrac{p_h^{\frac{1}{2}} - p_h^0}{\Delta t/2}}{\widetilde{p}} -\dfrac{1}{4} \aInnerproduct{ \varepsilon \left(  E_h^{\frac{1}{2}} + E_h^0 \right)}{\nabla \widetilde{p}} & = \aInnerproduct{ f_p^0}{\widetilde{p}}, \label{eqn:maxwell_p0_lf_full} \\
\dfrac{1}{4}  \aInnerproduct{\nabla \left( p_h^{\frac{1}{2}} + p_h^0 \right)}{\widetilde{E}} + \aInnerproduct{\varepsilon \dfrac{E_h^{\frac{1}{2}} - E_h^0}{\Delta t/2}}{\widetilde{E}} - \aInnerproduct{\dfrac{1}{2} \left( H_h^1 + H_h^0 \right)}{\nabla \times \widetilde{E}} &= \aInnerproduct{f_E^0}{\widetilde{E}}, \label{eqn:maxwell_E0_lf_full} \\
  \aInnerproduct{\mu \dfrac{H_h^1 - H_h^0}{\Delta t}}{\widetilde{H}} +  \dfrac{1}{4} \aInnerproduct{\nabla \times \left( E_h^{\frac{1}{2}} + E_h^0 \right)}{\widetilde{H}} &= \aInnerproduct{f_H^{\frac{1}{2}}}{\widetilde{H}}, \label{eqn:maxwell_H0_lf_full}
\end{align}
\end{subequations}
for all $(\widetilde{p}, \widetilde{E}, \widetilde{H}) \in U_h \times V_h \times W_h$ given $(p_h^0, E_h^0, H_h^0) \in U_h \times V_h \times W_h$. 

Using the same projection operators $\Pi_h^0, \; \Pi_h^1$ and  $\Pi_h^2$ for $p$, $E$, and $H$, respectively as described in Section~\ref{sec:cn_full_error} for the error estimate there, we can again define the errors for $p$, $E$ and $H$ here as follows:
\begin{alignat}{2}
  e_{p_h}^{n + \frac{1}{2}} &\coloneq p(t^{n + \frac{1}{2}}) - p_h^{n + \frac{1}{2}} &&= \eta^{n+\frac{1}{2}} - \eta_h^{n+\frac{1}{2}}, \label{eqn:p_fullerror_lf} \\
  e_{E_h}^{n + \frac{1}{2}} &\coloneq E(t^{n + \frac{1}{2}}) - E_h^{n + \frac{1}{2}} &&= \zeta^{n + \frac{1}{2}} - \zeta_h^{n + \frac{1}{2}}, \label{eqn:E_fullerror_lf} \\
  e_{H_h}^n &\coloneq H(t^n) - H_h^n &&= \xi^n - \xi_h^n, \label{eqn:H_fullerror_lf}
\end{alignat}
and in which we have the following definitions for the newly introduced terms:
\begin{alignat}{3}
  \eta^{n + \frac{1}{2}} &\coloneq p(t^{n + \frac{1}{2}}) - \Pi_h^0 p(t^{n + \frac{1}{2}}), &&\qquad \eta_h^{n + \frac{1}{2}} &&\coloneq p_h^{n + \frac{1}{2}}  - \Pi_h^0 p(t^{n + \frac{1}{2}}), \label{eqn:p_fullerror_sub_lf} \\
  \zeta^{n + \frac{1}{2}} &\coloneq E(t^{n + \frac{1}{2}}) - \Pi_h^1 E(t^{n + \frac{1}{2}}), &&\qquad \zeta_h^{n + \frac{1}{2}} &&\coloneq E_h^{n + \frac{1}{2}} - \Pi_h^1 E(t^{n + \frac{1}{2}}), \label{eqn:E_fullerror_sub_lf} \\
  \xi^n &\coloneq H(t^n) - \Pi_h^2 H(t^n), &&\qquad \xi_h^n &&\coloneq H_h^n - \Pi_h^2 H(t^n). \label{eqn:H_fullerror_sub_lf}
\end{alignat}

\begin{theorem}[Full Error Estimate]\label{thm:full_error_estmt_lf}
Let $p \in C^3[0, T] \times \mathring{H}^1_{\varepsilon^{-1}}(\Omega)$, $E \in C^3[0, T] \times \mathring{H}_{\varepsilon}(\curl; \Omega)$, and $H \in C^3[0, T] \times \mathring{H}_{\mu}(\divgn; \Omega)$ be the solution to the variational formulation of the Maxwell's equations as in Equations~\labelcref{eqn:maxwell_p_wf,eqn:maxwell_E_wf,eqn:maxwell_H_wf}, and let $(p_h^{n + \frac{1}{2}}, E_h^{n + \frac{1}{2}}, H_h^{n + 1})$ be the solution of the fully discretized Maxwell's equations using the implicit leapfrog scheme as in Equations~\labelcref{eqn:maxwell_p_lf_full,eqn:maxwell_E_lf_full,eqn:maxwell_H_lf_full,eqn:maxwell_p0_lf_full,eqn:maxwell_E0_lf_full,eqn:maxwell_H0_lf_full}. If the fixed time step $\Delta t > 0$ and the mesh parameter $h > 0$ are sufficiently small, then there exists a positive bounded constant $C$ independent of both $\Delta t$ and $h$ such that the following error estimate holds:
\[
  \norm{e_{p_h}^{N - \frac{1}{2}}}_{\varepsilon^{-1}} + \norm{e_{E_h}^{N - \frac{1}{2}}}_{\varepsilon} + \norm{e_{H_h}^N}_{\mu} \le C \left[ (\Delta t)^2 + h^r + h^r (\Delta t)^2 \right],
\]
where the finite element subspaces $U_h$, $V_h$ and $W_h$ are each spanned by their respective Whitney form basis of polynomial order $r \ge 1$.
\end{theorem}

\begin{proof}
  As with all our proofs for theorems corresponding to the implicit leapfrog scheme, our proof here will mimic that for Theorem~\ref{thm:full_error_estimate_cn} \textit{mutatis mutandis}. Nevertheless, we shall describe it to some detail next.

  We start from the variational formulation for the error terms $\eta$, $\zeta$, $\xi$, and their discrete analogues: 
\[
  \aInnerproduct{\dfrac{\left(\eta^{n + \frac{1}{2}} - \eta^{n - \frac{1}{2}}\right) - \left( \eta^{n + \frac{1}{2}}_h - \eta^{n - \frac{1}{2}}_h \right)}{\Delta t}}{\widetilde{p}}  - \dfrac{1}{2} \aInnerproduct{ \varepsilon \left( \left( \zeta^{n + \frac{1}{2}} + \zeta^{n - \frac{1}{2}} \right) - \left(\zeta_h^{n + \frac{1}{2}} + \zeta^{n + \frac{1}{2}}_h \right) \right)}{\nabla \widetilde{p}} =  \aInnerproduct{R_p^n - \varepsilon r_E^n}{\widetilde{p}},
\]
\vspace{-1em}\begin{multline*}
  \dfrac{1}{2} \aInnerproduct{\nabla \left( \left( \eta^{n + \frac{1}{2}} + \eta^{n - \frac{1}{2}} \right) - \left( \eta^{n + \frac{1}{2}}_h + \eta^{n - \frac{1}{2}}_h \right) \right)}{\widetilde{E}} + \aInnerproduct{\varepsilon \dfrac{\left( \zeta^{n + \frac{1}{2}} - \zeta^{n - \frac{1}{2}} \right) - \left(\zeta^{n + \frac{1}{2}}_h - \zeta^{n - \frac{1}{2}}_h \right)}{\Delta t}}{\widetilde{E}} \, - \\
  \dfrac{1}{2} \aInnerproduct{\left( \left( \xi^{n + 1} + \xi^{n} \right) - \left(\xi^{n + 1}_h - \xi^{n}_h \right) \right)}{\nabla \times \widetilde{E}} = \aInnerproduct{\varepsilon R_E^n + \nabla r_p^n - r_H^{n + \frac{1}{2}}}{\widetilde{E}},
\end{multline*}
\vspace{-1em}
\[
  \aInnerproduct{\mu \dfrac{\left(\xi^{n + 1} - \xi^{n} \right) - \left( \xi^{n + 1}_h - \xi^{n}_h \right)}{\Delta t}}{\widetilde{H}}  + \dfrac{1}{2} \aInnerproduct{\nabla \times \left( \left( \zeta^{n + \frac{1}{2}} + \zeta^{n - \frac{1}{2}} \right) - \left( \zeta^{n + \frac{1}{2}}_h + \zeta^{n - \frac{1}{2}}_h \right) \right)}{\widetilde{H}} = \aInnerproduct{\mu R_H^{n + \frac{1}{2}} + \nabla \times r_E^n}{\widetilde{H}}.
\]
Now, using appropriate test functions similar to their choice as in the proof of Theorem~\ref{thm:full_error_estimate_cn}, the exactness of the discrete de Rham sequence of finite element spaces, Equations~\labelcref{eqn:p_fullerror_sub_lf,eqn:E_fullerror_sub_lf,eqn:H_fullerror_sub_lf} Taylor's theorem with remainder, Cauchy-Schwarz, AM-GM, and Triangle inequalities, we obtain the following estimates:
\begin{multline*}
  2 \ainnerproduct{\varepsilon^{-1} \left( \eta^{n + \frac{1}{2}} - \eta^{n -\frac{1}{2}} \right)}{\eta^{n + \frac{1}{2}}_h + \eta_h^{n - \frac{1}{2}}} \le \Delta t \left[ \norm[\bigg]{(I - \Pi_h^0) \dfrac{\partial p}{\partial t}(t^n)}^2_{\varepsilon^{-1}} + \norm[\bigg]{(I - \Pi_h^0) R_p^n}^2_{\varepsilon^{-1}} \right] + \\
  4 \Delta t \Bigg[ \norm{\eta_h^{n + \frac{1}{2}}}^2_{\varepsilon^{-1}} +\norm{\eta_h^{n - \frac{1}{2}}}^2_{\varepsilon^{-1}} \Bigg],
\end{multline*}
\vspace{-1em} \begin{multline*}
  2 \ainnerproduct{\varepsilon \left(\zeta^{n + \frac{1}{2}} - \zeta^{n - \frac{1}{2}} \right)}{\zeta^{n + \frac{1}{2}}_h + \zeta_h^{n - \frac{1}{2}}} \le \Delta t \left[ \norm[\bigg]{(I - \Pi_h^1) \dfrac{\partial E}{\partial t}(t^n)}^2_{\varepsilon} + \norm[\bigg]{(I - \Pi_h^1) R_E^n}^2_{\varepsilon} \right] + \\
  4 \Delta t \Bigg[ \norm{\zeta_h^{n + \frac{1}{2}}}^2_{\varepsilon} +\norm{\zeta_h^{n - \frac{1}{2}}}^2_{\varepsilon} \Bigg],
\end{multline*}
\vspace{-1em} \begin{multline*}
  2 \ainnerproduct{\mu \left(\xi^{n + 1} - \xi^{n}\right)}{\xi^{n + 1}_h + \xi^n_h} \le \Delta t  \left[ \norm[\bigg]{(I - \Pi_h^2) \dfrac{\partial H}{\partial t}(t^{n + \frac{1}{2}})}^2_\mu + \norm[\bigg]{(I - \Pi_h^2) R_H^{n + \frac{1}{2}}}^2_{\mu} \right] + \\
  4 \Delta t \Bigg[ \norm{\xi_h^{n + 1}}^2_{\mu} + \norm{\xi_h^n}^2_{\mu} \Bigg].
\end{multline*}
Similarly, we will have an estimate for the errors for the initial $p^{\frac{1}{2}}$, $E^{\frac{1}{2}}$ and $H^1$. Combining these estimates and invoking the necessary arguments as in the proof of Theorem~\ref{thm:full_error_estimate_cn}, we arrive at the following inequality:
\begin{multline*}
  \norm{\eta_h^{N - \frac{1}{2}}}^2_{\varepsilon^{-1}} + \norm{\zeta_h^{N - \frac{1}{2}}}^2_{\varepsilon} + \norm{\xi_h^N}^2_{\mu} \le \\
  \dfrac{1 + 3 \Delta t}{1 - 3 \Delta t} \Bigg[ \norm{\eta_h^0}^2_{\varepsilon^{-1}} + \norm{\zeta_h^0}^2_{\varepsilon} + \norm{\xi_h^0}^2_{\mu} \Bigg] + \dfrac{6 \Delta t}{1 - 3\Delta t} \sum\limits_{n = 0}^{N - 1} \Bigg[ \norm{\eta_h^{n + \frac{1}{2}}}^2_{\varepsilon^{-1}} + \norm{\zeta_h^{n + \frac{1}{2}}}^2_{\varepsilon} +  \norm{\xi_h^{n + 1}}^2_{\mu} \Bigg] + \\
  \dfrac{ \Delta t}{1 - 3 \Delta t} \sum\limits_{n = 0}^{N - 1} \Bigg[ \norm[\bigg]{(I - \Pi_h^0) \dfrac{\partial p}{\partial t}(t^n)}^2_{\varepsilon^{-1}} + \norm[\bigg]{(I - \Pi_h^0) R_p^n}^2_{\varepsilon^{-1}} + \norm[\bigg]{(I - \Pi_h^1) \dfrac{\partial E}{\partial t}(t^n)}^2_{\varepsilon} + \\
    \norm[\bigg]{(I - \Pi_h^1) R_E^n}^2_{\varepsilon} + \norm[\bigg]{(I - \Pi_h^2) \dfrac{\partial H}{\partial t}(t^{n + \frac{1}{2}})}^2_\mu + \norm[\bigg]{(I - \Pi_h^2) R_H^{n + \frac{1}{2}}}^2_{\mu} + \norm{R_p^n}^2_{\varepsilon^{-1}} + \norm{R_E^n}^2_{\varepsilon} + \norm{R_H^{n + \frac{1}{2}}}^2_{\mu} + \\
\norm{\nabla r_p^n} ^2_{\varepsilon^{-1}} + \norm{r_E^n}^2_{\varepsilon} + \varepsilon^{-1} \mu^{-1} \norm{\nabla \times r_E^n}^2_{\varepsilon} +  \varepsilon^{-1} \mu^{-1} \norm{r_H^{n + \frac{1}{2}}}^2_{\mu} \Bigg].
\end{multline*}
Now, applying the discrete Gronwall inequality with $\Delta t < 1/18$, we get that:
\begin{multline*}
  \norm{\eta_h^{N - \frac{1}{2}}}^2_{\varepsilon^{-1}} + \norm{\zeta_h^{N - \frac{1}{2}}}^2_{\varepsilon} + \norm{\xi_h^N}^2_{\mu} \le \Bigg[ \dfrac{6 \Delta t}{5} \sum\limits_{n = 0}^{N - 1} \Bigg( \norm[\bigg]{(I - \Pi_h^0) \dfrac{\partial p}{\partial t}(t^n)}^2_{\varepsilon^{-1}} + \norm[\bigg]{(I - \Pi_h^0) R_p^n}^2_{\varepsilon^{-1}} + \\
  \norm[\bigg]{(I - \Pi_h^1) \dfrac{\partial E}{\partial t}(t^n)}^2_{\varepsilon}  + \norm[\bigg]{(I - \Pi_h^1) R_E^n}^2_{\varepsilon} + \norm[\bigg]{(I - \Pi_h^2) \dfrac{\partial H}{\partial t}(t^{n + \frac{1}{2}})}^2_\mu + \norm[\bigg]{(I - \Pi_h^2) R_H^{n + \frac{1}{2}}}^2_{\mu} + \\
  \norm{R_p^n}^2_{\varepsilon^{-1}} + \norm{R_E^n}^2_{\varepsilon} + \norm{R_H^{n + \frac{1}{2}}}^2_{\mu} + \norm{\nabla r_p^n} ^2_{\varepsilon^{-1}} + \norm{r_E^n}^2_{\varepsilon} + \varepsilon^{-1} \mu^{-1} \norm{\nabla \times r_E^n}^2_{\varepsilon} +  \varepsilon^{-1} \mu^{-1} \norm{r_H^{n + \frac{1}{2}}}^2_{\mu}\Bigg) + \\ \dfrac{7}{5} \Big( \norm{\eta_h^0}^2_{\varepsilon^{-1}} + \norm{\zeta_h^0}^2_{\varepsilon} + \norm{\xi_h^0}^2_{\mu} \Big) \Bigg] \exp \left( 12 T \right).
\end{multline*}
Using appropriate estimates for the right hand side terms in this equation as in the proof of Theorem~\ref{thm:full_error_estimate_cn}, and by setting $p_h^0 \coloneq \Pi_h^0 p_0$, $E_h^0 \coloneq \Pi_h^1 E_0$, and $H_h^0 \coloneq \Pi_h^2 \mathring{H}$, we finally obtain our required result:
\begin{align*}
  \norm{e_{p_h}^{N - \frac{1}{2}}}_{\varepsilon^{-1}} + \norm{e_{E_h}^{N - \frac{1}{2}}}_{\varepsilon} + \norm{e_{H_h}^{N}}_{\mu} &= \norm{\eta^{N - \frac{1}{2}} - \eta^{N - \frac{1}{2}}_h}_{\varepsilon^{-1}} + \norm{\zeta^{N - \frac{1}{2}} - \zeta^{N - \frac{1}{2}}_h}_{\varepsilon} + \norm{\xi^N - \xi^N_h}_\mu \\
  &\le C \left[ (\Delta t )^2 + h^r + h^r (\Delta t)^2 \right]. \qedhere
\end{align*}
\end{proof}

\subsection{Nonhomogeneous Boundary Conditions}

We wish to close our theoretical discussion by reiterating that we have shown all our stability and error estimates for the homogenous boundary conditions as in Equation~\eqref{eqn:BCs} for the system of Maxwell's equations in Equation~\eqref{eqn:maxwells_eqns}. However, standard arguments can be used if these boundary conditions are nonhomogeneous and all our proofs can be suitably updated to obtain essentially the same stability and error convergence results for the full discretization using both time integration schemes. Consequently, in our next section, we do indeed demonstrate computed solutions that well approximate the true solutions and have the discrete energy conservation property, and these are for model problems which have analytical energy conservation in $\R^2$ and $\R^3$ for both homogenous and nonhomogeneous boundary conditions.

\subsubsection{Remark on Energy Conservation}

An important caveat however needs to be stated with regard to Corollary~\labelcref{corr:smth_enrgy_cnsrvtn} and therefore also Corollaries~\labelcref{corr:dscrt_enrgy_cnsrvtn_cn,corr:dscrt_enrgy_cnsrvtn_lf}. A nonhomogeneous time varying boundary condition can act as a source of energy and drive time evolution of the pressure, electric and magnetic fields even when all right hand side forcing functions are zero in Equation~\eqref{eqn:maxwells_eqns}. Thus, some additional technical qualifications are necessary to warrant energy conservation for such cases. We do not know what this precise technical quantification is and leave it to a future work and reaffirm that we do not need it here for our results. Towards this end, we provide examples in the next section wherein the time evolution of the Maxwell's system is driven only by initial and nonhomogeneous boundary conditions but for which the energy varies over time in a bounded manner.

\section{Numerical Results}\label{sec:numerics}

We now present some empirical validation for our theoretical results for two model problems each on the unit square in $\R^2$ in Examples 1 and 2, and and on the unit cube $\R^3$ in Examples 3 and 4. Both our problem domains are discretely realized as simplicial meshes. Our unit square has $1903$ vertices and $3600$ triangles while the unit cube has $215$ vertices and $560$ tetrahedra. We performed all our computational experiments with linear and quadratic finite element spaces of Whitney forms but only provide plots of solutions for the quadratic case in $\R^2$ and for both choices in $\R^3$. For all our example problems, we provide the analytical solutions for $p$, $E$ and $H$, the physical parameters $\epsilon$ and $\mu$, and the initial and final times $T_{\min}$ and $T_{\max}$, respectively. For all these problems, $\Delta t = 0.01$. All our boundary conditions are essentially applied in our computations and the initial conditions are discretely obtained through $L^2$ projections of the analytical solutions at $t = 0$. All integrals in our computations are obtained using sufficiently high degree quadrature, and all linear systems are solved with a sparse direct solver. Our experiments are performed using code written in Python with the standard scientific Python stack consisting of \texttt{NumPy}, \texttt{SciPy} and \texttt{Matplotlib}, and also in part thanks to some utility support using the library \texttt{PyDEC}~\cite{BeHi2012}.

All our problem regions are simply connected domains, that is, with trivial relative homologies in all dimensions. Our successive examples are, in spirit, an attempt to provide an ``anti-ablation''-like study, that is, we change only one aspect of the problem while holding others fixed. We have performed computations with other general analytical functions in two and three dimensions as well and on domains with nontrivial $1$- or $1$- and $2$-homologies as the case may be in $\R^2$ and $\R^3$, respectively, but we do not present any of those here. We  believe that our presented examples suffice for our purposes of illustration of feasibility and validation of our theory.

\medskip \noindent \textbf{Example 1}: This problem consists of Maxwell's equations posed on a unit square in $\R^2$ with analytical solutions, material parameters, initial and final times as shown below:
\[
  p = 0, \quad E = %
  \begin{bNiceMatrix}
    \sin \pi y \cos \pi t \\
    \sin \pi x \cos \pi t
  \end{bNiceMatrix}, %
  \quad H = (\cos \pi y - \cos \pi x) \sin \pi t,
\]
\[\epsilon = 1, \quad \mu = 1, \quad T_{\min} = 0 , \quad T_{\max} = 2.\]
The boundary conditions for this problem are homogeneous for $p$ and $E$. $H$ here also has a zero boundary condition since the analytical functions shown here are all technically the vector calculus realizations of differential $k$-forms in $\R^2$. Therefore, for the $2$-form $H$, in particular, the boundary condition is the pullback of this $2$-form under the inclusion of the boundary into the domain and is thus vacuously zero. The results of computations performed using quadratic Whitney elements as bases for each of the finite element spaces and with the two time discretizations are shown in Figures~\labelcref{fig:example1_cn,fig:example1_lf}. We note that the homogenous boundary conditions are all imposed essentially by incorporating them into the finite element spaces. We also performed computations using the backward Euler time discretization but only provide a plot of its energy as in Figure~\ref{fig:2d_energies}.


\medskip \noindent \textbf{Example 2}: This problem is posed in $\R^3$ on a unit cube and analytical solutions, material parameters, initial and final times as below:
\[
  p = 0, \quad E = %
  \begin{bNiceMatrix}
   \sin \pi y \sin \pi z \cos \pi t \\
   \sin \pi x \sin \pi z \cos \pi t \\
   \sin \pi x \sin \pi y \cos \pi t 
  \end{bNiceMatrix}, %
  \quad H = %
   \begin{bNiceMatrix}
      \sin \pi x (\cos \pi z - \cos \pi y) \sin \pi t \\
      \sin \pi y (\cos \pi x - \cos \pi z) \sin \pi t \\
      \sin \pi z (\cos \pi y - \cos \pi x) \sin \pi t
   \end{bNiceMatrix},
\]
\[\epsilon = 2, \quad \mu = 1, \quad T_{\min} = 0 , \quad T_{\max} = 2.\]
Like in Example 1, the boundary conditions are homogeneous for $p$, $E$ and $H$ and are essentially imposed. The results of computations performed using linear and quadratic Whitney elements and with our two time discretizations are shown in Figures~\labelcref{fig:example2_cn,fig:example2_lf}. The energy computations are summarized in Figure~\ref{fig:3d_energies}.

\medskip \noindent \textbf{Example 3}: The problem here consists of nonhomogeneous boundary conditions for $p$ and $E$, and homogeneous for $H$ with analytical solutions, material parameters, initial and final times as shown below:
\[
  p = \left(\cos \pi x + \cos \pi y\right) \sin \pi t,
\]
\[
 E = %
  \begin{bNiceMatrix}
    \sin \pi (\sqrt{2} t - x - y) - \sin \pi x \cos \pi t \\
    -\sin \pi (\sqrt{2} t - x - y) - \sin \pi y \cos \pi t
  \end{bNiceMatrix}, %
  \quad H = -\sqrt{2} \sin \pi (\sqrt{2} t - x - y),
\]
\[\epsilon = 1, \quad \mu = 1, \quad T_{\min} = 0 , \quad T_{\max} = 2.\]
All boundary conditions for this finite element discretization are again imposed in an essential manner by incorporating them into the appropriate function spaces. The results of computations performed with quadratic Whitney elements are shown in Figures~\labelcref{fig:example3_cn,fig:example3_lf}, and the energies are visually summarized in Figure~\ref{fig:2d_energies}.

\medskip \noindent \textbf{Example 4}: Our next example is on the unit cube in $\R^3$ with analytical solutions, material parameters, initial and final times as below:
\[
  p = \left(\cos \pi x + \cos \pi y\right) \sin \pi t,
\]
\[
  E = %
  \begin{bNiceMatrix}
    \sin \pi (\sqrt{2} t - x - y) - \sin \pi x \cos \pi t \\
    -\sin \pi (\sqrt{2} t - x - y) - \sin \pi y \cos \pi t \\
    0
  \end{bNiceMatrix}, %
  \quad H = %
  \begin{bNiceMatrix}
    0 \\ 
    0 \\
    -\sqrt{2} \sin \pi (\sqrt{2} t - x - y)
   \end{bNiceMatrix},
\]
\[\epsilon = 1, \quad \mu = 1, \quad T_{\min} = 0 , \quad T_{\max} = 2.\]
Similar to Example 2, the boundary conditions for $p$ and $E$ are nonhomogeneous everywhere on the boundary while for $H$ it is nonhomogeneous in some parts of the boundary. The computational results are shown in Figures~\labelcref{fig:example4_cn,fig:example4_lf,fig:3d_energies}.

\medskip \noindent \textbf{Example 5}: Our next example is on the unit square in $\R^2$ with analytical solutions, material parameters, initial and final times as below:
\[
  p = \left(\cos x + \cos y\right) \sin t,
\]
\[
  E = %
  \begin{bNiceMatrix}
    \sin \pi (\sqrt{2} t - x - y) - \sin x \cos t \\
    -\sin \pi (\sqrt{2} t - x - y) - \sin y \cos t
  \end{bNiceMatrix}, %
  \quad H = -\sqrt{2} \sin \pi (\sqrt{2} t - x - y),
\]
\[\epsilon = 1, \quad \mu = 1, \quad T_{\min} = 0 , \quad T_{\max} = 3.14.\]

\medskip \noindent \textbf{Example 6}: Our final example is on the unit cube in $\R^3$ with analytical solutions, material parameters, initial and final times as below:
\[
  p = 0, \quad
  E = %
  \begin{bNiceMatrix}
   \sin y \sin z \cos t \\
   \sin x \sin z \cos t \\
   \sin x \sin y \cos t 
  \end{bNiceMatrix}, %
  \quad H = %
   \begin{bNiceMatrix}
      \sin x (\cos z - \cos y) \sin t \\
      \sin y (\cos x - \cos z) \sin t \\
      \sin z (\cos y - \cos x) \sin t
   \end{bNiceMatrix},
\]
\[\epsilon = 2, \quad \mu = 1, \quad T_{\min} = 0 , \quad T_{\max} = 3.\]

\subsection{Discussion}

For all problems in our examples, the right hand side functions for the stated analytical solutions are all zero in Equation~\eqref{eqn:maxwells_eqns}. We can readily obtain the analytical energies ($\norm{p}^2_{\epsilon^{-1}} + \norm{E}^2_\epsilon + \norm{H}^2_\mu$) for each of these problems by integration as:
\begin{center}
\begin{tabular}{|c|c|c|c|c|c|c|}
\hline
& Example 1 & Example 2 & Example 3 & Example 4 & Example 5 & Example 6 \\
\hline
\text{Energy at time } $t$ & $1$ & $\dfrac{3}{2}$ & $3$ & $3$ & $3 + \mathcal{E}_1(t)$ & $\dfrac{3}{2}  + \mathcal{E}_2(t)$ \\
\hline
\end{tabular}
  $\mathcal{E}_1(t) = -\dfrac{1}{2} \sin 2 + \left( 2 \sin 1 + \sin 2 \right) \sin^2 t, \quad  \mathcal{E}_2(t)  = \dfrac{3}{8} \sin^2 2 \cos 2t + \dfrac{3}{2} \sin 2 (\sin^2 1 \sin^2 t - \cos^2 t)$.
\end{center}
Consequently, the energies of the discrete solutions computed using the Crank-Nicholson and implicit leapfrog schemes are indeed conserved for the various problems in Examples 1 to 4 and as can be seen in Figures~\labelcref{fig:2d_energies,fig:3d_energies}. For a numerical check on the consistency of our computations we verify that for each of the four examples, the backward Euler time discretization is dissipative.

With respect to the choice of higher order finite elements, indeed choosing quadratic Whitney finite elements does lead to a better approximation for the various solutions $p_h^n$, $E_h^n$ and $H_h^n$ at the appropriate time indices and this is also evident from the solution plots which we have discussed. However, this effect is more pronounced in $\R^3$, and therefore we provide the linear and quadratic spatial discretization solutions in Figures~\labelcref{fig:example2_cn,fig:example2_lf,fig:example4_cn,fig:example4_lf}.

\subsection{Reproducibility} All our code for the various examples in \Cref{sec:numerics} is available via GitHub \cite{ArKa2024}.

\vspace*{\fill}
\begin{figure}[h]
  \centering
  \includegraphics[scale=1.2]{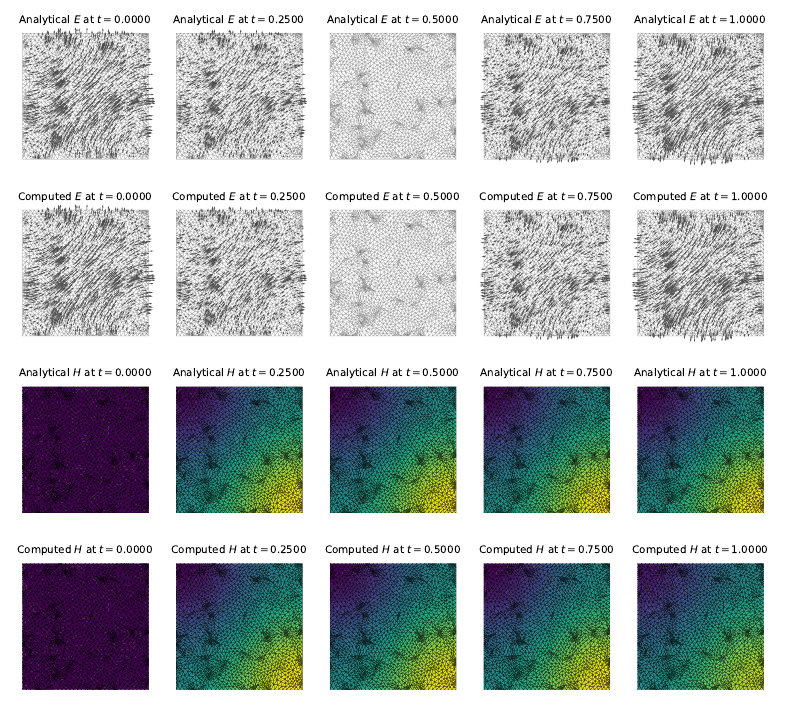}
  \caption{\textbf{Quadratic finite elements, Crank-Nicholson}: Plots of solutions at different time steps for the problem described in \textbf{Example 1} of Section~\ref{sec:numerics} using the Crank-Nicholson time integration and quadratic Whitney forms as basis for the finite dimensional finite element spaces. The solutions for $p$ are not shown due to them being identically equal to $0$. The computed solutions for $E$ and $H$ visually match with the analytical solutions near identically.}
  \label{fig:example1_cn}
\end{figure}
\vspace{\fill}
\clearpage

\vspace*{\fill}
\begin{figure}[h]
  \centering
  \includegraphics[scale=1.2]{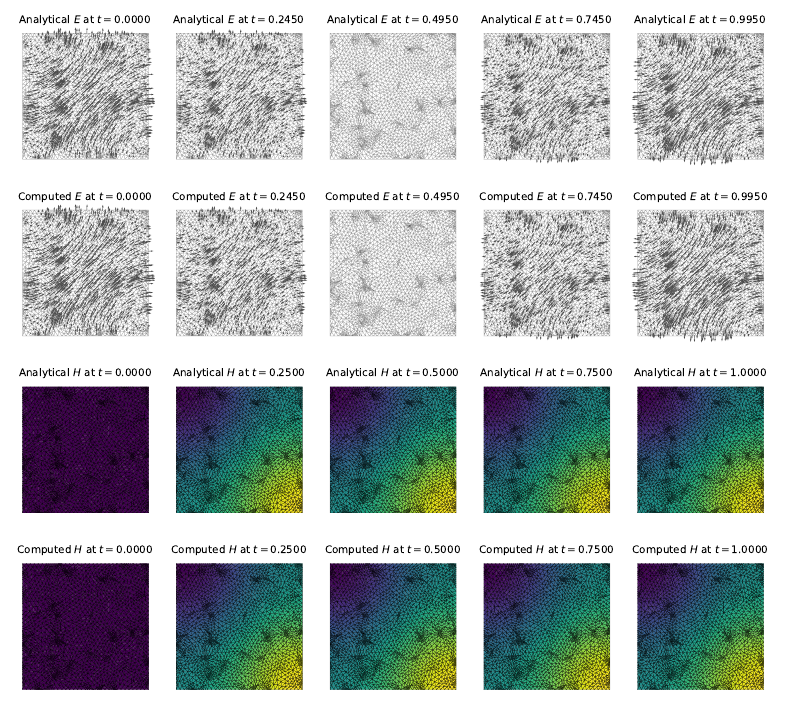}  
  \caption{\textbf{Quadratic finite elements, implicit leapfrog}: Plots of solutions at different time steps for the problem described in \textbf{Example 1} of Section~\ref{sec:numerics} using the implicit leapfrog scheme and quadratic Whitney forms as basis for the finite element spaces. The computed solutions for $E$ and $H$ match with the analytical solutions near identically.}
  \label{fig:example1_lf}
\end{figure}
\vspace{\fill}
\clearpage

\vspace{\fill}
\begin{figure}[h]
  \centering
  \includegraphics[scale=1.2]{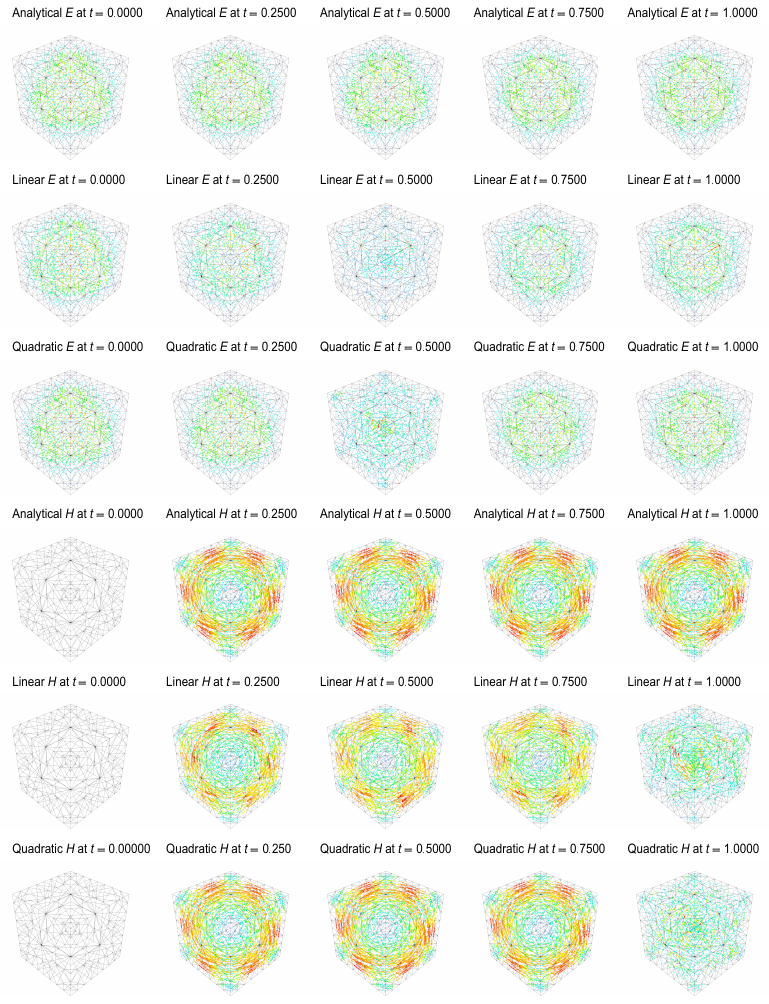}
  \caption{\textbf{Linear and Quadratic finite elements, Crank-Nicholson}: Plots of solutions at different time steps for the problem described in \textbf{Example 2} of Section~\ref{sec:numerics} using the Crank-Nicholson time integration, and with both linear and quadratic Whitney forms as basis for the finite dimensional finite element spaces. The solutions computed for $E$ and $H$ using the quadratic elements match with the analytical solutions near identically while those computed with the linear elements have a somewhat pronounced approximation error.}
  \label{fig:example2_cn}
\end{figure}
\vspace{\fill}
\clearpage

\vspace{\fill}
\begin{figure}[h]
  \centering
  \includegraphics[scale=1.2]{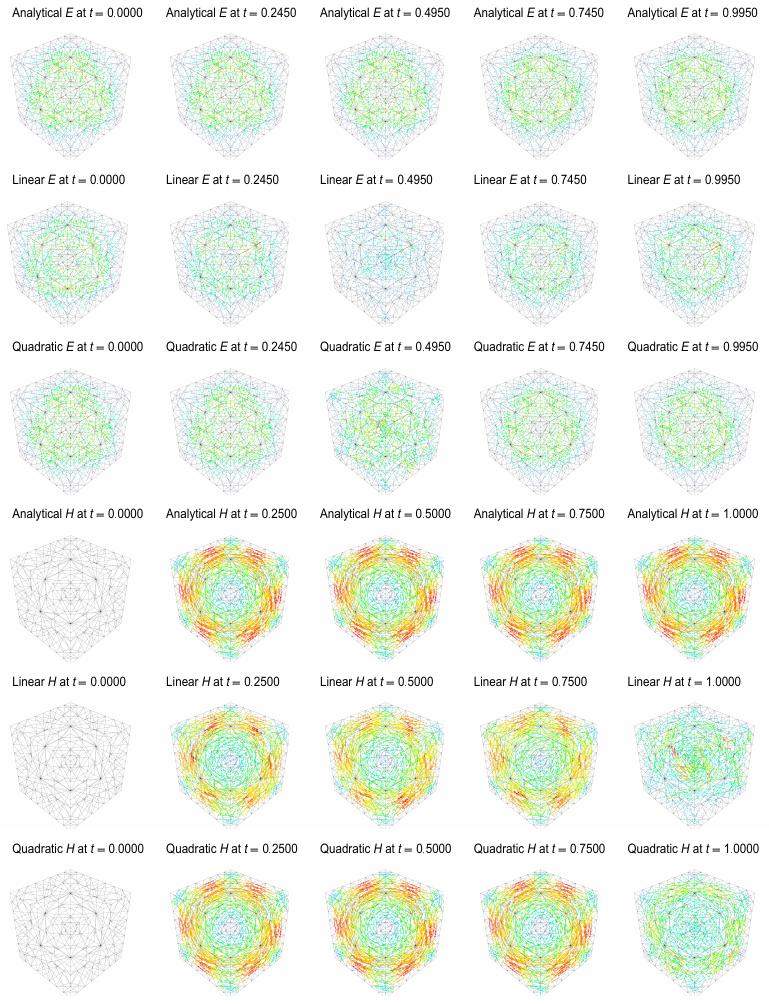}
  \caption{\textbf{Linear and Quadratic finite elements, implicit leapfrog}: Plots of solutions at different time steps for the problem described in \textbf{Example 2} of Section~\ref{sec:numerics} using the implicit leapfrog time integration, and with both linear and quadratic Whitney forms as basis for the finite element spaces. The solutions computed for $E$ and $H$ using the quadratic elements match with the analytical solutions near identically while those computed with the linear elements have a perceptible approximation error.}
  \label{fig:example2_lf}
\end{figure}
\vspace{\fill}
\clearpage

\vspace{\fill}
\begin{figure}[h]
  \centering
  \includegraphics[scale=1.2]{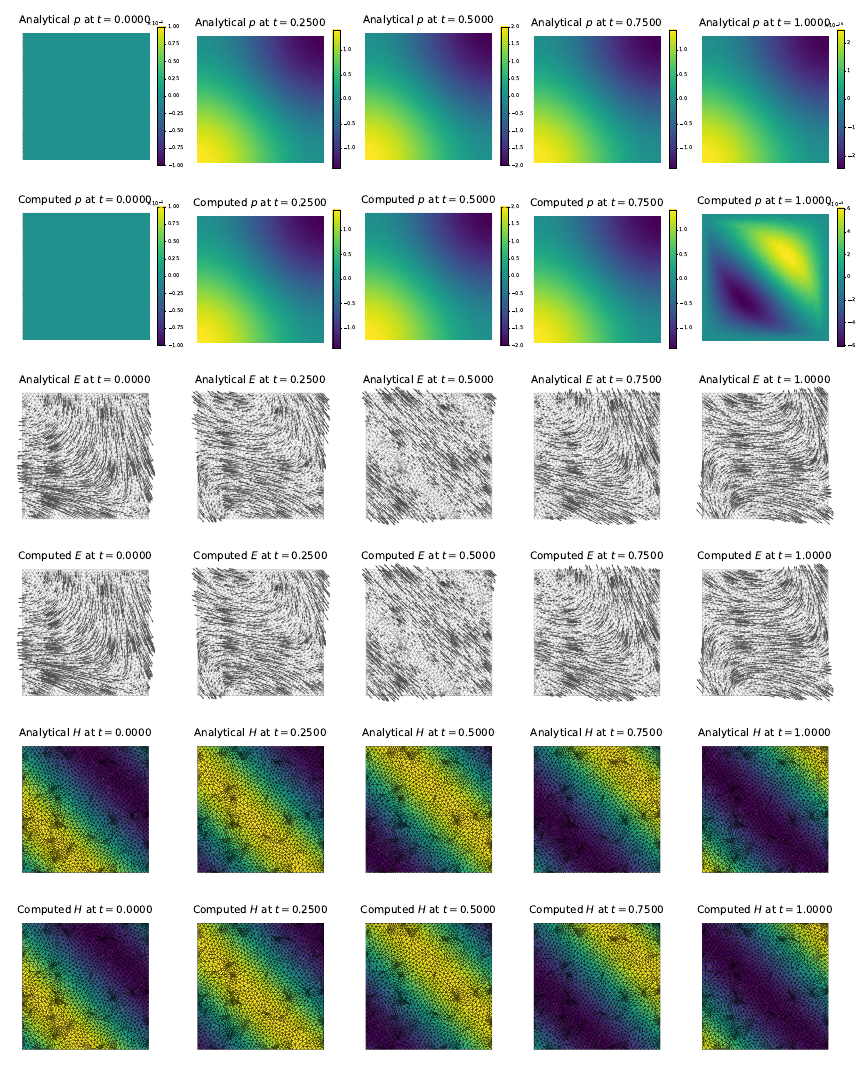}
  \caption{\textbf{Quadratic finite elements, Crank-Nicholson}: Plots of solutions at different time steps for the problem described in \textbf{Example 3} of Section~\ref{sec:numerics} using the Crank-Nicholson time integration and quadratic Whitney forms as basis for the finite element spaces. The computed solutions for $p$, $E$ and $H$ match with the analytical solutions near identically.}
  \label{fig:example3_cn}
\end{figure}
\vspace{\fill}
\clearpage

\vspace{\fill}
\begin{figure}[h]
  \centering
  \includegraphics[scale=1.2]{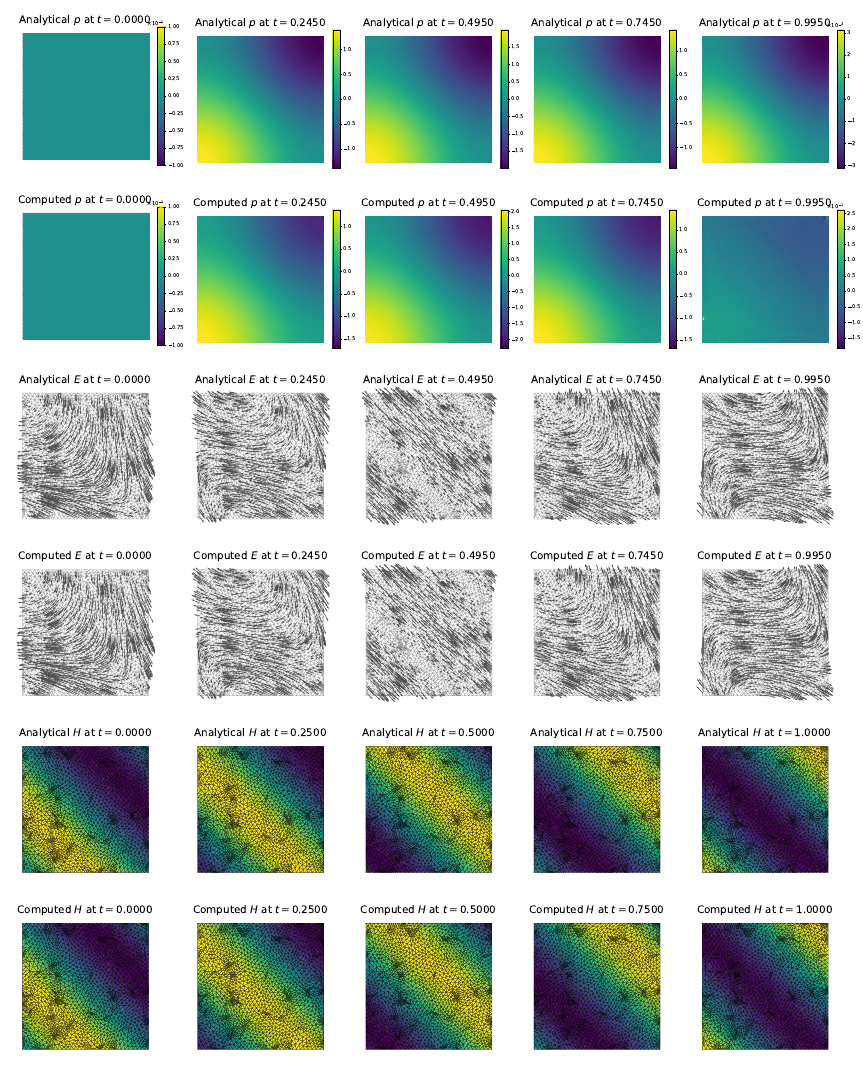}
  \caption{\textbf{Quadratic finite elements, implicit leapfrog}: Plots of solutions at different time steps for the problem described in \textbf{Example 3} of Section~\ref{sec:numerics} using the implicit leapfrog time integration and quadratic Whitney forms. The computed solutions for $p$, $E$ and $H$ match with the analytical solutions near identically.}
  \label{fig:example3_lf}
\end{figure}
\vspace{\fill}
\clearpage

\vspace{\fill}
\begin{figure}[h]
  \centering
  \includegraphics[scale=1.2]{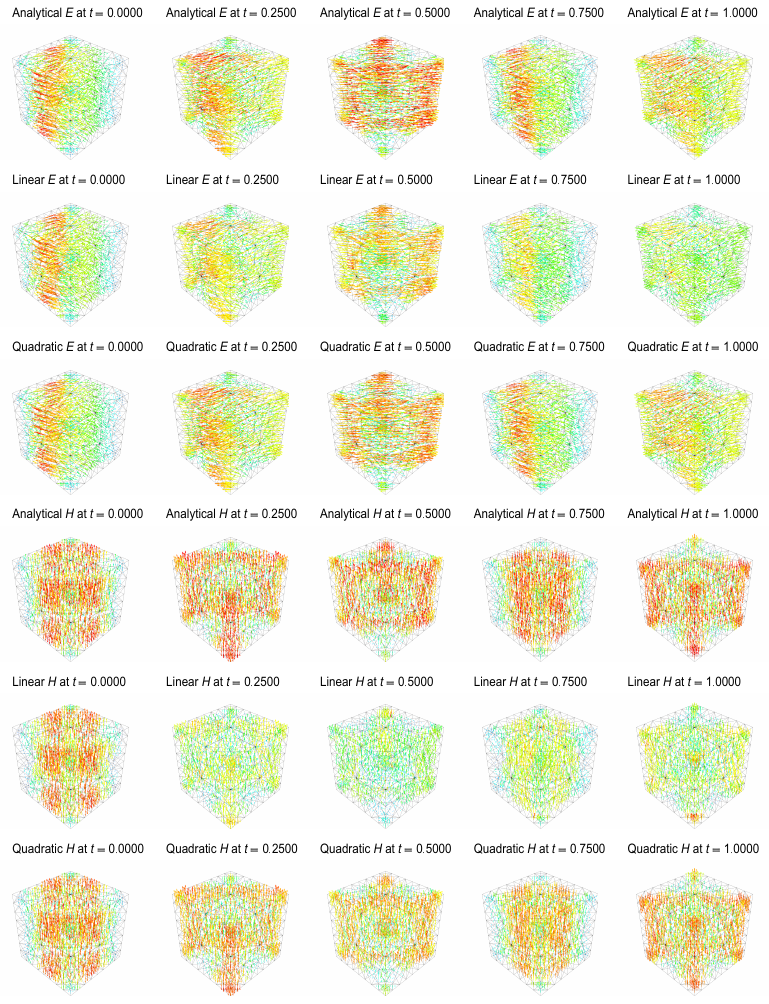}
  \caption{\textbf{Linear and Quadratic finite elements, Crank-Nicholson}: Plots of solutions at different time steps for the problem described in \textbf{Example 4} of Section~\ref{sec:numerics} using the Crank-Nicholson time integration, and with linear and quadratic Whitney forms. The solutions computed for $E$ and $H$ using the quadratic elements match with the analytical solutions near identically while those computed with the linear elements have a perceptible approximation error.}
  \label{fig:example4_cn}
\end{figure}
\vspace{\fill}
\clearpage

\vspace{\fill}
\begin{figure}[h]
  \centering
  \includegraphics[scale=1.2]{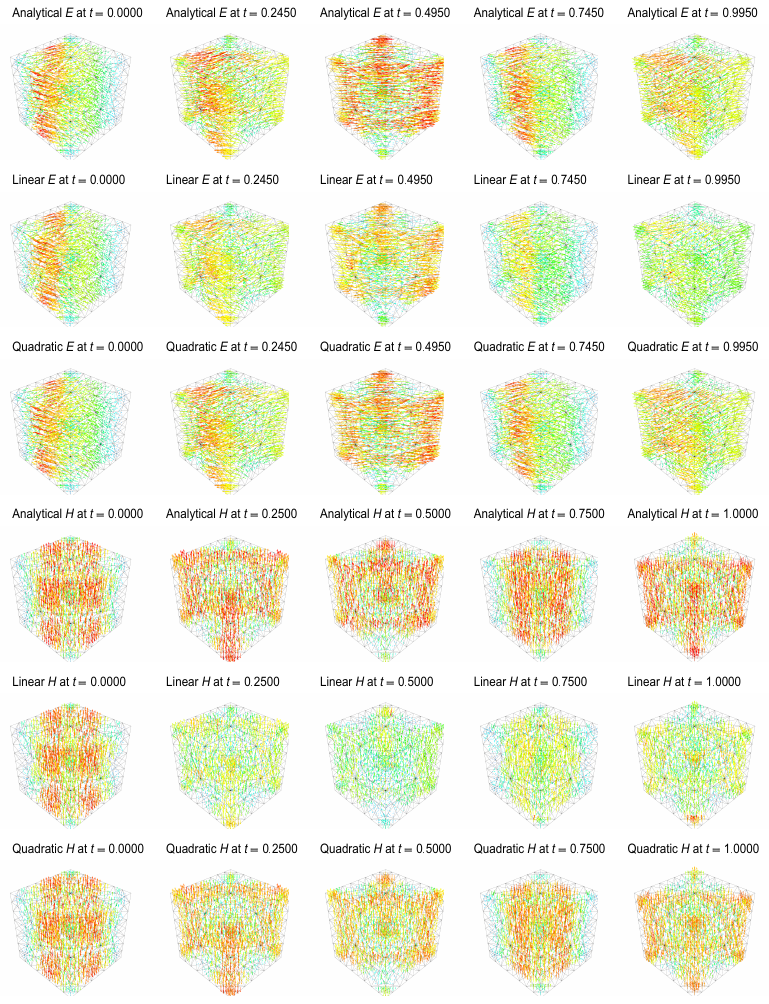}
  \caption{\textbf{Linear and Quadratic finite elements, implicit leapfrog}: Plots of solutions at different time steps for the problem described in \textbf{Example 4} of Section~\ref{sec:numerics} using the implicit leapfrog time integration with linear and quadratic Whitney basis. The solutions computed for $E$ and $H$ using the quadratic elements match with the analytical solutions near identically while some approximation errors are evident for the computations with the linear basis.}
  \label{fig:example4_lf}
\end{figure}
\vspace{\fill}
\clearpage

\vspace*{\fill}
\begin{figure}[h]
  \centering
  \begin{subfigure}[t]{0.48\linewidth}
    \centering \includegraphics[scale=0.5]{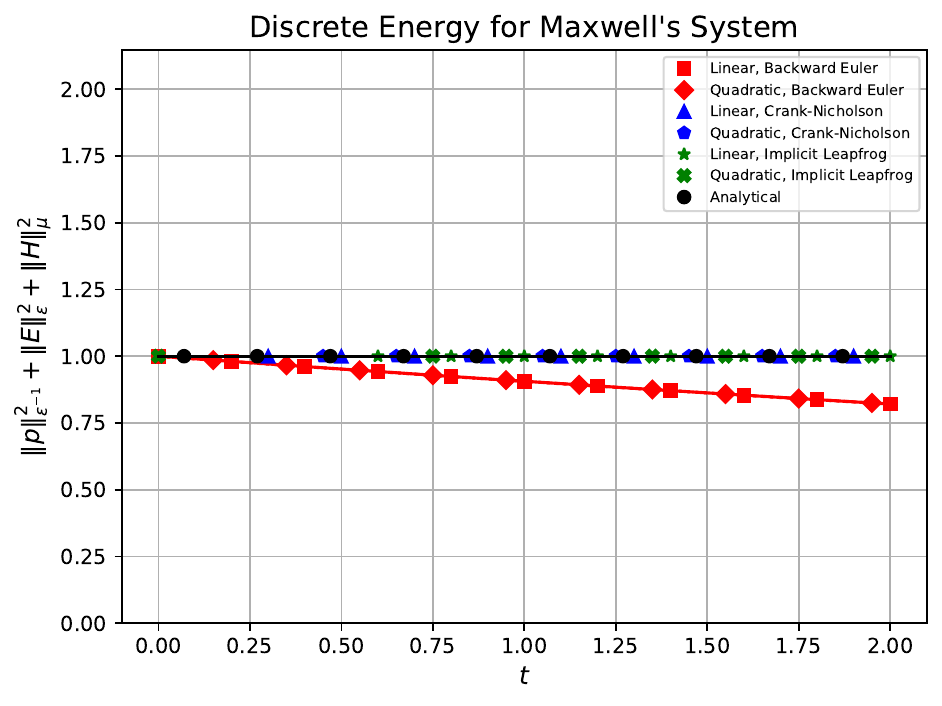}
  \end{subfigure}
  \hspace{0.025\linewidth}
  \begin{subfigure}[t]{0.48\linewidth}
    \centering \includegraphics[scale=0.5]{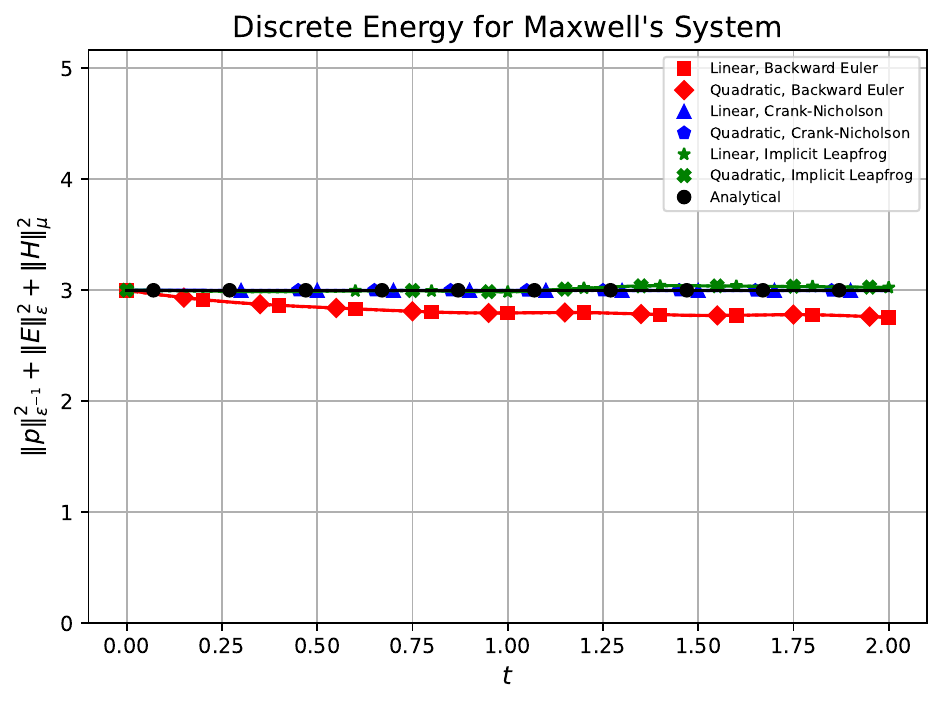}
  \end{subfigure}
  \caption{\textbf{Energy Conservation for Examples 1} (\textit{left}) \textbf{and 3} (\textit{right}): Energy is conserved for both Crank-Nicholson and implicit leapfrog time integration schemes with both linear and quadratic finite elements for the experiments whose solutions are shown in Figures~\labelcref{fig:example1_cn,fig:example1_lf,fig:example3_cn,fig:example3_lf}. For contrast, we also show the energy decay using the stable but dissipative backward Euler time integration scheme for both these problems.}
  \label{fig:2d_energies}
\end{figure}

\vspace*{\fill}
\begin{figure}[h]
  \centering
  \begin{subfigure}[t]{0.48\linewidth}
    \centering \includegraphics[scale=0.5]{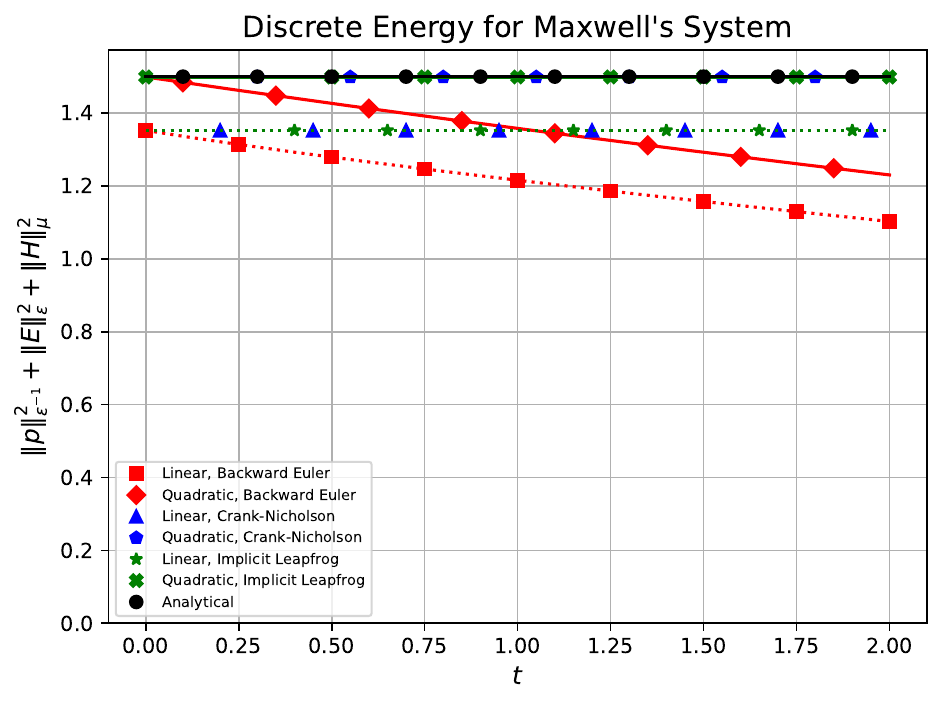}
  \end{subfigure}
  \hspace{0.02\linewidth}
  \begin{subfigure}[t]{0.48\linewidth}
    \centering \includegraphics[scale=0.5]{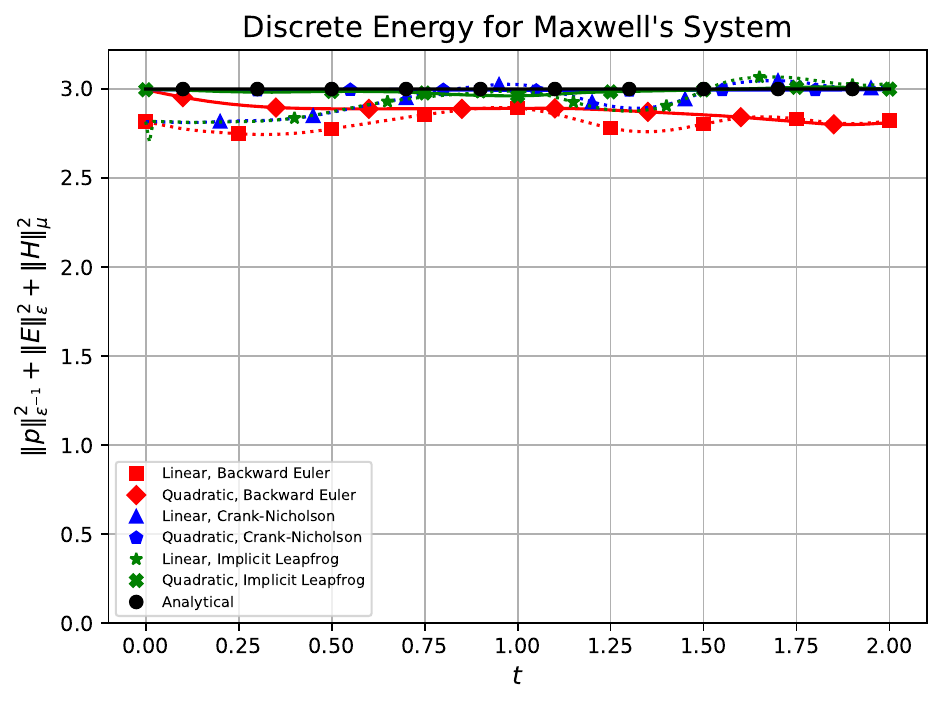}
  \end{subfigure}
  \caption{\textbf{Energy Conservation for Examples 2} (\textit{left}) \textbf{and 4} (\textit{right}): The discrete energy conservation for both Crank-Nicholson and implicit leapfrog time integration schemes in $\R^3$ are numerically more robust with quadratic finite elements for the experiments whose solutions are shown in Figures~\labelcref{fig:example2_cn,fig:example2_lf,fig:example4_cn,fig:example4_lf}. For contrast, we again show the energy decay for backward Euler time integration scheme with both linear and quadratic finite elements. But, we also empirically validate that computational errors in our second order accurate in time implicit methods are useful vis-a-vis energy conservation only by also using second order accurate in space finite elements.}
  \label{fig:3d_energies}
\end{figure}
\vspace{\fill}
\clearpage

\vspace{\fill}
\begin{figure}[h]
  \centering
  \includegraphics[scale=1.2]{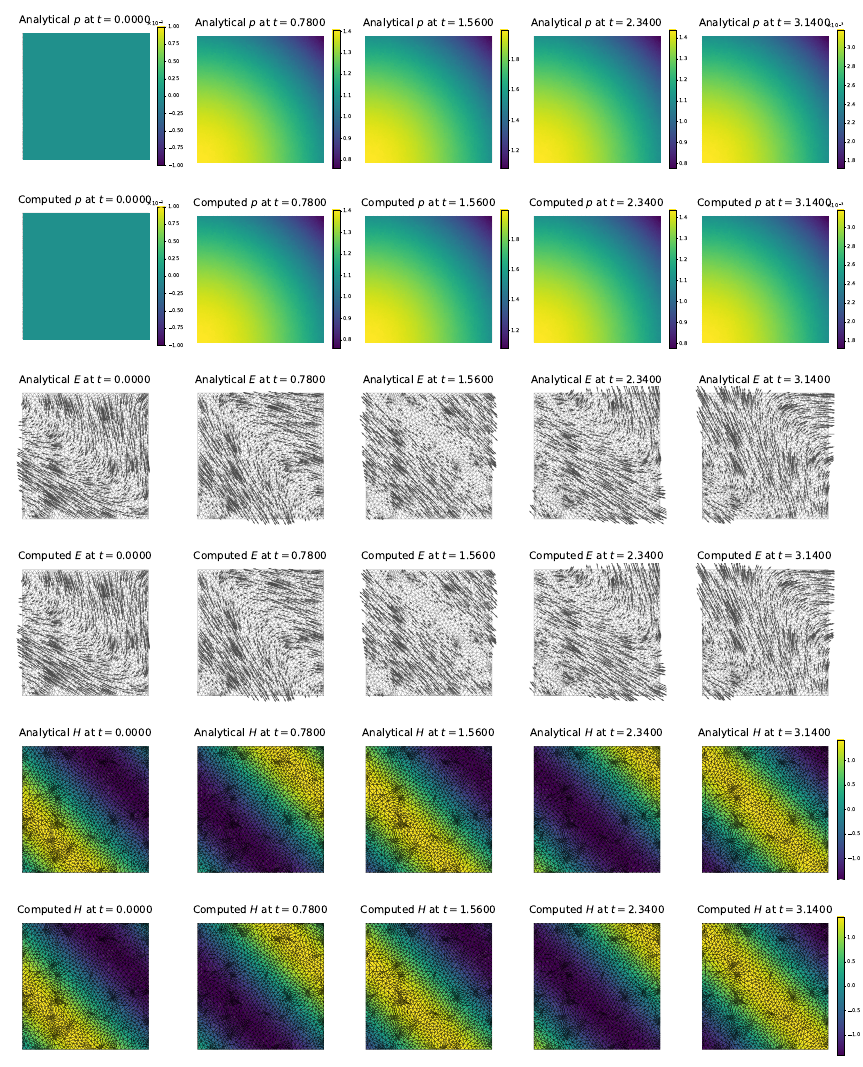}
  \caption{\textbf{Quadratic finite elements, Crank-Nicholson}: Plots of solutions at different time steps for the problem described in \textbf{Example 5} of Section~\ref{sec:numerics} using the Crank-Nicholson time integration and quadratic Whitney forms as basis for the finite element spaces. The computed solutions for $p$, $E$ and $H$ match with the analytical solutions near identically.}
  \label{fig:example5_cn}
\end{figure}
\vspace{\fill}
\clearpage

\vspace{\fill}
\begin{figure}[h]
  \centering
  \includegraphics[scale=1.2]{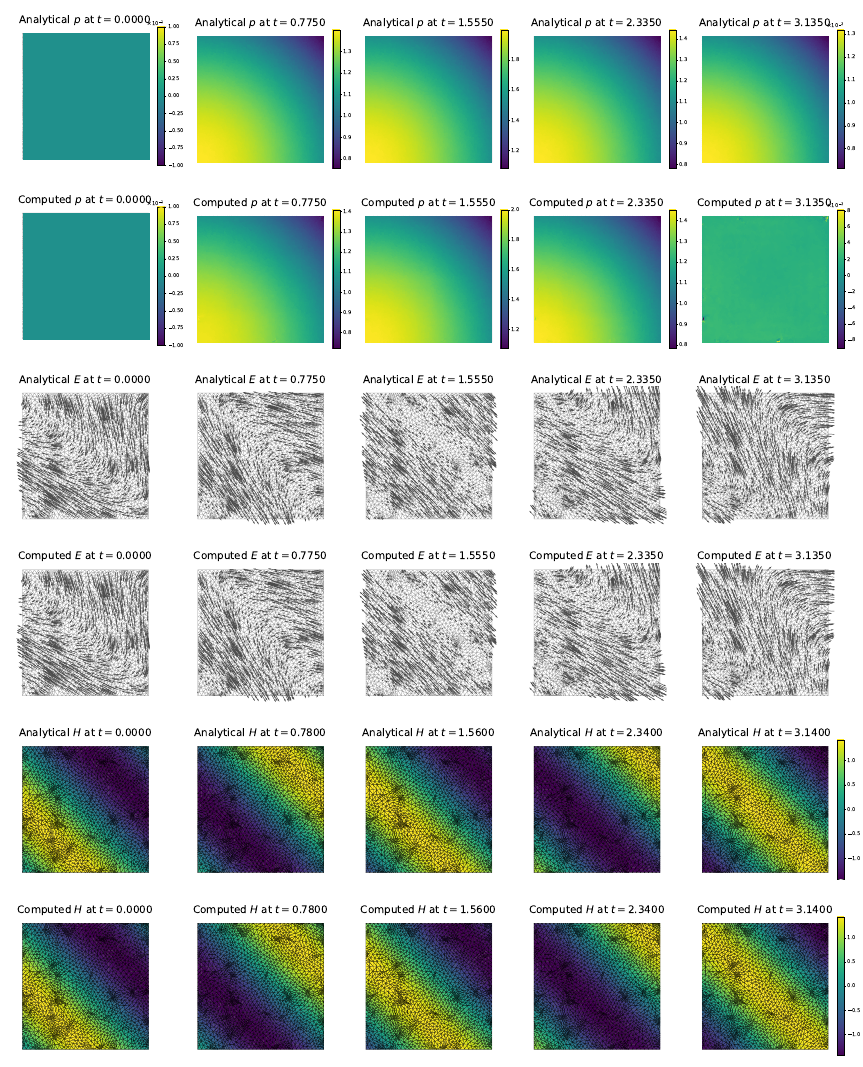}
  \caption{\textbf{Quadratic finite elements, implicit leapfrog}: Plots of solutions at different time steps for the problem described in \textbf{Example 5} of Section~\ref{sec:numerics} using the implicit leapfrog time integration and quadratic Whitney forms. The computed solutions for $p$, $E$ and $H$ match with the analytical solutions near identically.}
  \label{fig:example5_lf}
\end{figure}
\vspace{\fill}
\clearpage

\vspace{\fill}
\begin{figure}[h]
  \centering
  \includegraphics[scale=1.2]{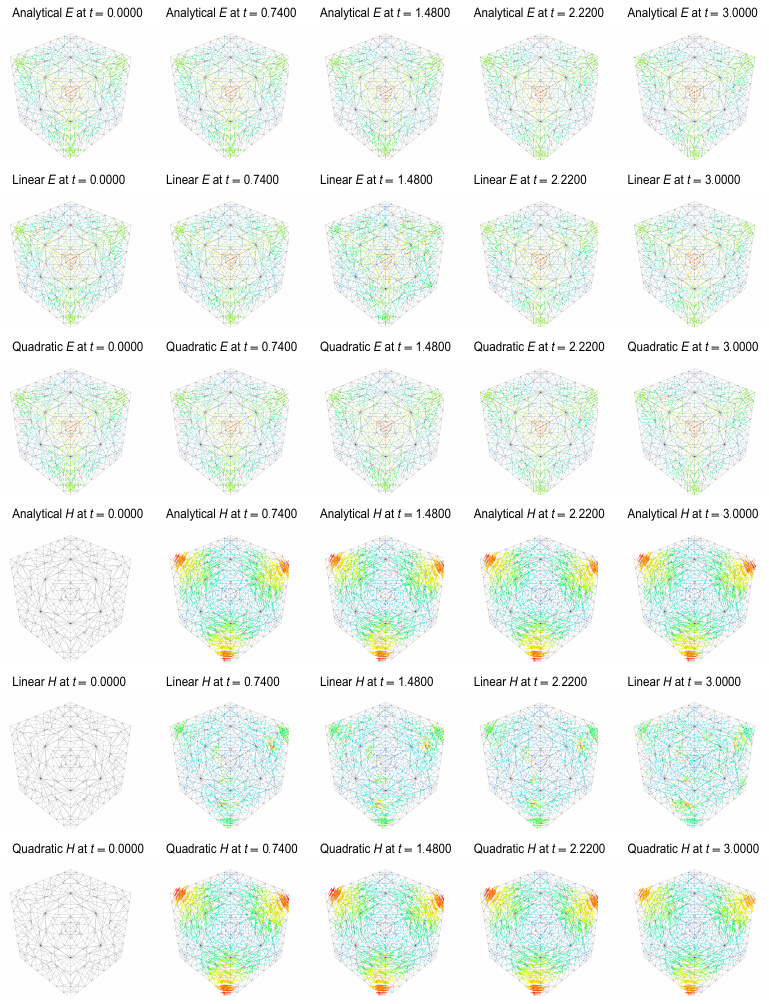}
  \caption{\textbf{Quadratic finite elements, Crank-Nicholson}: Plots of solutions at different time steps for the problem described in \textbf{Example 6} of Section~\ref{sec:numerics} using the Crank-Nicholson time integration and quadratic Whitney forms as basis for the finite element spaces. The computed solutions for $p$, $E$ and $H$ match with the analytical solutions near identically.}
  \label{fig:example6_cn}
\end{figure}
\vspace{\fill}
\clearpage

\vspace{\fill}
\begin{figure}[h]
  \centering
  \includegraphics[scale=1.2]{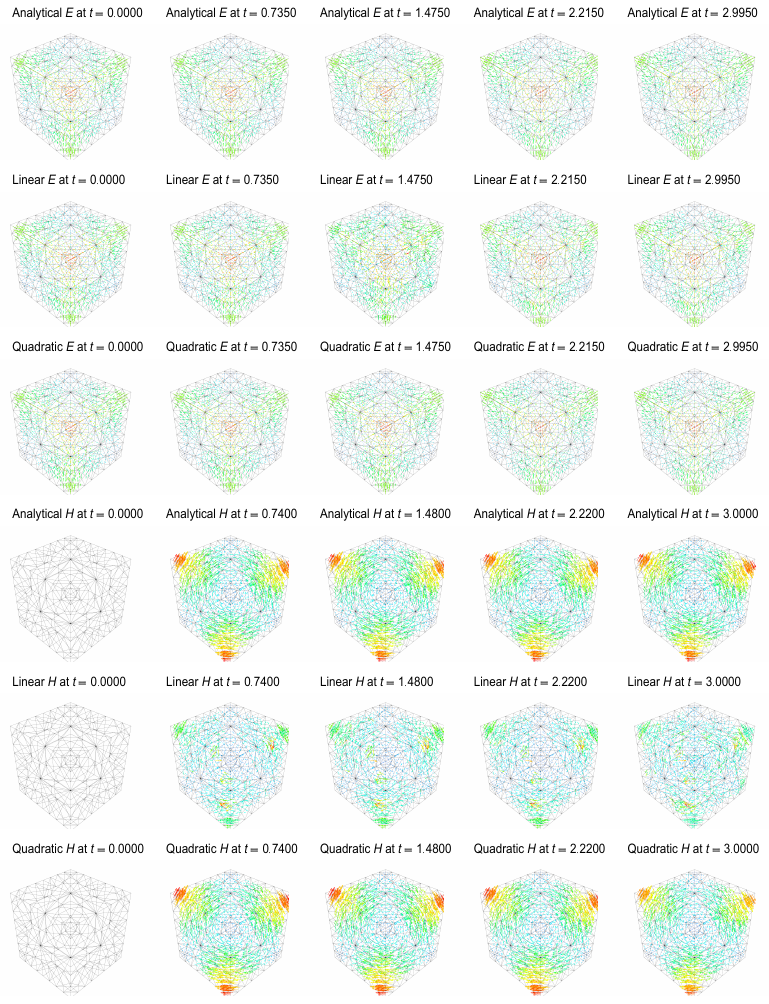}
  \caption{\textbf{Quadratic finite elements, implicit leapfrog}: Plots of solutions at different time steps for the problem described in \textbf{Example 6} of Section~\ref{sec:numerics} using the implicit leapfrog time integration and quadratic Whitney forms. The computed solutions for $p$, $E$ and $H$ match with the analytical solutions near identically.}
  \label{fig:example6_lf}
\end{figure}
\vspace{\fill}
\clearpage

\begin{figure}[h]
  \centering
  \begin{subfigure}[t]{0.48\linewidth}
    \centering \includegraphics[scale=0.5]{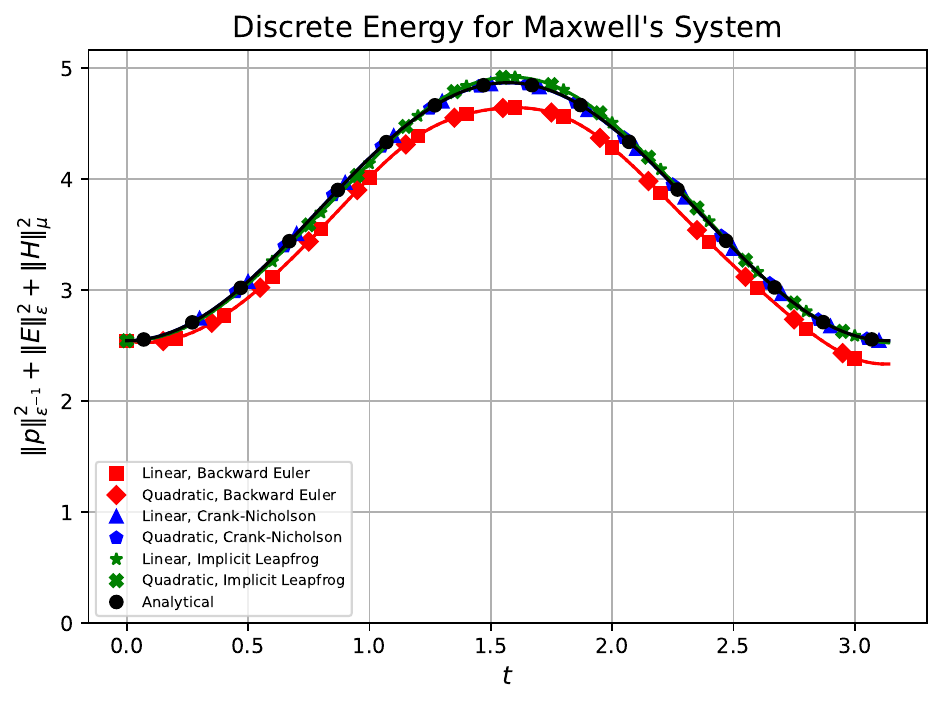}
  \end{subfigure}
  \hspace{0.025\linewidth}
  \begin{subfigure}[t]{0.48\linewidth}
    \centering \includegraphics[scale=0.5]{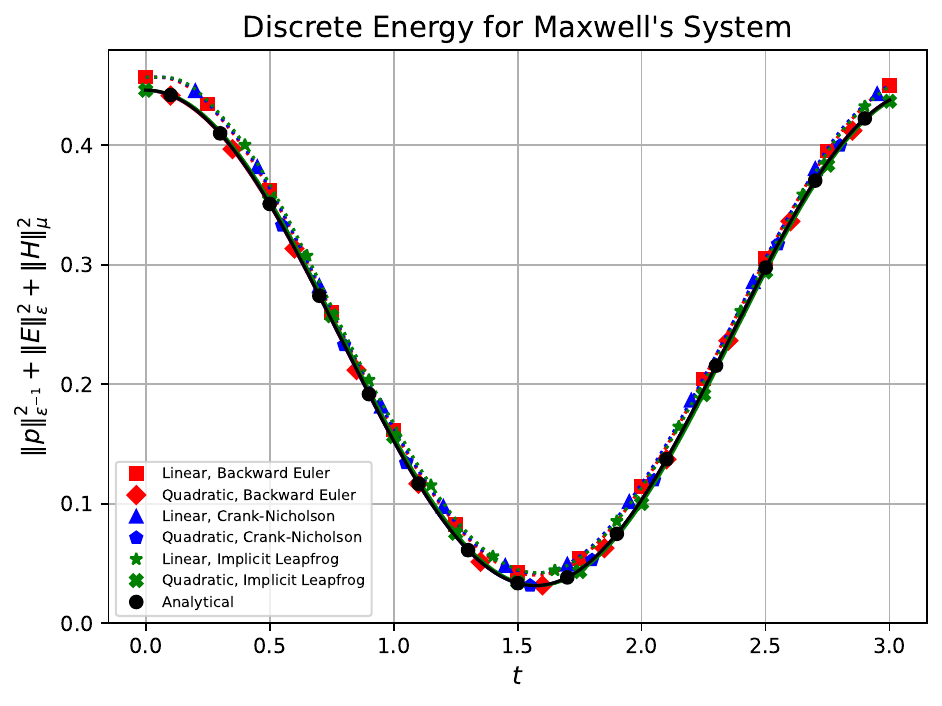}
  \end{subfigure}
  \caption{\textbf{Energy Tracking for Examples 5} (\textit{left}) \textbf{and 6} (\textit{right}): The discrete energy for these problems is time varying but bounded. Both Crank-Nicholson and implicit leapfrog time integration schemes in $\R^2$ (\textit{left}) and $\R^3$ (\textit{right}) are able to discretely exactly track this varying energy but the backward Euler scheme for Example 5 is dissipative.}
  \label{fig:varying_energies}
\end{figure}

\printbibliography
\end{document}